\documentclass[11pt]{article}
\usepackage{graphicx,psfrag}
\usepackage{qtree}
\usepackage{amssymb,amsfonts,amsthm}
\usepackage{amsmath,amsthm,amssymb}
\usepackage{amsfonts}
\usepackage{relsize}
\usepackage[T1]{fontenc}
\usepackage[utf8]{inputenc}
\usepackage{graphicx}
\usepackage{color}
\usepackage{subcaption,mathtools}
\usepackage{epstopdf}
\usepackage[all]{xy}

\newcommand{\PP}{\mathcal {P}}

\newcommand{\bra}{\mathrm{bra}}
\newcommand{\sem}{\mathrm{sem}}

\newcommand{\topp}{{\bf top}}

\newcommand{\bott}{{\bf bot}}

\newcommand{\DG}{\mathrm{DG}}

\newcommand{\Supp}{\mathrm{Supp}}

\newcommand{\Cl}{\mathrm{Cl}}

\newcommand{\T}{\mathrm{T}}
\newcommand{\PL}{\mathrm{PL}}

\newcommand{\dd}{\mathcal{D}}

\newcommand{\la}{\langle}
\newcommand{\ra}{\rangle}

\renewcommand{\P}{\mathcal{P}}
\newcommand{\F}{\mathcal{F}}

\newtheorem{theorem}{Theorem}[section]
\newtheorem{lemma}[theorem]{Lemma}

\theoremstyle{definition}

\newcommand{\kk}{\mathcal{K}}

\newcommand{\be}{\begin{equation}}
\newcommand{\ee}{\end{equation}}

\newcommand {\N}{\mathbb{N}} 

\renewcommand {\min}{\mathrm{min}}
\newcommand {\lab}{\mathrm{lab}}

\newcommand {\iv}{^{-1}}

\numberwithin{equation}{section}

\newcounter{AbcT}

\newcommand{\La}{\mathcal{L}}

\newcommand{\nc}{\newcommand}
\nc{\meet}{\wedge}
\nc{\op}{\operatorname}\nc{\FP}{\op{FP}}\nc{\FS}{\op{FS}}\nc{\FPhat}{\widehat{\op{FP}}}

\newtheorem {Theorem}    {Theorem}[section]

\newtheorem {Problem}    [Theorem]{Problem}

\newtheorem {Lemma}      [Theorem]    {Lemma}
\newtheorem {Corollary}   [Theorem] {Corollary}
\newtheorem {Proposition}[Theorem]    {Proposition}

\newtheorem {Example}      [Theorem]    {Example}
\newtheorem {Conjecture}[Theorem]    {Conjecture}

\theoremstyle{remark}
\newtheorem {Remark}		 [Theorem]    {\bf{Remark}}

\newtheorem {Definition} [Theorem]    {\bf{Definition}}

\newcommand{\ab}{\mathrm{ab}}
\newcommand{\DPiec}{\mathrm{DPiec}}

\begin{document}

\title{The generation problem in Thompson group $F$}
\author{Gili Golan Polak\thanks{Most of the research was conducted in Vanderbilt University. It was partially supported by a Fulbright grant and a post-doctoral scholarship from Bar-Ilan University.}}

\maketitle

\abstract{
We show that the generation problem in Thompson's group $F$ is decidable, i.e., there is an algorithm which decides if a finite set of elements of $F$ generates the whole $F$. 
The algorithm makes use of the Stallings $2$-core of subgroups of $F$, which can be defined in an analogous way to the Stallings core of subgroups of a finitely generated free group. Further study of the Stallings $2$-core of subgroups of $F$ provides a solution to another algorithmic problem in $F$. Namely, given a finitely generated subgroup $H$ of $F$, it is decidable if $H$ acts transitively on the set of finite dyadic fractions $\mathcal D$. 
Other applications of the study include the construction of new maximal subgroups of $F$ of infinite index, among which, a maximal subgroup of infinite index which acts transitively on the set $\mathcal D$ and the construction of an elementary amenable subgroup of $F$ which is maximal in a normal subgroup of $F$. 
}

\renewcommand{\thefootnote}{\fnsymbol{footnote}} 
\footnotetext{\emph{2010 Mathematics Subject Classification.} Primary 20F10, 20F65 Secondary 20E28
}    
\footnotetext{\emph{Key Words:} Thompson's group $F$, decision problems, diagram groups, directed $2$-complexes, the Stallings $2$-core, closed subgroups, maximal subgroups, homeomorphisms of the interval. 
}     
\renewcommand{\thefootnote}{\arabic{footnote}}

\tableofcontents

\section{Introduction}

Recall that R. Thompson's group $F$ is the group of all piecewise-linear\footnote{Throughout the paper we wil use ``linear'' for ``affine'', as is common in the literature.} homeomorphisms of the interval $[0,1]$ with finitely many breakpoints, where all breakpoints are finite dyadic fractions (i.e., elements of the set $\mathbb Z[\frac{1}{2}]\cap(0,1)$) and all slopes are integer powers of $2$.
The group $F$ has a presentation with two generators and two defining relations \cite{CFP, Sa} (see below).

Decision problems in $F$ have been extensively studied. It is well known that the word problem in $F$ is decidable in linear time 
\cite{CFP,SU,GuSa97} and that the conjugacy problem is decidable in linear time \cite{GuSa97,BM,BHMM}. The simultaneous conjugacy problem \cite{KM}  and twisted conjugacy problem \cite{BMV} have also been proven to be decidable. The article \cite{BBH} gives an algorithm for deciding if a finitely generated subgroup $H$ of $F$ is solvable. On the other hand, it is proved in \cite{BMV} that there are orbit undecidable subgroups of $\mathrm{Aut}(F)$ and hence, there are extensions of Thompson’s group $F$ by finitely generated free groups with unsolvable conjugacy problem.

In this paper we consider the generation problem in Thompson's group $F$. Namely, the problem of deciding for a given finite subset $X$ of $F$ whether $\la X\ra=F$. Note that  solvability of the membership problem for subgroups in $F$ is a very interesting open problem (it is mentioned in \cite{GuSa99}) and that the generation problem is an important special case of the membership problem (we need to check if the generators of $F$ are in the subgroup). The generation problem is known to be undecidable for $F_2\times F_2$ (see, for example, \cite{LS}). Using Rips' construction \cite{R}, one can find a hyperbolic group $G$ which projects onto $F_2 \times F_2$ and has undecidable generation problem. Moreover, using Wise's version of Rips'  construction \cite{W}, one can ensure that $G$ is linear over $\mathbb{Z}$.


We prove that the generation problem in Thompson's group $F$ is decidable.
 The algorithm solving the generation problem makes use of the definition of $F$ as a diagram group \cite{GuSa97} and the construction of the Stallings $2$-core of subgroups of diagram groups. The construction is due to Guba and Sapir from 1999, but appeared in print first in \cite{GS1}. 

Every element of $F$ can be viewed as a diagram $\Delta$. A diagram $\Delta$ in $F$ is a directed labeled planar graph tessellated by cells, defined up to an isotopy of the plane. It has one top edge and one bottom edge and the whole diagram is situated between them. Every cell in 
$\Delta$ is either a positive or a negative cell. Positive cells have one top edge and two bottom edges. Negative cells are the reflection of positive cells about a horizontal line (See Figure \ref{f:xx}). In particular, a diagram $\Delta$ can be naturally viewed as a directed $2$-complex; i.e., a $2$-complex which consists of vertices, directed edges and $2$-cells bounded by two directed paths with the same endpoints.

The Stallings $2$-core $\La(H)$ of a subgroup $H\le F$ can be viewed as a $2$-dimensional analogue of the Stallings core of subgroups of free groups (see \cite{Sta}). The Stallings $2$-core $\La(H)$ is a directed $2$-complex associated with $H$ which has a distinguished input/output edge. We say that $\La(H)$ is a \emph{$2$-automaton}. The core $\La(H)$ \emph{accepts} a diagram $\Delta$ in $F$ if there is a morphism of directed $2$-complexes from $\Delta$ to $\La(H)$ which maps the top and bottom edges of $\Delta$ to the input/output edge of $\La(H)$.

By construction \cite{GS1}, $\La(H)$ accepts all diagrams in $H$, but unlike in the case of free groups, the core $\La(H)$ can accept diagrams not in $H$. We define the \emph{closure} of $H$ to be the subgroup of $F$ of all diagrams accepted by the core $\La(H)$. The closure operation satisfies the usual properties of closure. Namely, $H\le \Cl(H)$, $\Cl(\Cl(H))=\Cl(H)$ and if $H_1\le H_2$ then $\Cl(H_1)\le \Cl(H_2)$. We say that $H$ is \emph{closed} if $H=\Cl(H)$. The closure of $H$ is a diagram group over the directed $2$-complex $\La(H)$. Thus, if $H$ is finitely generated, the membership problem in $\Cl(H)$ is decidable \cite{GuSa97}. (As of now, it is unknown if the problem of deciding whether a finitely generated subgroup $H$ of $F$ is closed is decidable.)

Given a finitely generated subgroup $H$ of $F$, one can check if the generators of $F$ are accepted by $\La(H)$. 
If not, then $\Cl(H)$, and in particular $H$, is a proper subgroup of $F$. This however, only gives a partial solution to the generation problem in $F$. Indeed, there are finitely generated proper subgroups $H$ of $F$ such that $\Cl(H)=F$. 

To solve the generation problem in $F$, we study the core and the closure of subgroups $H$ of $F$. 
The first result in the study is the following characterization of the closure of subgroups of $F$. 

\begin{Theorem}\label{thm:GS_int}
Let $H\le F$. Then $\Cl(H)$ is the subgroup of $F$ consisting of piecewise-linear functions $f$, with finitely many pieces, such that each piece has dyadic endpoints and such that on each piece $f$ coincides with a restriction of some function from $H$.
\end{Theorem}


Theorem \ref{thm:GS_int} proves a conjecture of Guba and Sapir about the closure of subgroups of diagram groups (see Conjecture \ref{GSC} below), in the special case of Thompson's group $F$. It follows from Theorem \ref{thm:GS_int} that the orbits of the action of $H$ on the interval $[0,1]$ coincide with the orbits of the action of $\Cl(H)$.

In this paper, we denote by $\mathcal D$ the set of finite dyadic fractions. That is, $\mathcal D=\mathbb Z[\frac{1}{2}]\cap(0,1)$. The proof of Theorem \ref{thm:GS_int} enables us to get the following characterization of subgroups $H\le F$ which act transitively on the set  $\mathcal D$. Cells in $\La(H)$, as in diagrams $\Delta$, are either positive or negative. Here as well, positive cells have one top edge and two bottom edges. The core $\La(H)$ has natural initial and terminal vertices. Every other vertex is an \emph{inner vertex} of $\La(H)$. 

\begin{Theorem}\label{thm:tra_int}
Let $H\le F$. Then $H$ acts transitively on the set of finite dyadic fractions $\mathcal D$ if and only if the following conditions are satisfied. 
\begin{enumerate}
\item[(1)] Every edge in $\La(H)$ is the top edge of some positive cell. 
\item[(2)] There is a unique inner vertex in $\La(H)$. 
\end{enumerate}
\end{Theorem}

An \emph{inner edge} is an edge of $\La(H)$ whose endpoints are inner vertices of $\La(H)$. We observe that if one replaces ``inner vertex'' in Theorem \ref{thm:tra_int} by ``inner edge'', one gets the criterion for $\Cl(H)$ to contain $[F,F]$ (Lemma \ref{fin_ind} below).
To solve the generation problem in $F$ we give a criterion for the group $H$ to contain the derived subgroup of $F$ (equivalently, for $H$ to be a normal subgroup of $F$ \cite{CFP}). It turns out that one only has to add a somewhat technical condition to the requirement that $\Cl(H)$ contains $[F,F]$. 

\begin{Theorem}\label{thm:der_int}
Let $H\le F$. Then $H$ contains the derived subgroup of $F$ if and only if the following conditions hold. 
\begin{enumerate}
\item[(1)] $[F,F]\subseteq \Cl(H)$
\item[(2)] There is a function $h\in H$ which fixes a finite dyadic fraction $\alpha\in \mathcal D$ such that the slope $h'(\alpha^-)=1$ and the slope $h'(\alpha^+)=2$. 
\end{enumerate}
\end{Theorem}

Given a finite subset $X$ of $F$, we let $H$ be the group generated by $X$. Then it is (easily) decidable if condition (1) holds for $H$. In Section \ref{sec:tuples}, we give an algorithm for deciding if $H$ satisfies condition (2), given that $H$ satisfies condition (1). As one can also decide if $H[F,F]=F$, Theorem \ref{thm:der_int} gives a solution for the generation problem in Thompson's group $F$. 

\begin{Corollary}\label{cor_int}
Let $H$ be a subgroup of $F$. Then $H=F$ if and only if the following conditions hold. 
\begin{enumerate}
\item[(1)] $[F,F]\subseteq \Cl(H)$
\item[(2)] $H[F,F]=F$.
\item[(3)] There is a function $h\in H$ which fixes a finite dyadic fraction $\alpha\in \mathcal D$ such that the slope $h'(\alpha^-)=1$ and the slope $h'(\alpha^+)=2$. 
\end{enumerate}
\end{Corollary}

Another application of Theorem \ref{thm:der_int} is the following. 

\begin{Theorem} 
There is a sequence of finitely generated subgroups $B<K<F$ such that $B$ is elementary amenable and maximal in $K$,  $K$ is normal in $F$ and $F/K$ is infinite cyclic.
\end{Theorem}

In Section \ref{sec:tech} we give some techniques related to the core $\La(H)$ of a subgroup $H$ of $F$. 
In Section \ref{ss:alg} we consider the problem of finding a generating set of $\Cl(H)$, and show that an algorithm due to Guba and Sapir \cite{GuSa97} for finding a generating set of a diagram group can be simplified in that special case. In Section \ref{ss:core} we give conditions for a $2$-automaton $\La$ over the Dunce hat $\kk$ (see Section \ref{ss:dc}) to coincide with the core $\La(H)$ of some subgroup $H$ of $F$, up to identification of vertices. These techniques are applied in Section \ref{ss:max} to the construction of maximal subgroups of infinite index in $F$.

In \cite{Sav1, Sav} Dmytro Savchuk studied subgroups $H_U$ of the group $F$ which are the stabilizers of finite sets of real numbers $U\subset(0,1)$. He proved that if $U$ consists of one number, then $H_U$ is a maximal subgroup of $F$. He also showed that the Schreier graphs of the subgroups $H_U$ are amenable.  He asked \cite[Problem 1.5]{Sav} whether every maximal subgroup of infinite index in $F$ is of the form $H_{\{\alpha\}}$, that is, fixes a number from $(0,1)$. In \cite{GS1}, the author and Sapir applied the core of subgroups of $F$ to prove the existence of maximal subgroups of $F$ of infinite index which do not fix any number in $(0,1)$. An explicit example of a $3$-generated maximal subgroup that does not fix any number in $(0,1)$ was also constructed in \cite{GS1}, but all the examples from \cite{GS1} stabilize proper subsets of $\mathcal D$.  
Applying Corollary \ref{cor_int}, Theorem \ref{thm:tra_int} and the techniques from Section \ref{sec:tech} we prove the following.

\begin{Theorem}
Thompson's group $F$ has a $3$-generated maximal subgroup  of infinite index which acts transitively on the set of finite dyadic fractions $\mathcal D$. 
\end{Theorem}


In Section \ref{sec:sg}, we consider the closure of solvable subgroups $H$ of $F$ and prove the following theorem. 

\begin{Theorem}\label{sol_int}
Let $H$ be a solvable subgroup of $F$ of derived length $n$. Then the following assertions hold. 
\begin{enumerate}
\item[(1)] The action of $H$ on the set of finite dyadic fractions $\mathcal D$ has infinitely many orbits.
\item[(2)] $\Cl(H)$ is solvable of derived length $n$. 
\item[(3)] If $H$ is finitely generated then $\Cl(H)$ is finitely generated. 
\end{enumerate}
\end{Theorem}

The theorem follows from results about solvable subgroups of Thompson's group $F$ (and more generally, solvable subgroups of $\PL_o(I)$, the group of piecewise linear orientation preserving homeomorphisms of the unit interval $[0,1]$ with finitely many pieces) from \cite{Bl1,Bl2} and \cite{BBH}.

In Section \ref{ss:sol_alg}, we give a characterization of solvable subgroups $H$ of $F$ in terms of the core $\La(H)$. We define 
a directed graph $\mathcal P(H)$ related to the core $\La(H)$ (Definition \ref{def:PH}) and prove the following. 

\begin{Theorem}\label{Alg_sol_int}
Let $H$ be a finitely generated subgroup of $F$. Then $H$ is solvable if and only if there is no directed cycle in $\mathcal P(H)$.
\end{Theorem}

We give a similar characterization for solvable (not necessarily finitely generated) subgroups of $F$.

Theorem \ref{Alg_sol_int} makes use of the characterization of solvable subgroups of $F$ in terms of their towers (see Definition \ref{def:tower} below) due to Bleak \cite{Bl1}. When $H$ is finitely generated, Theorem \ref{Alg_sol_int} translates to a simple algorithm for deciding if $H$ is solvable. As mentioned above, there is an algorithm for determining the solvability of finitely generated subgroups of $F$ due to Bleak, Brough and Hermiller \cite{BBH}. In fact, the algorithm from \cite{BBH} applies to all  finitely generated computable subgroups of $\PL_o(I)$ (see \cite[Section 4]{BBH}). For finitely generated subgroups $H$ of $F$, the algorithm given by Theorem \ref{Alg_sol_int} is arguably easier than the one in \cite{BBH}. It has polynomial time complexity in the number of cells in the diagrams of the generators of $H$. It is proved in \cite{BBH} that the algorithm terminates on every input, but the proof does not yield any upper bound on the complexity. 

In Section \ref{open}, we list some open problems.

\vskip .4em

{\bfseries Acknowledgments.} The author would like to thank Mark Sapir for many helpful discussions and for comments on the text. The author would also like to thank the anonymous referee for his careful reading of the manuscript and for his helpful  comments and suggestions. 

\section{Preliminaries on $F$}\label{sec:pre}

\subsection{$F$ as a group of homeomorphisms}\label{ss:nf}

Recall that $F$ consists of all piecewise-linear increasing self-homeomorphisms of the unit interval with slopes of all linear pieces powers of $2$ and all break points of the derivative finite dyadic fractions. The group $F$ is generated by two functions $x_0$ and $x_1$ defined as follows \cite{CFP}.

	\[
   x_0(t) =
  \begin{cases}
   2t &  \hbox{ if }  0\le t\le \frac{1}{4} \\
   t+\frac14       & \hbox{ if } \frac14\le t\le \frac12 \\
   \frac{t}{2}+\frac12       & \hbox{ if } \frac12\le t\le 1
  \end{cases} 	\qquad	
   x_1(t) =
  \begin{cases}
   t &  \hbox{ if } 0\le t\le \frac12 \\
   2t-\frac12       & \hbox{ if } \frac12\le t\le \frac{5}{8} \\
   t+\frac18       & \hbox{ if } \frac{5}{8}\le t\le \frac34 \\
   \frac{t}{2}+\frac12       & \hbox{ if } \frac34\le t\le 1 	
  \end{cases}
\]

The composition in $F$ is from left to right.

Every element of $F$ is completely determined by how it acts on the set of finite dyadic fractions. Every number in $(0,1)$ can be described as $.s$ where $s$ is an infinite word in $\{0,1\}$. For each element $g\in F$ there exists a finite collection of pairs of (finite) words $(u_i,v_i)$ in the alphabet $\{0,1\}$ such that every infinite word in $\{0,1\}$ starts with exactly one of the $u_i$'s. The action of $F$ on a number $.s$ is the following: if $s$ starts with $u_i$, we replace $u_i$ by $v_i$. For example, $x_0$ and $x_1$  are the following functions:

\[
   x_0(t) =
  \begin{cases}
   .0\alpha &  \hbox{ if }  t=.00\alpha \\
    .10\alpha       & \hbox{ if } t=.01\alpha\\
   .11\alpha       & \hbox{ if } t=.1\alpha\
  \end{cases} 	\qquad	
   x_1(t) =
  \begin{cases}
   .0\alpha &  \hbox{ if } t=.0\alpha\\
   .10\alpha  &   \hbox{ if } t=.100\alpha\\
   .110\alpha            &  \hbox{ if } t=.101\alpha\\
   .111\alpha                      & \hbox{ if } t=.11\alpha\
  \end{cases}
\]

For the generators $x_0,x_1$ defined above, the group $F$ has the following finite presentation \cite{CFP}.
$$F=\la x_0,x_1\mid [x_0x_1^{-1},x_1^{x_0}]=1,[x_0x_1^{-1},x_1^{x_0^2}]=1\ra,$$ where $a^b$ denotes $b\iv ab$.

\subsection{$F$ as a diagram group} 

\subsubsection{Directed complexes and  diagram groups}\label{ss:dc}

The definition of $F$ we will use most often in this paper is that of $F$ as a diagram group. 
This section closely follows \cite{GSdc}.

\begin{Definition}
\label{dirc}
{\rm
For every directed graph $\Gamma$ let $\PP$ be the set of all (directed)
paths in $\Gamma$, including the empty paths. A {\em directed $2$-complex\/}
is a directed graph $\Gamma$ equipped with a set $\F$ (called the {\em set
of $2$-cells\/}), and three maps $\topp{\cdot}\colon\F\to\PP$,
$\bott{\cdot}\colon\F\to\PP$, and $^{-1}\colon\F\to\F$ called {\em top\/},
{\em bottom\/}, and {\em inverse\/} such that
\begin{itemize}
\item for every $f\in\F$, the paths $\topp(f)$ and $\bott(f)$ are non-empty
and have common initial vertices and common terminal vertices,
\item $^{-1}$ is an involution without fixed points, and
$\topp(f^{-1})=\bott(f)$, $\bott(f^{-1})=\topp(f)$ for every $f\in\F$.
\end{itemize}
}
\end{Definition}
We will usually assume that $\F$ is given with an orientation, that is, a subset
$\F^+\subseteq\F$ of {\em positive\/} $2$-cells, such that $\F$ is the
disjoint union of $\F^+$ and the set $\F^-=(\F^+)^{-1}$ (the latter is called
the set of {\em negative\/} $2$-cells).


If $\kk$ is a directed $2$-complex, then paths on $\kk$ are called
{\em $1$-paths\/} (we are going to have $2$-paths later). The initial and
terminal vertex of a $1$-path $p$ are denoted by $\iota(p)$ and
$\tau(p)$, respectively. For every $2$-cell $f\in\F$, the vertices
$\iota(\topp(f))=\iota(\bott(f))$ and $\tau(\topp(f))=\tau(\bott(f))$ are
denoted $\iota(f)$ and $\tau(f)$, respectively.

We shall denote each cell $f$ by $\topp(f)\to \bott(f)$. And we can denote a directed 2-complex $\kk$ similar to a semigroup presentation
$\la E\mid \topp(f)\to \bott(f), f\in F^{+}\ra$ where $E$ is the set of all edges of  $\kk$ (note that we ignore the vertices of $\kk$).

For example, the directed $2$-complex $\la x\mid x\to x^2\ra$ is the {\em Dunce hat\/}
obtained by identifying all edges in the triangle (Figure \ref{f:1})
according to their directions.  It has one vertex, one edge, and one
positive $2$-cell.

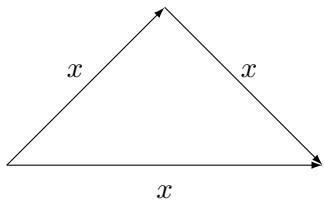
\begin{figure}[ht]
\begin{center}
\unitlength=0.7mm
\special{em:linewidth 0.4pt}
\linethickness{0.4pt}
\begin{picture}(61.00,36.00)
\put(1.00,6.00){\vector(1,1){30.00}}
\put(31.00,36.00){\vector(1,-1){30.00}}
\put(1.00,6.00){\vector(1,0){60.00}}
\put(31.00,1.00){\makebox(0,0)[cc]{$x$}}
\put(14.00,24.00){\makebox(0,0)[cc]{$x$}}
\put(47.00,24.00){\makebox(0,0)[cc]{$x$}}
\end{picture}

\caption{Dunce hat}
\label{f:1}
\end{center}
\end{figure}

With the directed $2$-complex $\kk$, one can associate a category $\Pi(\kk)$
whose objects are directed $1$-paths, and morphisms are {\em $2$-paths\/}, i.e., sequences of replacements of $\topp(f)$ by
$\bott(f)$ in $1$-paths, $f\in\F$. More precisely, an {\em atomic $2$-path\/}
(an {\em elementary homotopy\/}) is a triple $(p,f,q)$, where $p$, $q$ are
$1$-paths in $\kk$, and $f\in\F$ such that $\tau(p)=\iota(f)$,
$\tau(f)=\iota(q)$. If $\delta$ is the atomic $2$-path $(p,f,q)$, then
$p\topp(f)q$ is denoted by $\topp(\delta)$, and $p\bott(f)q$ is denoted by
$\bott(\delta)$; these are called the {\em top\/} and the {\em bottom\/}
$1$-paths of the atomic $2$-path. Every {\em nontrivial\/} $2$-path $\delta$
on $\kk$ is a sequence of atomic $2$-paths $\delta_1$, \dots, $\delta_n$, where
$\bott(\delta_i)=\topp(\delta_{i+1})$ for every $1\le i<n$. In this case $n$
is called the {\em length\/} of the $2$-path $\delta$. The {\em top\/} and
the {\em bottom\/} $1$-paths of $\delta$, denoted by $\topp(\delta)$ and
$\bott(\delta)$, are $\topp(\delta_1)$ and $\bott(\delta_n)$, respectively.
 Every $1$-path $p$ is considered as a trivial
$2$-path with $\topp(p)=\bott(p)=p$. We shall use any expression of the form $(u,\varepsilon,v)$ where $p= uv$ (and $u$ or $v$ are possiby empty $1$-paths) to denote it. 
These are the identity morphisms in the
category $\Pi(\kk)$. The composition of $2$-paths $\delta$ and $\delta'$ is
called {\em concatenation\/} and is denoted $\delta\circ\delta'$.

With every atomic $2$-path $\delta=(p,f,q)$, where $\topp(f)=u$, $\bott(f)=v$
we associate the labeled planar graph $\Delta$ in Figure \ref{f:2}. An arc
labeled by a word $w$ is subdivided into $|w|$ edges. All edges are
oriented from left to right. The label of each oriented edge of the
graph is a symbol from the alphabet $E$, the set of edges in $\kk$. As a
planar graph, it has only one bounded face; we label it by the corresponding
cell $f$ of $\kk$. This planar graph $\Delta$ is called the \emph{diagram of $\delta$}.
Such diagrams are called {\em atomic\/}. The leftmost and rightmost vertices
of $\Delta$ are denoted by $\iota(\Delta)$ and $\tau(\Delta)$, respectively. Directed paths in the diagram $\Delta$ are called $1$-paths. The diagram $\Delta$ has two distinguished $1$-paths from $\iota(\Delta)$ to
$\tau(\Delta)$ that correspond to the top and bottom $1$-paths of $\Delta$. Their
labels are $puq$ and $pvq$ respectively and they are denoted by $\topp(\Delta)$ and $\bott(\Delta)$.

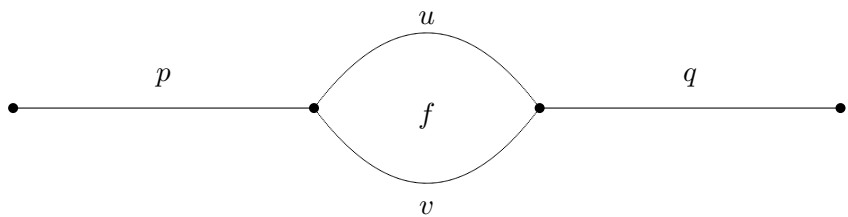
\begin{figure}[ht]
\begin{center}
\unitlength=1.00mm
\special{em:linewidth 0.4pt}
\linethickness{0.4pt}
\begin{picture}(110.66,26.33)
\put(0.00,13.33){\circle*{1.33}}
\put(40.00,13.33){\circle*{1.33}}
\put(70.00,13.33){\circle*{1.33}}
\put(110.00,13.33){\circle*{1.33}}
\bezier{200}(40.00,13.33)(55.00,33.33)(70.00,13.33)
\bezier{200}(40.00,13.33)(55.00,-6.67)(70.00,13.33)
\put(20.00,17.33){\makebox(0,0)[cc]{$p$}}
\put(55.00,25.33){\makebox(0,0)[cc]{$u$}}
\put(90.00,17.33){\makebox(0,0)[cc]{$q$}}
\put(55.00,12.33){\makebox(0,0)[cc]{$f$}}
\put(55.00,0.00){\makebox(0,0)[cc]{$v$}}
\put(0.33,13.33){\line(1,0){39.67}}
\put(70.00,13.33){\line(1,0){40.00}}
\end{picture}
\vspace{1ex}

\nopagebreak[4] 
\end{center}
\caption{An atomic diagram}
\label{f:2}
\end{figure}

The diagram corresponding to the trivial $2$-path $p$ is just an arc labeled
by $p$; it is called a {\em trivial diagram\/} and it is denoted by $\varepsilon(p)$.

Let $\delta=\delta_1\circ\delta_2\circ\cdots\circ\delta_n$ be a $2$-path
in $\kk$, where $\delta_1$, \dots, $\delta_n$ are atomic $2$-paths. Let
$\Delta_i$ be the atomic diagram corresponding to $\delta_i$. Then the
bottom $1$-path of $\Delta_i$ has the same label as the top $1$-path of $\Delta_{i+1}$
($1\le i<n$). Hence we can identify the bottom $1$-path of $\Delta_i$ with the
top $1$-path of $\Delta_{i+1}$ for all $1\le i<n$, and obtain a planar graph
$\Delta$, which is called the {\em diagram of the $2$-path\/} $\delta$.
If the top $1$-path of $\Delta$ is labeled $u$ and the bottom $1$-path is labeled $v$, then we say that $\Delta$ is a $(u,v)$-diagram. Similarly, a cell $\pi$ in a diagram $\Delta$ is a $(u,v)$-cell if it is labeled by a $2$-cell $f$ of $\kk$ with $\topp(f)=u$ and $\bott(f)=v$.
If $\Delta$ is a diagram of some $2$-path $\delta$ in $\kk$, we say that $\Delta$ is a diagram over $\kk$.

Two diagrams $\Delta_1,\Delta_2$ are considered {\em equal\/} if they are isotopic as planar
graphs. In that case, we write $\Delta_1\equiv \Delta_2$.  
The isotopy must take vertices to vertices, edges to edges, it must
also preserve labels of edges and cells. Two $2$-paths are
called {\em isotopic\/} if the corresponding diagrams are equal.

Concatenation of $2$-paths corresponds to concatenation of diagrams:
if the bottom $1$-path of $\Delta_1$ and the top $1$-path of $\Delta_2$ have the
same labels, we can identify them and obtain a new diagram
$\Delta_1\circ\Delta_2$.

Note that for any atomic $2$-path $\delta=(p,f,q)$ in $\kk$ one can naturally
define its {\em inverse\/} $2$-path $\delta^{-1}=(p,f^{-1},q)$. The inverses of
all $2$-paths and diagrams are defined naturally. The inverse diagram $\Delta^{-1}$
of $\Delta$ is obtained by taking the mirror image of $\Delta$ with respect to a
horizontal line, and replacing labels of cells by their inverses.

Let us identify in the category $\Pi(\kk)$ all isotopic $2$-paths and also identify
each $2$-path of the form
$\delta'\circ\delta\circ\delta^{-1}\circ\delta''$ with $\delta'\circ\delta''$ (in particular, since a $2$-path of the form
$\delta\circ\delta^{-1}$ is isotopic to $\delta_p\circ\delta\circ\delta^{-1}\circ\delta_p$ , where 
$\delta_p=(p,\varepsilon,\varepsilon)$ is the trivial $2$-path corresponding to the $1$-path $p=\topp(\delta)$, the $2$-path $\delta\circ\delta^{-1}$ is 
identified with the trivial $2$-path $\delta_p=(p,\varepsilon,\varepsilon)$).
The quotient category is obviously a groupoid (i.\,e. a category with invertible
morphisms). It is denoted by $\dd(\kk)$ and is called the {\em diagram groupoid\/}
of $\kk$. Two $2$-paths are called {\em homotopic\/} if they correspond to the same
morphism in $\dd(\kk)$. For each $1$-path $p$ of $\kk$, the local group of $\dd(\kk)$
at $p$ (i.e., the group of homotopy classes of $2$-paths connecting $p$ with
itself) is called the {\em diagram group of the directed $2$-complex $\kk$ with base\/}
$p$ and is denoted by $\DG(\kk,p)$.

The following theorem is proved in \cite{GuSa97} (see also \cite{GSdc}).

\begin{Theorem} If $\kk$ is the Dunce hat (see Figure \ref{f:1}) and $x$ is the edge of it,
then $\DG(\kk,x)$ is isomorphic to  R. Thompson's group $F$. The generators $x_0, x_1$ of $F$ are depicted in Figure \ref{f:xx} (all edges in the diagrams are labeled by $x$ and oriented from left to right).
\end{Theorem}

\begin{figure}[h!]
\begin{center}
\unitlength 1mm 
\linethickness{0.4pt}
\ifx\plotpoint\undefined\newsavebox{\plotpoint}\fi 
\begin{picture}(94.5,37)(0,0)
\put(39.5,24){\circle*{2}}
\put(29.5,24){\circle*{2}}
\put(19.5,24){\circle*{2}}
\put(9.5,24){\circle*{2}}
\put(93.5,24){\circle*{2}}
\put(83.5,24){\circle*{2}}
\put(73.5,24){\circle*{2}}
\put(63.5,24){\circle*{2}}
\put(53.5,24){\circle*{2}}
\put(39.5,24){\line(-1,0){10}}
\put(29.5,24){\line(-1,0){10}}
\put(19.5,24){\line(-1,0){10}}
\put(93.5,24){\line(-1,0){10}}
\put(83.5,24){\line(-1,0){10}}
\put(73.5,24){\line(-1,0){10}}
\put(63.5,24){\line(-1,0){10}}
\bezier{120}(29.5,24)(19.5,35)(9.5,24)
\bezier{120}(39.5,24)(30.5,13)(19.5,24)
\bezier{256}(39.5,24)(23.5,52)(9.5,24)
\bezier{256}(39.5,24)(28.5,-4)(9.5,24)
\bezier{132}(83.5,24)(75.5,37)(63.5,24)
\bezier{108}(93.5,24)(82.5,15)(73.5,24)
\bezier{208}(93.5,24)(76.5,45)(63.5,24)
\bezier{176}(93.5,24)(81.5,8)(63.5,24)
\bezier{296}(93.5,24)(79.5,55)(53.5,24)
\bezier{296}(93.5,24)(84.5,-6)(53.5,24)
\put(24.5,2){\makebox(0,0)[]{$x_0$}}
\put(73.5,2){\makebox(0,0)[]{$x_1$}}
\end{picture}
\end{center}
\caption{Generators of  R. Thompson's group $F$}
\label{f:xx}
\end{figure}
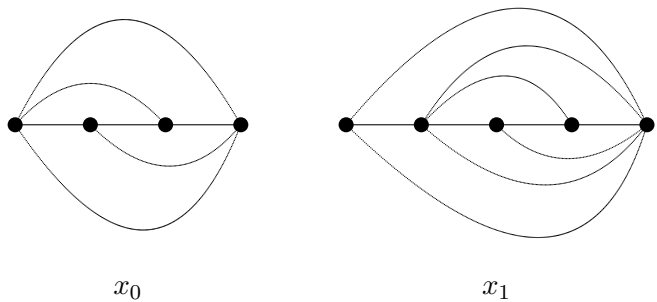

Diagrams $\Delta_1,\Delta_2$ over $\kk$ corresponding to homotopic $2$-paths are called
{\em equivalent\/} (denoted by $\Delta_1=\Delta_2$). The equivalence relation on the set of diagrams
(and the homotopy relation on the set of $2$-paths of $\kk$) can be
defined very easily as follows. We say that two cells $\pi_1$ and
$\pi_2$ in a diagram $\Delta$ over $\kk$ form a {\em dipole\/} if
$\bott(\pi_1)$ coincides with $\topp(\pi_2)$ and the labels of the
cells $\pi_1$ and $\pi_2$ are mutually inverse. Clearly, if $\pi_1$
and $\pi_2$ form a dipole, then one can remove the two cells from the
diagram and identify $\topp(\pi_1)$ with $\bott(\pi_2)$. The result
will be some diagram $\Delta'$. As in \cite{GuSa97}, it is easy to
prove that if $\delta$ is a $2$-path corresponding to $\Delta$, then
the diagram $\Delta'$ corresponds to a $2$-path $\delta'$, which is
homotopic to $\delta$. We call a diagram {\em reduced\/} if it does
not contain dipoles. A $2$-path $\delta$ in $\kk$ is called {\em reduced\/}
if the corresponding diagram is reduced.

Thus one can define morphisms in the diagram groupoid $\dd(\kk)$ as reduced
diagrams over $\kk$ with operation ``con\-cat\-enat\-ion + reduc\-tion''
(that is, the product of two reduced diagrams $\Delta$ and $\Delta'$ is the
result of removing all dipoles from $\Delta\circ\Delta'$ step by step; that process is confluent and terminating, so the result is uniquely determined \cite[Lemma 3.10]{GuSa97}). 
Thus, for each $1$-path $u$, the diagram group $\DG(\kk,u)$, is composed of all reduced $(u,u)$-diagrams over $\kk$. 
We will often consider a non-reduced diagram as an element of $\DG(\kk,u)$ identified with the reduced diagram equivalent to it.

One can naturally define a partial addition operation in the diagram groupoid $\dd(\kk)$. Let $\Delta_1$ and $\Delta_2$ be diagrams over $\kk$ with $\topp(\Delta_1)$ labeled $u$ and $\topp(\Delta_2)$ labeled $v$. If the $1$-paths $u$ and $v$ in $\kk$ satisfy $\tau(u)=\iota(v)$, then one can identify $\tau(\Delta_1)$ and $\iota(\Delta_2)$ to get a new diagram over $\kk$, $\Delta_1+\Delta_2$. Note that if $\kk$ has only one vertex, then the addition operation in $\dd(\kk)$ is everywhere defined.

\begin{Remark}[\cite{GSdc}]\label{one_vertex}
Let $\kk$ be a directed $2$-complex and let $\kk'$ be a directed $2$-complex resulting from $\kk$ by identification of vertices in $\kk$. Let $u$ be a $1$-path in $\kk$. In particular, $u$ is also a $1$-path in $\kk'$. Then the diagram groups $\DG(\kk,u)$ and $\DG(\kk',u)$ coincide. 
\end{Remark}

Remark \ref{one_vertex} is the reason we usually ignore vertices in the description of a directed $2$-complex $\kk$. We will follow this tradition in this paper. On occasion however, the vertices of $\kk$ will be important to us. On those occasions we will be careful to distinguish between different vertices of $\kk$. 




%

\subsubsection{A normal form of elements of $F$}

Let $x_0, x_1$ be the standard generators of $F$. Recall that $x_{i+1}, i\ge 1$,  denotes
$x_0^{-i}x_1x_0^i$.
In these generators, the group $F$ has the following presentation
$\la x_i, i\ge 0\mid x_i^{x_j}=x_{i+1} \hbox{ for every}$ $j<i\ra$ \cite{CFP}.

There exists a clear connection between representation of elements of $F$ by
diagrams and the normal form of elements in $F$. Recall~\cite{CFP} that every
element in $F$ is uniquely representable in the following form:
\begin{equation}\label{NormForm}
x_{i_1}^{s_1}\ldots x_{i_m}^{s_m}x_{j_n}^{-t_n}\ldots x_{j_1}^{-t_1},
\end{equation}
where $i_1\le\cdots\le i_m\ne j_n\ge\cdots\ge j_1$;
$s_1,\ldots,s_m,t_1,\ldots t_n\ge 1$, and if $x_i$ and $x_i\iv$ occur in
(\ref{NormForm}) for some $i\ge0$ then either $x_{i+1}$ or $x_{i+1}\iv$
also occurs in~(\ref{NormForm}).
This form is called the {\em normal form} of elements in $F$.

Let $\kk$ be the Dunce hat and let $x$ be the edge of $\kk$. Every cell in a diagram $\Delta$ over $\kk$ is either an $(x,x^2)$-cell or an $(x^2,x)$-cell. A dipole $\pi_1\circ\pi_2$ in $\Delta$ is a dipole of \emph{type $1$} if $\pi_1$ is an $(x,x^2)$-cell and $\pi_2$ is an $(x^2,x)$-cell. Otherwise, the dipole is of \emph{type 2}. 
 It was noticed in \cite{GS1} that if $\Delta$ is a diagram over $\kk$ with no dipoles of type $2$, then there is a (unique) \emph{horizontal $1$-path} in $\Delta$; i.e., a $1$-path from $\iota(\Delta)$ to $\tau(\Delta)$ which passes through all the vertices in the diagram. The horizontal $1$-path of $\Delta$ separates it into two parts,  
 {\em positive} and {\em negative}, denoted by $\Delta^+$ and $\Delta^-$
respectively. So $\Delta\equiv\Delta^+\circ\Delta^-$. It is easy to prove by
induction on the number of cells that all cells in $\Delta^+$ are
$(x,x^2)$-cells and all cells in $\Delta^-$ are $(x^2,x)$-cells.

Let us show how given a reduced diagram $\Delta$ in $\DG(\kk,x)$ one can
get the normal form of the element of $F$ represented by this diagram. This is the left-right dual of the procedure described in \cite[Example 2]{GS1} and after Theorem 5.6.41 in \cite{Sa}.

\begin{lemma}\label{lm1}
Let us number the cells of $\Delta^+$ by
numbers from $1$ to $k$ by taking every time the ``leftmost'' cell,
that is, the cell which is to the left of any other cell attached to the
bottom $1$-path of the diagram formed by the previous cells. The first cell is
attached to the top $1$-path of $\Delta^+$ (which is the top $1$-path of $\Delta$). The $i^{th}$ cell in
this sequence of cells corresponds to an atomic diagram,
which has the form $(x^{\ell_i},x\to x^2,x^{r_i})$, where $\ell_i$ ($r_i$)
is the length of the $1$-path from the initial (resp. terminal) vertex of the
diagram (resp. the cell) to the initial (resp. terminal) vertex of the cell
(resp. the diagram), such that the $1$-path is contained in the bottom $1$-path of the diagram formed by the first $i-1$ cells.
If $r_i=0$ then we label this cell by 1. If
$r_i\ne0$ then we label this cell by the element $x_{\ell_i}$ of $F$.
Multiplying the labels of all cells, we get the ``positive'' part of the
normal form. In order to find the ``negative'' part of the normal form, consider
$(\Delta^-)\iv$, number its cells as above and label them as above. The normal form of $\Delta$ is then the product of the normal form of $\Delta^+$ and the inverse of the normal form of $(\Delta^-)\iv$.
\end{lemma}

For example, applying the procedure from Lemma \ref{lm1} to the diagram on Figure \ref{f3}
\begin{figure}[h!]
\begin{center} 
\unitlength=0.7mm
\special{em:linewidth 0.4pt}
\linethickness{0.4pt}
\begin{picture}(83.00,90.00)
\put(2.00,43.00){\circle*{2.00}}
\put(12.00,43.00){\circle*{2.00}}
\put(22.00,43.00){\circle*{2.00}}
\put(32.00,43.00){\circle*{2.00}}
\put(42.00,43.00){\circle*{2.00}}
\put(52.00,43.00){\circle*{2.00}}
\put(62.00,43.00){\circle*{2.00}}
\put(72.00,43.00){\circle*{2.00}}
\put(82.00,43.00){\circle*{2.00}}
\put(2.00,43.00){\line(1,0){10.00}}
\put(12.00,43.00){\line(1,0){10.00}}
\put(22.00,43.00){\line(1,0){10.00}}
\put(32.00,43.00){\line(1,0){10.00}}
\put(42.00,43.00){\line(1,0){10.00}}
\put(52.00,43.00){\line(1,0){10.00}}
\put(62.00,43.00){\line(1,0){10.00}}
\put(72.00,43.00){\line(1,0){10.00}}
\bezier{132}(62.00,43.00)(72.00,56.00)(82.00,43.00)
\bezier{132}(62.00,43.00)(52.00,56.00)(42.00,43.00)
\bezier{132}(12.00,43.00)(22.00,56.00)(32.00,43.00)
\bezier{212}(12.00,43.00)(26.00,65.00)(42.00,43.00)
\bezier{392}(12.00,43.00)(33.00,85.00)(62.00,43.00)
\bezier{516}(2.00,43.00)(33.00,100.00)(62.00,43.00)
\bezier{676}(2.00,43.00)(34.00,117.00)(82.00,43.00)
\bezier{132}(22.00,43.00)(31.00,30.00)(42.00,43.00)
\bezier{212}(22.00,43.00)(30.00,22.00)(52.00,43.00)
\bezier{120}(52.00,43.00)(63.00,32.00)(72.00,43.00)
\bezier{304}(12.00,43.00)(29.00,11.00)(52.00,43.00)
\bezier{412}(2.00,43.00)(30.00,-2.00)(52.00,43.00)
\bezier{556}(2.00,43.00)(35.00,-17.00)(72.00,43.00)
\bezier{740}(2.00,43.00)(34.00,-40.00)(82.00,43.00)
\put(57.00,64.00){\makebox(0,0)[cc]{1}}
\put(72.00,46.00){\makebox(0,0)[cc]{7}}
\put(16.00,56.00){\makebox(0,0)[cc]{2}}
\put(40.00,55.00){\makebox(0,0)[cc]{3}}
\put(52.00,46.00){\makebox(0,0)[cc]{6}}
\put(33.00,48.00){\makebox(0,0)[cc]{4}}
\put(22.00,46.00){\makebox(0,0)[cc]{5}}
\put(38.00,6.00){\makebox(0,0)[cc]{1}}
\put(48.00,23.00){\makebox(0,0)[cc]{2}}
\put(62.00,40.00){\makebox(0,0)[cc]{7}}
\put(17.00,30.00){\makebox(0,0)[cc]{3}}
\put(22.00,35.00){\makebox(0,0)[cc]{4}}
\put(41.00,38.00){\makebox(0,0)[cc]{5}}
\put(32.00,39.00){\makebox(0,0)[cc]{6}}
\end{picture}
\end{center}
\caption{Reading the normal form of an element of $F$ off its diagram.}
\label{f3}
\end{figure}
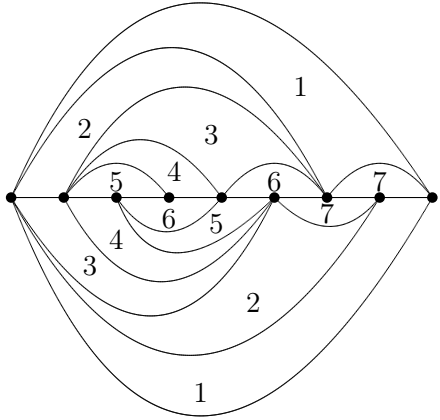
\noindent we get the normal form
$x_0x_1^3x_4(x_0^2x_1x_2^2x_5)\iv$.

\subsection{The relation between $F$ as a diagram group and $F$ as a group of homeomorphisms}\label{sec:bra}

One can define an isomorphism from $F$ viewed as a diagram group to $F$ as a group of homeomorphisms of $[0,1]$. To do so, we define pairs of branches of a diagram $\Delta$. 

Let $\Delta$ be a diagram in $F$, i.e., a diagram in $\DG(\kk,x)$, where $\kk$ is the Dunce hat. 
All diagrams in $F$ considered in this paper will be assumed to have no dipoles of type $2$. Thus, we will not mention it when considering a diagram $\Delta$ in $F$. Since $\Delta$ has no dipoles of type $2$, $\Delta\equiv\Delta^+\circ\Delta^-$.

A \emph{tree-diagram} is an $(x,x^n)$-diagram $\Psi$ over $\kk$ for some $n\in\mathbb{N}$ where all cells are $(x,x^2)$-cells. 
We note that if $\Delta$ is a diagram in $F$ with no dipoles of type $2$ then $\Delta^+$ is a tree-diagram and so is $(\Delta^-)^{-1}$.

Given a tree-diagram $\Psi$, one can put a vertex at the middle of each edge of $\Psi$ and inside every cell $\pi$ of $\Psi$ draw two edges; from the vertex on $\topp(\pi)$ to the vertex on the left bottom edge of $\pi$ and  from the vertex on $\topp(\pi)$ to the vertex on the right bottom edge of $\pi$. 
The result is a finite binary tree with $n$ leaves, where $n$ is the number of edges in $\bott(\Psi)$. 
A \emph{path} on a rooted  binary tree is a directed simple path starting from the root. A \emph{branch} in a binary tree is a maximal path. That gives rise to the following definition of paths on tree-diagrams. Note the difference from $1$-paths in diagrams defined above.

\begin{Definition}\label{def:branch}
Let $\Psi$ be a  tree-diagram over $\kk$. A \emph{path} on $\Psi$ is a sequence of edges $e_1,\dots,e_m$ such that 
\begin{enumerate}
\item[(1)] $e_1=\topp(\Psi)$.
\item[(2)] For each $i=1,\dots,m-1$, there is an $(x,x^2)$-cell $\pi_i$ in $\Psi$  with $\topp(\pi_i)=e_i$ and such that $e_{i+1}$ is either the left or the right bottom edge of $\pi_i$.
\end{enumerate}
A maximal path on $\Psi$ is called a \emph{branch} of $\Psi$. If $p=e_1,\dots,e_m$ is a path on $\Psi$, then the \emph{label} $\lab(p)$ of $p$ is defined to be a binary word $u$ of length $m-1$, $u\equiv u_1\cdots u_{m-1}$\footnote{Throughout this paper, for words $u$ and $v$, $u\equiv v$ denotes letter-by-letter equality.}, where for each $i$, the letter $u_i\equiv 0$ if it corresponds to a step to the left (i.e., if $e_{i+1}$ is the left bottom edge of the cell $\pi_i$); $u_i\equiv 1$ if it corresponds to a step to the right.

If $\Psi$ is a tree-diagram, paths on $\Psi^{-1}$ are defined in a similar way. The initial edge of a path on $\Psi^{-1}$ is $\bott(\Psi^{-1})$. Every other edge in the path is one of the edges on the top $1$-path of an $(x^2,x)$-cell in $\Psi^{-1}$ whose bottom edge is the preceding edge in the path. 
 If $\Delta$ is a diagram in $F$ (with no dipoles of type $2$) then paths on $\Delta^+$ are said to be \emph{positive paths} on $\Delta$. Paths on $\Delta^-$ are \emph{negative paths} on $\Delta$. Positive and negative branches of $\Delta$ are defined in a similar way. 
\end{Definition}

A path $p$ on a tree-diagram $\Psi$ is uniquely determined by its label $u$. Thus, we will often consider the path and its label as the same object. If $p$ is a path on $\Psi$ and $\lab(p)\equiv u$, we will denote by $p^+$ or by $u^+$ the last edge in the path. Since for each branch $p$ of $\Psi$ the terminal edge $p^+$ lies on the $1$-path $\bott(\Psi)$, the branches of $\Psi$ are naturally ordered from left to right. If $p$ is a positive (resp. negative) path on a diagram $\Delta$, then the terminal edge $p^+$ is an edge of $\Delta^+$ (resp. $\Delta^-$). If $p$ is a positive or negative branch of $\Delta$, then $p^+$ lies on the horizontal $1$-path of $\Delta$. 

Let $\Psi$ be a tree-diagram over $\kk$. We make the following observation about consecutive branches in $\Psi$.
It will be used often throughout the paper with no specific reference.
\begin{Remark}\label{cons}
Let $u_1$ and $u_2$ be (the labels of) consecutive branches of $\Psi$. Let $u$ be the longest common prefix of $u_1$ and $u_2$ ($u$ can be empty). 
Then $$u_1\equiv u01^{m}\ \mbox{ and }\ u_2\equiv u10^{n}$$
for some $m,n\ge 0$. 
\end{Remark}

Let $\Delta$ be a diagram in $F$. Let $u_1,\dots,u_n$ be the (labels of the) positive branches of $\Delta$ and $v_1,\dots,v_n$ be the (labels of the)
 negative branches of $\Delta$, ordered from left to right. For each $i=1,\dots,n$, we say that $\Delta$ has  a \emph{pair of branches} $u_i\rightarrow v_i$. The function $g$ from $F$ corresponding to this diagram takes binary expansion $.u_i\omega$ to $.v_i\omega$ for every $i$ and every infinite binary word $\omega$. We will also say that the element $g$ takes the branch $u_i$ to the branch $v_i$.



If $e$ is an edge on the horizontal $1$-path of $\Delta$, then one can replace $e$ by a dipole $\pi_1\circ\pi_2$ of type $1$ to get an equivalent diagram $\Delta'$ with no dipoles of type $2$. It is obvious that $\Delta$ and $\Delta'$ represent the same homeomorphism of $[0,1]$. 

A slightly different way of describing the function in $F$ corresponding to a given diagram $\Delta$ is the following. For each finite binary word $u$, we let the \emph{interval associated with $u$}, denoted by $[u]$, be the dyadic interval $[.u,.u1^\N]$. If $\Delta$ is a diagram representing a homeomorphism $f\in F$, we let $u_1,\dots,u_n$ be the positive branches of $\Delta$ and $v_1,\dots,v_n$ be the negative branches of $\Delta$. Then the intervals $[u_1],\dots,[u_n]$ (resp. $[v_1],\dots,[v_n]$) form a subdivision of the interval $[0,1]$. The function $f$ maps each interval $[u_i]$ linearly onto the interval $[v_i]$.

Below, when we say that a function $f$ has a pair of branches $u\rightarrow v$, the meaning is that some diagram representing $f$ has this pair of branches. In other words, this is equivalent to saying that $f$ maps $[u]$ linearly onto $[v]$. In particular, if $f$ takes the branch $u$ to the branch $v$, then for any finite binary word $w$, $f$ takes the branch $uw$ to the branch $vw$, where $uw$ and $vw$ are concatenated words. 
Note also that if $f$ has a pair of branches $u\rightarrow v$ then the reduced diagram $\Delta$ of $f$ has a pair of branches $u_1\rightarrow v_1$ where $u\equiv u_1w$ and $v\equiv v_1w$ for some common (possibly empty) suffix $w$. 

We will often be interested in finite dyadic fractions $\alpha\in \mathcal D$ fixed by a function $f\in F$. 
More generally, if $S\subset(0,1)$, then we say that an element $f\in F$ \emph{fixes} $S$, if it fixes $S$ pointwise. 
The following lemma will be useful. 

\begin{Lemma}\label{4parts}
Let $f\in F$ be an element which fixes some finite dyadic fraction $\alpha\in \mathcal D$. Let $u\equiv u'1$ be the finite binary word such that $\alpha=.u$. Then the following assertions hold.
\begin{enumerate}
\item[(1)] $f$ has a pair of branches $u0^{m_1}\rightarrow u0^{m_2}$ for some $m_1,m_2\ge 0$. 
\item[(2)] $f$ has a pair of branches $u'01^{n_1}\rightarrow u'01^{n_2}$ for some $n_1,n_2\ge 0$.
\item[(3)] If $f'(\alpha^+)=2^k$ for $k\neq 0$, then every diagram representing $f$ has a pair of branches $u0^m\rightarrow u0^{m-k}$ for some $m\ge \max\{0,k\}$. 
\item[(4)] If $f'(\alpha^-)=2^\ell$ for $\ell\neq 0$, then every diagram representing $f$ has a pair of branches $u'01^n\rightarrow u'01^{n-\ell}$ for some $n\ge \max\{0,\ell\}$. 
\end{enumerate}
\end{Lemma}

\begin{proof}
We only prove parts (1) and (3). The proof of parts (2) and (4) is analogous. 

(1) Let $\Delta$ be a diagram of $f$ such that for some $m_1,m_2\ge 0$, $u0^{m_1}$ is a positive branch and $u0^{m_2}$ is a negative branch of $\Delta$ (such a diagram clearly exists; indeed, one can always prolong branches of $\Delta$ by inserting dipoles of type $1$). Since $f$ fixes $\alpha=.u$, it must take the branch $u0^{m_1}$ to the branch $u0^{m_2}$. 

(3) Let $\Delta$ be a diagram of $f$.  We claim that $u0^{m_1}$ must be a positive branch in $\Delta$ for some $m_1\ge 0$. Otherwise, $\Delta$ has a positive branch $u_1$ where $u_1$ is a proper prefix of $u$. In that case, $\alpha$ belongs to the interior of $[u_1]$. Let $v_1$ be the negative branch of $\Delta$ such that $u_1\rightarrow v_1$ is a pair of branches of $\Delta$. We assume that $|u_1|\le |v_1|$, the argument being similar in the opposite case.  
If $u_1$ is not a prefix of $v_1$ then $[u_1]$ and $[v_1]$ are disjoint in contradiction to $f$ fixing $\alpha$. Thus, $v_1\equiv u_1s$. If $s$ is empty then the function $f$ fixes the interval $[u_1]$, in contradiction to the slope $f'(\alpha^+)\neq 1$. 
Otherwise, $f$ fixes the number $.u_1s^{\mathbb{N}}$. Since $f$ is linear on $[u_1]$, the fixed point $.u_1s^{\mathbb{N}}$ must coincide with $\alpha$. Thus, $.u=.u_1s^{\N}$ which implies that $s\equiv 0^r$ or $s\equiv 1^r$ for some $r\in\mathbb{N}$, as $.u$ is finite dyadic. Then the equality $.u=.u_1s^{\mathbb N}$ gives a contradiction, since $u$ ends with $1$ and $u_1$ is a proper prefix of $u$. 


A similar argument shows that $\Delta$ must have a negative branch of the form $u0^{m_2}$ for some $m_2\ge 0$. 
As in (1) it follows that $\Delta$ has a pair of branches $u0^{m_1}\rightarrow u0^{m_2}$. 
The assumption in (3) and the linearity of $f$ on $[u0^{m_1}]$ imply that the slope on the interval is $2^k$. 
Thus $m_2=m_1-k$, and so $m_1\ge \max\{0,k\}$.
\end{proof}

\subsection{On orbitals and stabilizers in $F$}

We will often consider the action of $F$ on the interval $[0,1]$. Let $\PL_o(I)$ be the group of piecewise-linear orientation-preserving homeomorphisms of $[0,1]$ with finitely many breakpoints. Thompson's group $F$ is clearly a subgroup of $\PL_o(I)$. 

Let $G\le \PL_o(I)$. The \emph{support of $G$ in $[0,1]$}, denoted $\Supp(G)$, is the closure of the set $S$ of all points in $(0,1)$, which are not fixed by $G$. $S$ is a union of countably many open intervals. Each such open interval is called an orbital of $G$. Equivalently, an orbital of $G$ can be defined as the the convex hull of an orbit of a point $x$ in $(0,1)$, under the action of $G$, if $x$ is not fixed by $G$.
If $h\in \PL_o(I)$, then the support of $h$ and the orbitals of $h$ are defined to be the support and orbitals of the group $\la h\ra$. Notice that an element $h\in \PL_o(I)$ has finitely many orbitals and that the orbitals of $h$ coincide with the orbitals of $h^n$ for all $n\neq 0$. 

An interval $(a,b)$ is an orbital of $h$ if and only if $h$ fixes the points $a$ and $b$ and does not fix any number in $(a,b)$. The orbital $(a,b)$ is said to be a \emph{push-up} orbital of $h$, if for all $x\in (a,b)$, $f(x)>x$; equivalently, if the slope $f'(a^+)>1$. Similarly, the orbital $(a,b)$ is said to be a \emph{push-down} orbital of $h$, if for all $x\in (a,b)$, $f(x)<x$; equivalently, if the slope $f'(a^+)<1$. Notice that every orbital of $h$ is either a push-up or a push-down orbital. If $(a,b)$ is a push-up orbital of $h$ then $(a,b)$ is a push-down orbital of $h^{-1}$. 

The following observation will be used below with no specific reference. We say that an element $h$ has support in an interval $J$ if the support of $h$ is contained in $J$.

\begin{Remark}
If $h,g\in \PL_o(I)$ and $(a,b)$ is an orbital of $h$, then $h^g$ has an orbital $(g(a),g(b))$.
If the support of $h$ is contained in $J_1$, then the support of $h^g$ is contained in $g(J_1)$.
If $h$ fixes an interval $J_2$ then $h^g$ fixes the interval $g(J_2)$. 
\end{Remark}


\subsection{The derived subgroup of $F$}\label{der_sub}

The derived subgroup of $F$ is a simple group \cite{CFP}. It can be characterized as the subgroup of $F$ of all functions $f$ with slope $1$ both at $0^+$ and at $1^-$. In other words, it is the subgroup of all functions in $F$ with support in $(0,1)$. 

%

We will often be interested in subgroups of $F$ which are not contained in any finite index subgroup of $F$. 
Since $[F,F]$ is infinite and simple, every finite index subgroup of $F$ contains the derived subgroup of $F$. Thus, for a subgroup $H\le F$ we have $H[F,F]=F$ if and only if $H$ is not contained in any proper subgroup of finite index in $F$. 
To determine if $H[F,F]=F$, one can consider the image of $F$ in the abelianization $\mathbb{Z}^2$ of $F$. The abelianization map $\pi_{ab}\colon F\to \mathbb{Z}^2$ sends an element $f\in F$ to $(\log_2(f'(0^+)),\log_2(f'(1^-)))$ (see, for example, \cite{CFP}). 
Clearly, $H[F,F]=F$ if and only if 
 $\pi_{\ab}(H)=\pi_{\ab}(F)=\mathbb{Z}^2$. 

Let $a<b$ be finite dyadic fractions in $\mathcal D$. We denote by $F_{[a,b]}$ the subgroup of $F$ of all functions with support in $[a,b]$. The group $F_{[a,b]}$ is isomorphic to $F$. Indeed, they are conjugate subgroups of $\PL_2(\mathbb{R})$, the group of all piecewise linear homeomorphisms of $\mathbb{R}$ with finitely many breakpoints, all of which are finite dyadic fractions and where all slopes are integer powers of $2$. This isomorphism implies that the derived subgroup of $F_{[a,b]}$ is the subgroup of all functions with slope $1$ both at $a^+$ and at $b^-$. In other words, it is the subgroup of $F$ of all functions with support in $(a,b)$. 

\section{The Stallings $2$-core of subgroups of diagram groups}\label{sec:cor}

The Stallings $2$-core of a subgroup of a diagram group was defined in 1999 by Guba and Sapir and appeared first in print in \cite{GS1}. The motivation was to develop a method for checking if a subgroup $H$ of $F$ is a strict subgroup of $F$ (equivalently if $\{x_0,x_1\}\not\subseteq H$).  This section follows \cite{GS1} closely. 

Recall the procedure (first discovered by Stallings \cite{Sta}) of checking if an element $g$ of a free group $F_n$ belongs to the subgroup $H$ generated by elements $h_1,...,h_k$. Take paths labeled by $h_1,...,h_k$. Identify the initial and terminal vertices of these paths to obtain a bouquet of circles $K'$ with a distinguished vertex $o$. Then \emph{fold} edges as follows: if there exists a vertex with two outgoing edges of the same label, we identify the edges. As a result of all the foldings (and removing the hanging trees), we obtain the \emph{Stallings core} of the subgroup $H$ which is a finite automaton $A(H)$ with $o$ as its input/output vertex. Then $g\in H$ if and only if $A(H)$ accepts the reduced form of $g$. It is well known that the Stallings core does not depend on the generating set of the subgroup $H$.

In the case of diagram groups, an analogous construction was given in \cite{GS1}. Instead of automata we have directed $2$-complexes and instead of words -- diagrams. 

\begin{Definition} Let $\kk=\la E_{\kk}\mid \F_{\kk}^+\ra$ be a directed 2-complex. A \emph{2-automaton} over $\kk$ is a directed 2-complex $\La=\la E_{\La}\ |\F_{\La}^+\ra$ with two distinguished 1-paths $p_{\La}$ and $q_{\La}$ (the input and output 1-paths), together with a map $\phi$ from $\La$ to $\kk$ which takes vertices to vertices, edges to edges and cells to cells, which is a homomorphism of directed graphs and commutes with the maps $\topp, \bott$ and $\iv$.
We shall call $\phi$ an \emph{immersion}.
\end{Definition}

Given a diagram $\Delta$ over $\kk$, we can naturally view $\Delta$ as a directed $2$-complex, by considering every cell in $\Delta$ to be a pair of inverse cells.  Then $\Delta$ is a 2-automaton with a natural immersion $\phi_\Delta$ and the distinguished $1$-paths $\topp(\Delta)$ and $\bott(\Delta)$.

\begin{Definition} Let $\La, \La'$ be two 2-automata over $\kk$. A map $\psi$ from $\La'$ to $\La$ which takes vertices to vertices, edges to edges and cells to cells, which is a homomorphism of directed graphs and commutes with the maps $\topp, \bott, \iv$ and the immersions is called a \emph{morphism} from $\La'$ to $\La$ provided $\psi(p_{\La'})=p_{\La}, \psi(q_{\La'})=q_{\La}$.
\end{Definition}

\begin{Definition} We say that a 2-automaton $\La$ over $\kk$ \emph{accepts} a diagram $\Delta$ over $\kk$ if there is a morphism $\psi$ from the 2-automaton $\Delta$ to the 2-automaton $\La$.
\end{Definition}

Let $\Delta_i$, $i\in\mathcal I$ be reduced diagrams from the diagram group $\DG(\kk,u)$ i.e., diagrams over $\kk$ with the same label $u$ of their top and bottom $1$-paths. Then we can identify all $\topp(\Delta_i)$ with all $\bott(\Delta_i)$ and obtain a 2-automaton $\La'$ over $\kk$ with the distinguished $1$-paths $p=q=\topp(\Delta_i)=\bott(\Delta_i)$. We can view $\La'$ as a ``bouquet of spheres''. That automaton accepts any concatenation of diagrams $\Delta_i$ and their inverses. 

To get a 2-automaton that accepts all reduced diagrams in the subgroup generated by $\{\Delta_i\mid i\in\mathcal I\}$, we do an analog of the Stallings foldings. Namely, let $\La'$ be the 2-automaton as above.  Now every time we see two cells that have the same image under the immersion of $\La'$ and share the top (resp. bottom) $1$-path, then we identify their bottom (resp. top) $1$-paths and identify the cells too. This operation is called \emph{folding of cells} (see \cite[Remark 8.8]{GSdc}).

The result (after infinitely many foldings if $\mathcal I$ is infinite) is a directed $2$-complex and the immersion of $\La'$ induces an immersion of the new directed 2-complex. Thus we again get a 2-automaton. Let $\La$ be the 2-automaton obtained after all possible foldings were applied to $\La'$.
The following $3$ lemmas were proved in \cite{GS1}.

\begin{Lemma}
The $2$-automaton $\La$ does not depend on the order in which foldings were applied to $\La'$. 
\end{Lemma}

\begin{Lemma}
If the 2-automaton $\La$ accepts a diagram $\Delta$ in $\DG(\kk,u)$, then it also accepts the reduced diagram equivalent to $\Delta$. Thus, one can talk about the subgroup of $\DG(\kk,u)$ of all diagrams accepted by $\La$. 
\end{Lemma}

\begin{Lemma}\label{l:au}
The 2-automaton $\La$ accepts all reduced diagrams from the subgroup of the diagram group $\DG(\kk,u)$ generated by $\{\Delta_i\mid i\in\mathcal I\}$.
\end{Lemma}

It was noted in \cite{GS1} that the $2$-automaton $\La$ is determined uniquely by the subgroup $H=\la \Delta_i\mid i\in\mathcal I\ra$ of $\DG(\kk,u)$ and does not depend on the chosen generating set $\{\Delta_i\mid i\in\mathcal I\}$ (as long as all diagrams $\Delta_i$ are taken to be reduced).  Thus, $\La$ can be called the {\em Stallings 2-core} of the subgroup $H$. We will denote the Stallings $2$-core of a subgroup $H$ by $\La(H)$.

We note that unlike for subgroups of free groups, the Stallings $2$-core of a subgroup $H$ can accept diagrams not in $H$. Following \cite{GS1}, we make the following definition.

\begin{Definition} The \emph{closure} $\Cl(H)$ of a subgroup $H$ of a diagram group $\DG(\kk,u)$ is the subgroup of $\DG(\kk,u)$ consisting of all diagrams that are accepted by the 2-core $\La(H)$ of $H$. If $H=\Cl(H)$ we say that $H$ is a \emph{closed} subgroup of $\DG(\kk,u)$.
\end{Definition}

We note that all usual conditions of the closure operation are satisfied, that is, $H\le \Cl(H)$ (Lemma \ref{l:au}), $\Cl(\Cl(H))=\Cl(H)$ (indeed, the core $\La(\Cl(H))$ coincides with the core of $H$)  and if $H_1\le H_2$, then $\Cl(H_1)\le\Cl(H_2)$. If $H$ is finitely generated, then the core $\La(H)$ is finite, and so the membership problem in $\Cl(H)$ is decidable.

\vskip .2cm

We demonstrate the construction of the Stallings $2$-core of the subgroup $H=\la x_0,x_1x_2x_1^{-1}\ra$ of Thompson's group $F$ and demonstrate how to check if an element $f\in F$ belongs to $\Cl(H)$.  Let us denote the positive cell of the Dunce hat $\la x\mid x\to x^2\ra$ by $\pi$. The diagrams for $x_0$ and $x_1x_2x_1\iv$ viewed as 2-automata are in Figure \ref{f:xxxx} below (the immersion to the Dunce hat maps all positive cells to $\pi$, and all edges to the only edge of the Dunce hat).

\begin{figure}[ht]
\begin{center}
\unitlength .7mm 
\linethickness{0.4pt}
\ifx\plotpoint\undefined\newsavebox{\plotpoint}\fi 
\begin{picture}(181.291,100)(0,0)
\put(21.75,43.25){\line(1,0){63.5}}
\qbezier(21.75,43.25)(56.25,101.375)(84.75,43)
\qbezier(22,43.5)(48,71.625)(66,43.25)
\qbezier(22,43.25)(57,-15.75)(85,43.25)
\qbezier(44.5,43.25)(62.875,6.25)(84.75,43.25)
\put(22,43.25){\circle*{1.581}}
\put(95.5,42.75){\circle*{1.581}}
\put(105.75,43){\circle*{1.581}}
\put(117,43){\circle*{1.581}}
\put(139.75,42.75){\circle*{1.581}}
\put(163.5,42.5){\circle*{1.581}}
\put(180.5,42.75){\circle*{1.581}}
\put(45,43.25){\circle*{1.581}}
\put(65.75,43.5){\circle*{1.581}}
\put(84.75,43.75){\circle*{1.581}}
\put(59.75,60.5){\makebox(0,0)[cc]{$\gamma_1$}}
\put(44.75,49.5){\makebox(0,0)[cc]{$\gamma_2$}}
\put(64,34.75){\makebox(0,0)[cc]{$\gamma_3$}}
\put(40.75,29.5){\makebox(0,0)[cc]{$\gamma_4$}}
\put(54.75,75){\makebox(0,0)[cc]{$e_1$}}
\put(44,60.5){\makebox(0,0)[cc]{$e_2$}}
\put(32.5,46.25){\makebox(0,0)[cc]{$e_3$}}
\put(55,45.75){\makebox(0,0)[cc]{$e_4$}}
\put(75,46){\makebox(0,0)[cc]{$e_5$}}
\put(61,23){\makebox(0,0)[cc]{$e_6$}}
\put(53.75,11.25){\makebox(0,0)[cc]{$e_7$}}
\put(94.5,43){\line(1,0){86.5}}
\qbezier(94.75,43)(132.375,140)(180.5,43)
\qbezier(180.5,43)(146.5,-41.625)(94.5,43.25)
\qbezier(105,42.75)(132.25,113.25)(180.5,42.75)
\qbezier(105.25,43)(136.5,86.5)(163.75,43)
\qbezier(116.75,43)(134.75,69.375)(163.75,42.25)
\qbezier(105.5,43)(145.875,-17.5)(179.75,43)
\qbezier(105.75,43)(130.875,12)(139.5,43)
\qbezier(139.5,43)(158.375,13.5)(179.75,43)
\put(132,83.75){\makebox(0,0)[cc]{$\gamma_5$}}
\put(135.5,70.75){\makebox(0,0)[cc]{$\gamma_6$}}
\put(130,58.5){\makebox(0,0)[cc]{$\gamma_7$}}
\put(135.75,48.25){\makebox(0,0)[cc]{$\gamma_8$}}
\put(159.75,36.25){\makebox(0,0)[cc]{$\gamma_9$}}
\put(124.75,36.5){\makebox(0,0)[cc]{$\gamma_{10}$}}
\put(143.25,23.5){\makebox(0,0)[cc]{$\gamma_{11}$}}
\put(136.5,6.5){\makebox(0,0)[cc]{$\gamma_{12}$}}
\put(138,94.5){\makebox(0,0)[cc]{$e_8$}}
\put(117.75,72.5){\makebox(0,0)[cc]{$e_9$}}
\put(99,45.5){\makebox(0,0)[cc]{$e_{10}$}}
\put(151.5,61.75){\makebox(0,0)[cc]{$e_{11}$}}
\put(111.75,46){\makebox(0,0)[cc]{$e_{12}$}}
\put(138.5,58.25){\makebox(0,0)[cc]{$e_{13}$}}
\put(127,45.5){\makebox(0,0)[cc]{$e_{14}$}}
\put(148.25,46){\makebox(0,0)[cc]{$e_{15}$}}
\put(169.75,46){\makebox(0,0)[cc]{$e_{16}$}}
\put(124.75,31.25){\makebox(0,0)[cc]{$e_{17}$}}
\put(158.25,30.75){\makebox(0,0)[cc]{$e_{18}$}}
\put(143.75,17){\makebox(0,0)[cc]{$e_{19}$}}
\put(162.5,6.25){\makebox(0,0)[cc]{$e_{20}$}}
\end{picture}
\end{center}
\caption{The diagrams for $x_0$ and $x_1x_2x_1\iv$}
\label{f:xxxx}
\end{figure}
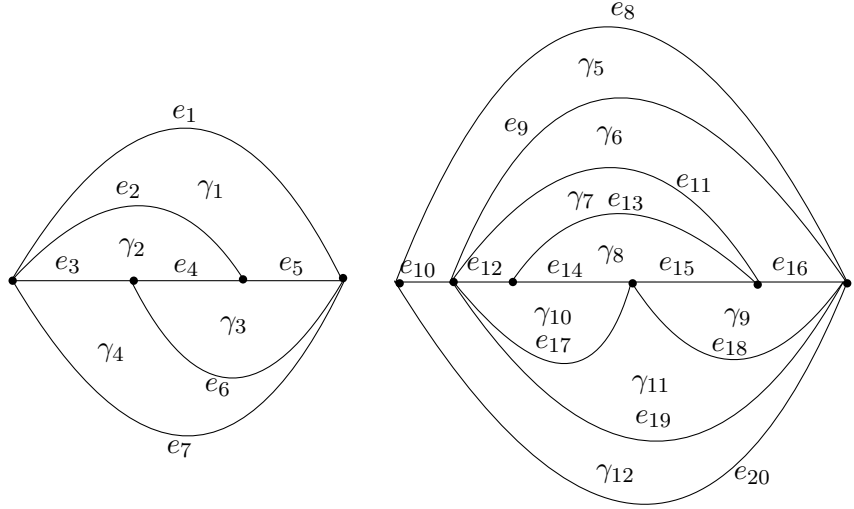

Together the two diagrams have 20 edges (labeled by $e_1,\ldots, e_{20}$) and 12 cells (labeled $\gamma_1, \ldots,  \gamma_{12}$). To construct the Stallings $2$-core $\La(H)$, we first need to identify the top and the bottom $1$-paths of both diagrams. So we set $e_1=e_7=e_8=e_{20}$. Now the positive cells $\gamma_1, \gamma_4, \gamma_5, \gamma_{12}$ need to be folded because these cells share the top $1$-path $e_1$. So we need to identify $\gamma_1=\gamma_4=\gamma_5=\gamma_{12}$ and the edges
$e_2=e_3=e_{10}$ and $e_5=e_6=e_9=e_{19}$. Now the cells $\gamma_3$, $\gamma_6$, $\gamma_{11}$ have common top edge $e_5$. So we need to fold these three cells. Thus  $\gamma_3=\gamma_6=\gamma_{11}$, $e_4=e_{11}=e_{17}$, $e_5=e_{16}=e_{18}$.  Then the cells $\gamma_7$ and $\gamma_{10}$ share the top edge $e_4$. So we set $\gamma_7=\gamma_{10}$, $e_{13}=e_{14}$. Furthermore $\gamma_9$ and $\gamma_3$ now share the top edge $e_5$. So we need to set $e_4=e_{15}$. No more foldings are needed, and the Stallings $2$-core $\La(H)$ is presented in Figure \ref{f:x1} (there the cells and edges are supposed to be identified according to their labels: all $e_1$ edges are the same, all $\gamma_7$-cells are the same, etc.).

\begin{figure}[ht!]
\begin{center}
\unitlength .6 mm 
\linethickness{0.4pt}
\ifx\plotpoint\undefined\newsavebox{\plotpoint}\fi 
\begin{picture}(181.291, 90)(0,0)
\put(21.75,43.25){\line(1,0){63.5}}
\qbezier(21.75,43.25)(56.25,101.375)(84.75,43)
\qbezier(22,43.5)(48,71.625)(66,43.25)
\qbezier(22,43.25)(57,-15.75)(85,43.25)
\qbezier(44.5,43.25)(62.875,6.25)(84.75,43.25)
\put(22,43.25){\circle*{1.581}}
\put(95.5,42.75){\circle*{1.581}}
\put(105.75,43){\circle*{1.581}}
\put(117,43){\circle*{1.581}}
\put(139.75,42.75){\circle*{1.581}}
\put(163.5,42.5){\circle*{1.581}}
\put(180.5,42.75){\circle*{1.581}}
\put(45,43.25){\circle*{1.581}}
\put(65.75,43.5){\circle*{1.581}}
\put(84.75,43.75){\circle*{1.581}}\put(59.75,60.5){\makebox(0,0)[cc]{$\gamma_1$}}
\put(44.75,49.5){\makebox(0,0)[cc]{$\gamma_2$}}
\put(64,34.75){\makebox(0,0)[cc]{$\gamma_3$}}
\put(40.75,29.5){\makebox(0,0)[cc]{$\gamma_1$}}
\put(54.75,75){\makebox(0,0)[cc]{$e_1$}}
\put(44,60.5){\makebox(0,0)[cc]{$e_2$}}
\put(32.5,46.25){\makebox(0,0)[cc]{$e_2$}}
\put(55,45.75){\makebox(0,0)[cc]{$e_4$}}
\put(75,46){\makebox(0,0)[cc]{$e_5$}}
\put(61,23){\makebox(0,0)[cc]{$e_5$}}
\put(53.75,11.25){\makebox(0,0)[cc]{$e_1$}}
\put(94.5,43){\line(1,0){86.5}}
\qbezier(94.75,43)(132.375,140)(180.5,43)
\qbezier(180.5,43)(146.5,-41.625)(94.5,43.25)
\qbezier(105,42.75)(132.25,113.25)(180.5,42.75)
\qbezier(105.25,43)(136.5,86.5)(163.75,43)
\qbezier(116.75,43)(134.75,69.375)(163.75,42.25)
\qbezier(105.5,43)(145.875,-17.5)(179.75,43)
\qbezier(105.75,43)(130.875,12)(139.5,43)
\qbezier(139.5,43)(158.375,13.5)(179.75,43)
\put(132,83.75){\makebox(0,0)[cc]{$\gamma_1$}}
\put(135.5,70.75){\makebox(0,0)[cc]{$\gamma_3$}}
\put(130,58.5){\makebox(0,0)[cc]{$\gamma_7$}}
\put(135.75,48.25){\makebox(0,0)[cc]{$\gamma_8$}}
\put(159.75,36.25){\makebox(0,0)[cc]{$\gamma_3$}}
\put(124.75,36.5){\makebox(0,0)[cc]{$\gamma_7$}}
\put(143.25,23.5){\makebox(0,0)[cc]{$\gamma_3$}}
\put(136.5,6.5){\makebox(0,0)[cc]{$\gamma_1$}}
\put(138,94.5){\makebox(0,0)[cc]{$e_1$}}
\put(117.75,72.5){\makebox(0,0)[cc]{$e_5$}}
\put(99,45.5){\makebox(0,0)[cc]{$e_2$}}
\put(151.5,61.75){\makebox(0,0)[cc]{$e_{4}$}}
\put(111.75,46){\makebox(0,0)[cc]{$e_{12}$}}
\put(138.5,58.25){\makebox(0,0)[cc]{$e_{13}$}}
\put(127,45.5){\makebox(0,0)[cc]{$e_{13}$}}
\put(148.25,46){\makebox(0,0)[cc]{$e_{4}$}}\put(169.75,46){\makebox(0,0)[cc]{$e_{5}$}}
\put(124.75,31.25){\makebox(0,0)[cc]{$e_{4}$}}
\put(158.25,30.75){\makebox(0,0)[cc]{$e_{5}$}}
\put(143.75,17){\makebox(0,0)[cc]{$e_5$}}
\put(162.5,6.25){\makebox(0,0)[cc]{$e_{1}$}}
\end{picture}
\end{center}
\caption{The Stallings $2$-core of $H=\la x_0, x_1x_2x_1\iv\ra$}
\label{f:x1}
\end{figure}

\nopagebreak[1000]

\begin{figure}[ht!]
\begin{center}
\unitlength .8mm 
\linethickness{0.4pt}
\ifx\plotpoint\undefined\newsavebox{\plotpoint}\fi 
\begin{picture}(120.791,100)(0,0)
\put(33.5,46.75){\line(1,0){86.75}}
\qbezier(33.5,47)(76.125,117)(120.25,47)
\qbezier(120.25,47)(77,-31.25)(33.75,46.5)
\qbezier(45.75,46.75)(77.75,94.125)(119.75,47)
\qbezier(119.75,47)(82.625,-.375)(46,46.75)
\qbezier(46,46.75)(76.375,73.625)(101.25,47)
\qbezier(74.75,46.5)(95.25,22.875)(118.75,46.75)
\put(33.5,46.75){\circle*{1.581}}
\put(46.75,46.75){\circle*{1.581}}
\put(74.75,46.75){\circle*{1.581}}
\put(101.25,46.5){\circle*{1.581}}
\put(120,46.75){\circle*{1.581}}
\put(77.5,85.25){\makebox(0,0)[cc]{$f_1$}}
\put(78.75,73.25){\makebox(0,0)[cc]{$f_3$}}
\put(40.5,49.75){\makebox(0,0)[cc]{$f_2$}}\put(82.75,62){\makebox(0,0)[cc]{$f_4$}}
\put(109,49.5){\makebox(0,0)[cc]{$f_5$}}
\put(61.5,50){\makebox(0,0)[cc]{$f_6$}}
\put(86.5,49.75){\makebox(0,0)[cc]{$f_7$}}
\put(99,38.5){\makebox(0,0)[cc]{$f_8$}}
\put(79.25,27){\makebox(0,0)[cc]{$f_9$}}
\put(75.75,10.5){\makebox(0,0)[cc]{$f_{10}$}}
\put(55.75,67.5){\makebox(0,0)[cc]{$\delta_1$}}
\put(70.5,64.75){\makebox(0,0)[cc]{$\delta_2$}}
\put(73.5,52.5){\makebox(0,0)[cc]{$\delta_3$}}
\put(88.75,41.75){\makebox(0,0)[cc]{$\delta_4$}}
\put(63.5,38.75){\makebox(0,0)[cc]{$\delta_5$}}
\put(57.25,22.75){\makebox(0,0)[cc]{$\delta_6$}}
\end{picture}
\end{center}
\caption{The $2$-automaton for $x_1$}
\label{f:x23}
\end{figure}
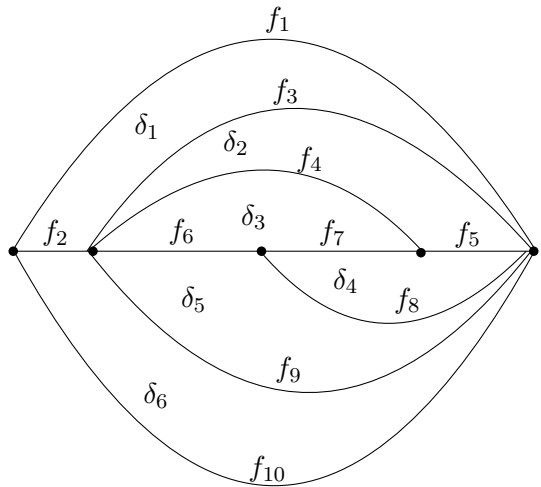

Now consider the element $x_1$. The diagram $\Delta$ for $x_1$ with labels of edges and cells is in Figure \ref{f:x23}. If $x_1\in \Cl(H)$, then we should have a morphism $\psi$ from $\Delta$ to $\La(H)$ sending $f_1$ and $f_{10}$ to $e_1$. Then $\psi(\delta_1)=\psi(\delta_6)=\gamma_1$ since $\La(H)$ has only one cell with top edge $e_1$. This forces $\psi(f_2)=e_2$, $\psi(f_3)=\psi(f_9)=e_5$. Since $\La(H)$ has only one positive cell with top edge $e_5$, we should have $\psi(\delta_2)=\gamma_3$. That means $\psi(f_4)=e_4, \psi(f_5)=e_5$. Again $\La(H)$ has only one positive cell with top edge $e_4$. Therefore $\psi(\delta_3)=\gamma_7$, hence $\psi(f_6)=e_{12}, \psi(f_7)=e_{13}$. Now $\psi$ must map the positive cell $\delta_4$ to a cell with bottom edges $\psi(f_7)=e_{13}$ and $\psi(f_5)=e_5$. But $\La(H)$ does not have such a cell, a contradiction.  Thus, $x_1\notin \Cl(H)$.

\vskip 1em

Now, let $\kk$ be a directed $2$-complex and let $\DG(\kk,u)$ be a diagram group over $\kk$. We make the following simple observation. 

\begin{Lemma}\label{diagram_group}
Let $H$ be a subgroup of the diagram group $\DG(\kk,u)$. Let $\La(H)$ be the core of $H$ and let $p=p_{\La(H)}$ be the distinguished $1$-path of $\La(H)$. The closure $\Cl(H)$ is naturally isomorphic to the diagram group $\DG(\La(H),p)$, where $\La(H)$ is viewed as a directed $2$-complex. 
\end{Lemma}

\begin{proof}
The immersion $\phi$ from $\La(H)$ to $\kk$, enables to view any diagram over $\La(H)$ as a diagram over $\kk$. Indeed, if $\Delta$ is a diagram over $\La(H)$, then every edge (resp. cell) of $\Delta$ is labeled by an edge $e$ (resp. cell $f$) of $\La(H)$. One can relabel it by the edge $\phi(e)$ (resp. the cell $\phi(f)$) of $\kk$. In particular, every diagram in $\DG(\La(H),p)$ can be viewed as a diagram over $\kk$ which is obviously accepted by $\La(H)$, hence belongs to $\Cl(H)$. We claim that this mapping $\psi$ from $\DG(\La(H),p)$ to $\Cl(H)$ is an isomorphism. It is easy to see that $\psi$ is well defined (i.e., equivalent diagrams are mapped by $\psi$ to equivalent diagrams) and that $\psi$ is a homomorphism. To prove injectivity, let $\Delta$ be a reduced diagram in $\DG(\La(H),p)$ and let $\Delta'$ be its image in $\Cl(H)$. Suppose that cells $\pi_1'$ and $\pi_2'$ form a dipole in $\Delta'$ and let $\pi_1$ and $\pi_2$ be the cells of $\Delta$ which map onto $\pi_1'$ and $\pi_2'$. Let $f_1$ and $f_2$ be the labels of the cells $\pi_1$ and $\pi_2$ in $\La(H)$. Since $\bott(\pi_1)=\topp(\pi_2)$, we have $\bott(f_1)=\topp(f_2)=\bott(f_2^{-1})$. Since $\pi_1'\circ{\pi_2'}$ is a dipole in $\Delta'$, we have $\phi(f_1)=\phi(f_2)^{-1}=\phi(f_2^{-1})$. As no foldings are applicable to $\La(H)$, the cells $f_1$ and $f_2^{-1}$ must coincide. Hence, $\pi_1$ and $\pi_2$ have mutually inverse labels and so $\pi_1\circ\pi_2$ forms a dipole in $\Delta$, in contradiction to $\Delta$ being reduced. To prove that the mapping $\psi$ is onto, we observe that if $\Delta'$ is a diagram in $\Cl(H)$, then the morphism from $\Delta'$ to $\La(H)$, enables to view it as a diagram in $\DG(\La(H),p)$, which maps to $\Delta'$ by the homomorphism $\psi$. 
\end{proof}

We say that a $2$-automaton over $\kk$ is \emph{folded} if no foldings are applicable to it.

\begin{Definition}\label{com1}
Let $\Delta$ be a diagram in $\DG(\kk,u)$. Assume that there are diagrams $\Psi,\Delta_1$ and $\Delta_2$ in the diagram groupoid $\dd(\kk)$ such that $\Psi$ is a $(vw,u)$-diagram, $\Delta_1$ is a $(v,v)$-diagram and $\Delta_2$ is a $(w,w)$-diagram. Assume also that
$\Delta\equiv\Psi^{-1}\circ(\Delta_1+\Delta_2)\circ\Psi$ (see Figure \ref{fig:sum_core}). Then the diagrams 
$$\Psi^{-1}\circ(\Delta_1+\varepsilon(w))\circ\Psi \ \mbox{ and }\ \Psi^{-1}\circ(\varepsilon(v)+\Delta_2)\circ\Psi $$
are called \emph{components} of the diagram $\Delta$.
\end{Definition}

\begin{figure}[h!]
\centering
\includegraphics[width=.35\linewidth]{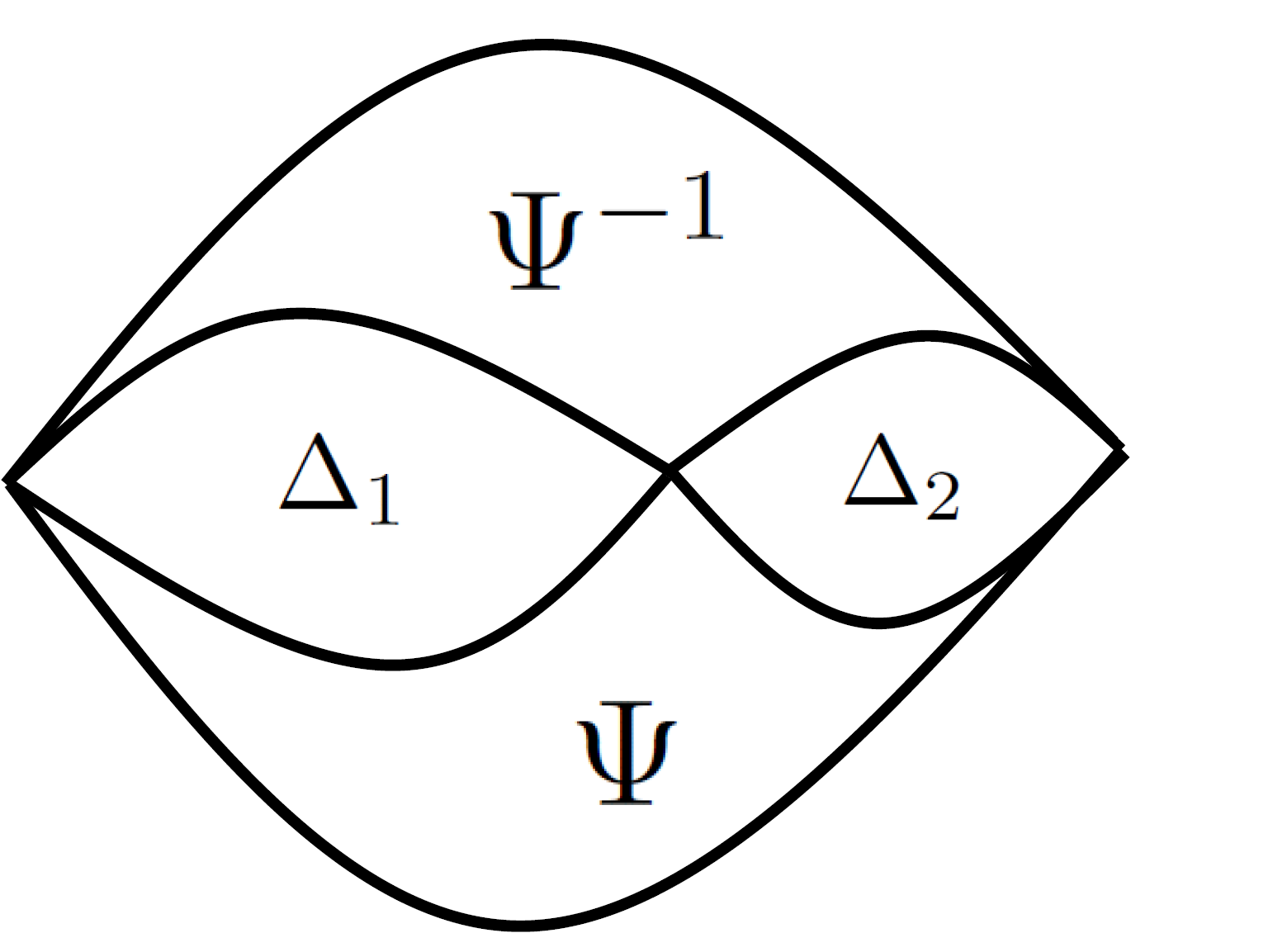}
\caption{$\Delta\equiv\Psi^{-1}\circ(\Delta_1+\Delta_2)\circ\Psi$} 
\label{fig:sum_core}
\end{figure}

Note that a diagram $\Delta$ can have more than $2$ components. We make the following simple observation. 
\begin{Remark}[Guba and Sapir]\label{r:GS}
Let $\La$ be a folded $2$-automaton over the directed $2$-complex $\kk$ with distinguished $1$-path $p_{\La}=q_{\La}$.
 Let $\Delta$ be a diagram in $\DG(\kk,u)$. If $\Delta$ is accepted by $\La$ then all components of $\Delta$ are also accepted by $\La$. In particular, if $H\le \DG(\kk,u)$, and $\Delta\in \Cl(H)$, then all components of $\Delta$ also belong to $\Cl(H)$. We say that $\Cl(H)$ is \emph{closed for components}.
\end{Remark}

The following conjecture is due to Guba and Sapir. The conjecture was made around 1999 but was never formulated in print \cite{SG_P}. 

\begin{Conjecture}[Guba and Sapir]\label{GSC}
Let $H$ be a subgroup of a diagram group $\DG(\kk,u)$. Then the closure $\Cl(H)$ is the minimal subgroup of $\DG(\kk,u)$ which contains $H$ and is closed for components. 
\end{Conjecture}

In Section \ref{sec:GSC} we prove Conjecture \ref{GSC} for subgroups of Thompson's group $F$. The conjecture for general diagram groups remains open.

\section{Paths on the core of a subgroup $H\le F$}\label{sec:paths}

Let $\kk$ be the Dunce hat. From now on, all $2$-automata $\La$ considered in this paper are $2$-automata over $\kk$, unless explicitly stated otherwise. Positive cells of $\La$ are mapped by the immersion $\phi_{\La}$ to the positive cell of the Dunce hat. 
Other than diagrams $\Delta$ in $F$ which are sometimes viewed as $2$-automata, we will only consider $2$-automata $\La$ where the distinguished $1$-paths $p_{\La}$ and $q_{\La}$ coincide and are composed of a single edge $e$ so that $e$ is mapped by $\phi_{\La}$ to the unique edge of the Dunce hat. We assume that all $2$-automata $\La$ below satisfy these properties, even if it is not mentioned explicitly. 
Since every positive cell $\pi$ in $\La$ is mapped to the positive cell of $\kk$, the top $1$-path $\topp(\pi)$ is composed of one edge and the bottom $1$-path $\bott(\pi)$ is composed of two (left and right) edges. 

We will need to distinguish between two types of foldings applicable to a $2$-automaton $\La$ over $\kk$. If two positive cells $\pi_1$ and $\pi_2$ of $\La$ share their top $1$-paths but not their bottom $1$-paths, then folding $\pi_1$ and $\pi_2$ (and their inverse cells) is considered a folding \emph{of type 1}. If $\pi_1$ and $\pi_2$ share their bottom $1$-paths, then a folding of $\pi_1$ and $\pi_2$ is a folding \emph{of type 2}.

We define paths on a $2$-automaton $\La$ in a similar way to the definition of paths on tree-diagrams over $\kk$ (see Section \ref{sec:bra}).

\begin{Definition}\label{def:path}
Let $\La$ be a $2$-automaton over $\kk$ with distinguished edge $p_{\La}=q_{\La}$.
A finite sequence of edges $e_1,\dots,e_n$ in the $2$-automaton $\La$ is said to be a \emph{path} on $\La$ if
\begin{enumerate}
\item[(1)] $e_1=p_{\La}=q_{\La}$; and 
\item[(2)] for each $i=1,\dots,n-1$, the edge $e_i$ is the top edge of some positive cell $\pi_i$ in $\La$ and $e_{i+1}$ is a bottom left or right edge of the same cell.
\end{enumerate}
The label $\lab(p)$ of a path $p=e_1,\dots,e_n$ on $\La$ is defined in the same way as the label of a path on a tree-diagram $\Psi$ (see Definition \ref{def:branch}). 
\end{Definition}

Note that if $\La$ is a $2$-automaton to which no foldings of type $1$ are applicable then every edge in $\La$ is the top edge of at most one positive cell. Then, a path $p$ on $\La$ is uniquely determined by its label. In that case, we will often abuse notation and refer to a path in terms of its label.
In particular, given a finite binary word $u$, we could refer to a path $u$ on $\La$ and to a path $u$ on a $2$-automaton $\La'$ (where no foldings of type $1$ are applicable) at the same time. We can also refer to $u$ as a (positive or negative) path on a diagram $\Delta$. This should not cause any confusion as we are careful to mention the $2$-automaton or diagram we are referring to. 
 
Note that if $\La$ is a $2$-automaton and $u$ is a finite binary word then $u$ does not necessarily label a path on $\La$. If every edge in $\La$ is the top edge of some positive cell, then every binary word $u$ labels (at least one) path on $\La$. If $p$ is a path on the $2$-automaton $\La$, then $p^+$ denotes the last edge of the path on $\La$. If no foldings of type $1$ are applicable to $\La$ and $\lab(p)\equiv u$, then $u^+$ also denotes the last edge of the path. 


The following remarks will often be used below with no specific reference.

\begin{Remark}\label{min_tree}
Let $\La$ be a $2$-automaton over $\kk$ with distinguished edge $p_{\La}=q_{\La}$. Let $v$ be a finite binary word. Then $v0$ labels a path on $\La$ if and only if $v1$ labels a path on $\La$. Thus, if a word $u$ labels a path on $\La$ and $\Psi$ is the minimal tree-diagram over $\kk$ with branch $u$ then every branch $b$ of $\Psi$  labels a path on $\La$. Indeed, every branch $b\not\equiv u$ in $\Psi$ is of the form $b\equiv va$ where $v$ is a proper prefix of $u$ and $a\in\{0,1\}$. 
\end{Remark}

\begin{Remark}\label{mor_path}
Let $\La'$ and $\La$ be $2$-automata over the Dunce hat $\kk$ such that $p_{\La'}=q_{\La'}$ and $p_{\La}=q_{\La}$ are composed of a single edge. A morphism $\psi$ from $\La'$ to $\La$ naturally sends every path on the $2$-automaton $\La'$ to a path on $\La$ with the same label. 
\end{Remark}

\begin{Lemma}\label{lif_2}
Let $\La'$ and $\La$ be $2$-automata over $\kk$ such that $\La$ results from $\La'$ by  at most countably many applications of foldings of type $2$. Then any path on $\La$ can be lifted to a unique path on $\La'$. In particular, by Remark \ref{mor_path},
there is a $1-1$ correspondence between paths on $\La'$ and paths on $\La$.
\end{Lemma}

\begin{proof}
Assume first that $\La$ results from $\La'$ by a single application of a folding of type $2$. Clearly, that induces a morphism $\psi$ from $\La'$ to $\La$, so by Remark \ref{mor_path}, any path on $\La'$ is mapped to a path on $\La$. Let $\pi_1$ and $\pi_2$ be the positive cells of $\La'$ which are folded in the transition to $\La$. In particular, in $\La'$, $\bott(\pi_1)=\bott(\pi_2)$; the cells $\pi_1$ and $\pi_2$ become identified in $\La$ (i.e., $\psi(\pi_1)=\psi(\pi_2)$) and the top edges $\topp(\pi_1)$ and $\topp(\pi_2)$ are folded to a single edge $e$ of $\La$. Let $p=e_1,\dots,e_n$ be a path on $\La$. We prove by induction on $n$ that $p$ can be lifted to a unique path $q=e_1',\dots,e_n'$ on $\La'$ such that for all $i$, $\psi(e_i')=e_i$. If $n=1$, the result is clear. Assume that the lemma holds for $n$ and let $p=e_1,\dots,e_n,e_{n+1}$. By assumption, the path $e_1,\dots,e_n$ can be lifted to a unique path $e_1',\dots,e_n'$. Notice that $e_n$ must be the top edge of some positive cell $\pi$ in $\La$ such that $e_{n+1}$ is a bottom edge of $\pi$. 
If $\pi$ is not the folded cell, i.e., $\pi$ is not the cell $\psi(\pi_1)=\psi(\pi_2)$, then there is a unique cell 
$\pi'$ in $\La'$ such that $\pi=\psi(\pi')$. Then $e_n'$ is the top edge of $\pi'$. If $e_{n+1}$ is the left (resp. right) bottom edge of $\pi$, then one should take $e_{n+1}'$ to be the left (resp. right) bottom edge of $\pi'$. That would complete the lifting of the path $p$ and it is obviously the only choice for $e_{n+1}'$. If $\pi$ is the folded cell, then 
$e_n'=\topp(\pi_1)$ or $e_n'=\topp(\pi_2)$. In that case, if $e_{n+1}$ is a left (resp. right) bottom edge of $\pi$, one should take $e_{n+1}'$ to be the common left (resp. right) bottom edge of $\pi_1$ and $\pi_2$. 

If $\La$ results from $\La'$ by finitely many foldings of type $2$ we are done by induction on the number of foldings applied. Thus, assume that $\La$ results from $\La'$ by countably many foldings of type $2$. Then there is a sequence of $2$-automata $\La'_n$, $n\ge 0$ such that $\La'_0=\La'$, $\La'_n$ results from $\La'_{n-1}$ by a single application of a folding of type $2$ and such that $\La$ is the limit $2$-automaton.
Let $p=e_1,\dots,e_n$ be a path on $\La$. Since folding is a ``local'' operation and $p$ is finite, it can be lifted to a path $q'$ on $\La'_n$ for a large enough $n$. Then $q'$ can be lifted to a path $q$ on $\La'_0=\La'$ by the case where finitely many foldings are applied. Clearly, $q$ is a lifting of $p$. To prove uniqueness, assume that $q_1$ and $q_2$ are two liftings of $p$ to paths on $\La'$. Then for a large enough $m$, the images of $q_1$ and $q_2$ in $\La'_m$ (under the natural morphism) coincide. That contradicts the fact that any path on $\La'_m$ has a unique lifting to a path on $\La'$, as $\La'_m$ results from $\La'$ by finitely many applications of foldings of type $2$. 
\end{proof}

For the proof of Lemma \ref{path_lem} below we will have to consider paths on a $2$-automaton $\La$ which do not start from the distinguished edge $p_{\La}=q_{\La}$. A sequence $p$ of edges $e_1,\dots,e_n$ in $\La$ is called a \emph{trail} if it satisfies the second condition in Definition \ref{def:path}. 
If $p_1$ and $p_2$ are trails such that the terminal edge of $p_1$ is the initial edge of $p_2$, then the concatenation of trails is naturally defined. We denote the concatenation of $p_1$ and $p_2$ by $p_1p_2$.

\begin{Lemma}\label{con_lif}
Let $\La'$ and $\La$ be $2$-automata over $\kk$ such that $\La'$ projects onto $\La$. 
 Let $p$ be a path on $\La$. Then $p$ is a concatenation of trails $p_1,\dots,p_n$ such that each trail $p_i$ can be lifted to a trail on $\La'$. \qed
\end{Lemma}

\begin{proof}
	Let $p=e_1,\dots,e_n$ be a path on $\La$. Then $p$ can be viewed as a concatenation of $n-1$ trails $p_1,\dots,p_{n-1}$ where for each $i$, $p_i=e_i,e_{i+1}$. Clearly, each trail $p_i$ can be lifted to $\La'$. 
\end{proof}

Let $H$ be a subgroup of $F$. We consider paths on the core $\La(H)$. Note that if $e$ is an edge of $\La(H)$, then there is a path $p$ on $\La(H)$ such that $p^+=e$. Indeed, this is already true for the bouquet of spheres $\La'$ defined in the construction of $\La(H)$ and $\La'$ projects onto $\La(H)$.

\begin{Lemma}\label{path_lem}
Let $H$ be a subgroup of $F$ and let $\La(H)$ be the core of $H$. Let $p$ and $q$ be two paths on the core $\La(H)$
with labels $\lab(p)\equiv u$ and $\lab(q)\equiv v$. 
Assume that $p^+=q^+$ (i.e., the paths $p$ and $q$ terminate on the same edge of $\La(H))$. Then there is an integer $k\ge 0$, such that for any finite binary word $w$ of length $\ge k$, there is an element $h\in H$ with a pair of branches $uw\rightarrow vw$. 
\end{Lemma}

\begin{proof}
We consider the construction of the core $\La(H)$. Let $\{\Delta_i\mid i\in \mathcal I\}$ be a generating set of $H$.
It is enough to consider the case where $\mathcal I$ is infinite. 
The first step in the construction of $\La(H)$ is to identify all the top and bottom edges of the generators $\Delta_i$ to get a $2$-automaton $\La'=\La'_0$. Next, we apply countably many foldings to $\La'$, so that if $\La'_n$, $n\in\mathbb{N}$ are the $2$-automata resulting in the process, then no folding is applicable to the limit automaton $\La=\La(H)$. 

It is enough to prove that the lemma holds for each of the $2$-automata $\La'_n$, $n\ge 0$. Indeed, if $p$ and $q$ are paths on the core $\La(H)$ such that $p^+=q^+$, then for a large enough $n\in\mathbb{N}$ they can be lifted to paths $p_n$, $q_n$ on $\La'_n$ such that $p_n^+=q_n^+$ and $\lab(p_n)\equiv \lab(p)$, $\lab(q_n)\equiv \lab(q)$. 

We make the following claim. For each $n\ge 0$, let $k_n$ be the number of foldings of type $2$ out of the $n$ foldings applied to $\La'$ to get the $2$-automaton $\La'_n$. Then the lemma holds for any pair of paths $p$ and $q$ on the $2$-automaton $\La'_n$ with the constant $k=k_n$. 

We prove the claim by induction on $n$. For $n=0$, let $p$ and $q$ be paths on $\La'_0$ such that $p^+=q^+=e$. Recall that $\La'_0$ is a bouquet  of the diagrams $\Delta_i$ (each, with the top and bottom edges identified). If $e$ is not an edge on the horizontal $1$-path of any of the diagrams $\Delta_i$, then the paths $p$ and $q$, and their labels $u$ and $v$, must coincide. Then for any word $w$ of length $\ge k_0=0$, the identity element of $H$ has the pair of branches $uw\rightarrow vw$. If $e$ lies on the horizontal $1$-path of some $\Delta_i$ and $p$ and $q$ do not coincide, 
then $u\rightarrow v$ is a pair of branches of the diagram $\Delta_i$ or its inverse. In particular, for every binary word $w$ of length $\ge k_0=0$, a diagram equivalent to $\Delta_i$ or $\Delta_i^{-1}$ has the pair of branches $uw\rightarrow vw$. 

Let $n\in\mathbb{N}$ and assume that the claim holds for $n-1$. We consider two cases. 

Case 1: The $n^{th}$ folding is a folding of type $1$. In that case $k_n=k_{n-1}$.

Let $p$ and $q$ be paths on $\La'_n$ such that $p^+=q^+$. By Lemma \ref{con_lif}, $p$ (resp. $q$) is a concatenation of trails $p_1,\dots,p_m$ (resp. $q_1,\dots,q_r$) which can be lifted to trails on $\La'_{n-1}$. We prove the claim by induction on $m+r$ (when $m$ and $r$ are taken to be the smallest possible for $p$ and $q$). Assume first that $m=1$ and $r=1$. 

Let $p'$ and $q'$ be liftings of $p$ and $q$ to paths on $\La'_{n-1}$. If $(p')^+=(q')^+$ then we are done by the induction hypothesis. Otherwise, the edges $(p')^+$ and $(q')^+$ of $\La'_{n-1}$ are identified in $\La'_n$ as a result of the unique folding of type $1$ applied to $\La'_{n-1}$. It follows that there are two positive cells $\pi_1$ and $\pi_2$ in $\La'_{n-1}$, such that $\topp(\pi_1)=\topp(\pi_2)$ and such that $(p')^+$ is the left or right bottom edge of $\pi_1$ and $(q')^+$ is the respective bottom edge of $\pi_2$. We assume that $(p')^+$ is the left bottom edge of $\pi_1$, the other case being similar. 
Let $z$ be a path on $\La'_{n-1}$ with terminal edge $z^+=\topp(\pi_1)=\topp(\pi_2)$ (such a path clearly exists). Then $z$ can be extended to paths $z_1$ and $z_2$ on $\La'_{n-1}$, with the same label $\lab(z_1)\equiv\lab(z_2)\equiv\lab(z)0$ such that $z_1^+=(p')^+$ and $z_2^+=(q')^+$. By the induction hypothesis, for any finite binary word $w$ of length $\ge k_{n-1}=k_n$, there are elements $h_1,h_2\in H$ with pairs of branches $\lab(z_1)w\rightarrow \lab(p')w$ and $\lab(z_2)w\rightarrow\lab(q')w$ respectively. Since $\lab(z_1)\equiv\lab(z_2)$, the element $h_1^{-1}h_2\in H$ has the pair of branches $\lab(p')w\equiv\lab(p)w\rightarrow \lab(q')w\equiv\lab(q)w$ as required. 


If $m>1$ and $p=p_1\cdots p_m$, we let $p_i'$, $i=1,\dots,m$ be the lifting of the trail $p_i$ to $\La'_{n-1}$. Let $e$ be the initial edge of $p_2'$. There is a path $p_1''$ on $\La'_{n-1}$ with terminal edge $e$. Let $s$ be the projection of $p_1''$ to a path on $\La'_n$. Then $p_1$ and $s$ terminate on the same edge. By the case $m=r=1$, we have that for any word $w'$ of length $\ge k_n$ there is an element $h_1$ in $H$ with a pair of branches $\lab(p_1)w'
\rightarrow \lab(s)w'$. 
We consider the path $sp_2p_3\cdots p_m$ on $\La'_n$. Note that $sp_2$ can be lifted to the path $p_1''p_2'$ on $\La'_{n-1}$, thus $sp_2p_3\cdots p_m$ is a concatenation of $m-1$ trails such that each one can be lifted to a trail on $\La'_{n-1}$. By the induction hypothesis, for each $w$ of length $\ge k_n$ there is an element $h_2$ in $H$ with a pair of branches $\lab(s)\lab(p_2\dots p_m)w\rightarrow \lab(q)w$. Then if one takes $w'\equiv \lab(p_2\cdots p_m)w$, then $|w'|\ge k_n$ and from the above there exists $h_1\in H$ with a pair of branches $\lab(p_1)\lab(p_2\cdots p_m)w
\rightarrow \lab(s)\lab(p_2\cdots p_m)w$. Then $h_1h_2$ is an element of $H$ with a pair of branches $\lab(p)w\rightarrow \lab(q)w$, as required. The argument for $r>1$ is similar.

Case 2: The $n^{th}$ folding is a folding of type $2$. In that case, $k_n=k_{n-1}+1$.

Let $p$ and $q$ be paths on $\La'_n$ such that $p^+=q^+$. Let $p_1$ and $q_1$ be liftings of $p$ and $q$ to paths on $\La'_{n-1}$ (see Lemma \ref{lif_2}). As in case (1), we only have to consider the case where $p_1^+$ and $q_1^+$ are distinct edges in $\La'_{n-1}$ which are identified in $\La'_n$ as a result of the folding applied to $\La'_{n-1}$. 
Since the folding is of type $2$, there are two positive cells $\pi_1$ and $\pi_2$ in $\La'_{n-1}$ such that $\bott(\pi_1)=\bott(\pi_2)$, the edge $p_1^+=\topp(\pi_1)$ and the edge $q_1^+=\topp(\pi_2)$. 
Let $w$ be a word of length $\ge k_{n}$. Then $w\equiv aw'$ where $a\in\{0,1\}$ and $|w'|\ge k_{n-1}$. The path $p_1$ can be extended to a path $p_1'$ on $\La'_{n-1}$ such that $\lab(p_1')\equiv \lab(p_1)a$ and the terminal edge $(p_1')^+$ is a bottom edge of $\pi_1$ (it is the left bottom edge if $a\equiv 0$ and the right bottom edge if $a\equiv 1$). Similarly, the path $q_1$ can be extended to a path $q_1'$ on $\La'_{n-1}$ such that $\lab(q_1')\equiv \lab(q_1)a$ and the terminal edge $(q_1')^+$ is a bottom edge of $\pi_2$. 
Clearly, in $\La'_{n-1}$, the edges $(p_1')^+$ and $(q_1')^+$ coincide. Since $|w'|\ge k_{n-1}$, by the induction hypothesis, there is an element $h\in H$ with a pair of branches $\lab(p_1')w'\rightarrow \lab(q_1')w'$. Since $\lab(p_1')w'\equiv \lab(p)w$ and $\lab(q_1')w'\equiv \lab(q)w$, $h$ has a pair of branches $\lab(p)w\rightarrow \lab(q)w$, as required. 
\end{proof}

\begin{Remark}
The proof of Lemma \ref{path_lem} implies that if $H$ is a finitely generated subgroup of $F$ then there is a uniform constant $k\in\mathbb{N}$ such that for any two paths $u$ and $v$ on the core $\La(H)$ and for any finite binary word $w$ of length $\ge k$, if $u^+=v^+$ then there is an element $h\in H$ with a pair of branches $uw\rightarrow vw$. 
\end{Remark}

The following simple lemma can be seen as a partial converse to Lemma \ref{path_lem}.

\begin{Lemma}\label{trivial}
Let $H\le F$ be a subgroup of $F$. If $h\in H$ has a pair of branches $u\rightarrow v$ such that $u$ and $v$ label paths on the core $\La(H)$, then the terminal edges $u^+$ and $v^+$ coincide. In particular, if $\Delta$ is a reduced diagram in $H$ then for any pair of branches $w_1\rightarrow w_2$ of $\Delta$, we have $w_1^+=w_2^+$ in $\La(H)$.
\end{Lemma}

\begin{proof}
Let $\Delta_1$ be the reduced diagram of $h$. Then $\Delta_1$ has a pair of branches $u_1\rightarrow v_1$ such that $u\equiv u_1w$ and $v\equiv v_1w$ for some common suffix $w$. Clearly, it suffices to prove that the terminal edges $u_1^+$ and $v_1^+$ coincide in $\La(H)$. The natural morphism from $\Delta_1$ to the core $\La(H)$, maps the branches $u_1$ and $v_1$ of $\Delta_1$ to paths $p_1$ and $q_1$ on $\La(H)$ such that $\lab(p_1)\equiv u_1$ and $\lab(q_1)\equiv v_1$. Since the terminal edges of the branches $u_1$ and $v_1$ coincide in $\Delta_1$, the terminal edges $p_1^+=u_1^+$ and $q_1^+=v_1^+$ coincide in $\La(H)$.

For the last statement of the lemma, notice that if $\Delta$ is reduced, then $w_1$ and $w_2$ must label paths on $\La(H)$ since $\Delta$ is accepted by $\La(H)$. 
\end{proof}

Lemma \ref{path_lem} implies the following. 

\begin{Corollary}\label{01}
Let $H$ be a subgroup of $F$.
Consider the finite binary words $u\equiv \emptyset$, $v\equiv 0^m$ and $w\equiv 1^n$ for $m,n\in\mathbb{N}$. Let $s$ be a finite binary word which contains both digits $0$ and $1$ and assume that $u, v, w$ and $s$ label paths on the core $\La(H)$ (the empty word always labels a path on $\La(H)$). Then the terminal edges $u^+$, $v^+$, $w^+$ and $s^+$ on $\La(H)$ are all distinct edges.
\end{Corollary}

\begin{proof}
If $u^+=v^+$, then by Lemma \ref{path_lem} there is an integer $k\in\mathbb{N}$ such that for every finite binary word $r$ of length $\ge k$, there is an element $h\in H$ with a pair of branches $ur\rightarrow vr$. Let $r\equiv 1^k$. Then there is an element $h\in H$ with a pair of branches $ur\equiv 1^k\rightarrow vr\equiv 0^m1^k$. Then $h$ maps $1=.1^{\N}$ to 
$.0^{m}1^{\mathbb N}=.0^{m-1}1\neq 1$, in contradiction to $h$ being a homeomorphism of $[0,1]$. 
The proof for the other cases is similar.
\end{proof}

Now let $H\le F$ and let $\La(H)$ be the core of $H$. 
Let $p_{\La(H)}=q_{\La(H)}$ be the distinguished $1$-path of $\La(H)$. 
We denote by $\iota(\La(H))$ the initial vertex of $p_{\La(H)}$ and by $\tau(\La(H))$ the terminal vertex of $p_{\La(H)}$. The vertex $\iota(\La(H))$ is called the \emph{initial vertex} of $\La(H)$ and $\tau(\La(H))$ is the \emph{terminal vertex} of the core $\La(H)$. Any other vertex of $\La(H)$ is an \emph{inner vertex}. An edge of $\La(H)$ is an \emph{inner edge} if both of its endpoints are inner vertices.
Note that every inner vertex in $\La(H)$ has at least one incoming and one outgoing edge. $\iota(\La(H))$ has only ougoing edges and $\tau(\La(H))$ has only incoming edges. 

As noted above, for any edge $e$ in $\La(H)$ there is a path $u$ on $\La(H)$ such that $u^+=e$. 
It is easy to see (or prove by induction) that if $e$ is incident to $\iota(\La(H))$ then $u$ can be taken to be of the form $u\equiv 0^m$ for $m\ge 0$. Conversely, if $u\equiv 0^k$ for some $k\ge 0$ then $u^+$ is incident to $\iota(\La(H))$. Similarly, if $e$ is incident to $\tau(\La(H))$ then there is a path $u\equiv 1^n$ for $n\ge 0$ such that $u^+=e$ and if $u\equiv 1^r$ for some $r\ge 0$ then $u^+$ is incident to $\tau(\La(H))$. Corollary \ref{01} implies the following. 

\begin{Corollary}\label{inner}
Let $H\le F$. Let $u$ be a finite binary word which labels a path on $\La(H)$. Then 
\begin{enumerate} 
\item[(1)] $u$ contains the digit $0$ if and only if $u^+$ is not incident to $\tau(\La(H))$. 
\item[(2)] $u$ contains the digit $1$ if and only if $u^+$ is not incident to $\iota(\La(H))$.
\item[(3)] $u$ contains both digits $0$ and $1$ if and only if $u^+$ is an inner edge of $\La(H)$. 
\end{enumerate}
\end{Corollary}

An edge of $\La(H)$ is called a \emph{boundary edge} if it is not an inner edge. If it is incident to $\iota(\La(H))$ (resp. $\tau(\La(H))$, but is not the distinguished edge of $\La(H)$, then it is a \emph{left} (resp. \emph{right}) boundary edge. 

\begin{Remark}\label{rem:tree}
Given a subgroup $H\le F$, we describe the core $\La(H)$ as a $2$-automaton by listing its edges, listing its positive cells and noting what the distinguished edge is. Moreover, 
to describe a positive cell of $\La(H)$ uniquely it is enough to note the labels 
of its top and bottom $1$-paths. 

It is often convenient to describe $\La(H)$ using a labeled binary tree $T$, where every vertex is labeled by an edge of $\La(H)$; the root is labeled by the distinguished edge $p_{\La(H)}=q_{\La(H)}$ and every caret in $T$ corresponds to a positive cell of $\La(H)$ and vice versa. For example, the core $\La(H)$ for $H=\la x_0,x_1x_2x_1^{-1}\ra$, constructed in Section \ref{sec:cor}, can be described by the following binary tree. 

\Tree[.$e_1$ [.$e_2$ [.$e_2$ ] [.$e_4$ ] ] [.$e_5$ [.$e_4$ [.$e_{12}$ ] [.$e_{13}$ [.$e_{13}$ ] [.$e_4$ ] ] ]  [.$e_5$ ] ] ]

The distinguished edge of $\La(H)$ is $e_1$, the inner edges of $\La(H)$ are $e_4$, $e_{12}$ and $e_{13}$. The edge $e_2$ is a left boundary edge and $e_5$ is a right boundary edge. We note that even though vertices are not listed in the tree we can tell from the tree that there are $2$ inner vertices in $\La(H)$: $\iota(e_5)=\iota(e_4)=\iota(e_{12})$ and $\iota(e_{13})$. The vertices of $\La(H)$ will be discussed more in detail in Section \ref{sec:tra}.
\end{Remark}

\section{The closure of a subgroup $H\le F$}\label{sec:GSC}

In this section we prove Conjecture \ref{GSC} for the closure of subgroups of Thomspon group $F$. 
First we give an equivalent definition for components of an element of $F$. 

\begin{Definition}\label{com2}
Let $f$ be a function in $F$. If $f$ fixes a finite dyadic fraction $\alpha\in (0,1)$, then the following functions $f_1,f_2\in F$ are called \emph{components} of the function $f$, or \emph{components of $f$ at $\alpha$}.  
\[
   f_1(t) =
  \begin{cases}
   f(t) &  \hbox{ if }  t\in[0,\alpha] \\
    t       & \hbox{ if } t\in[\alpha,1]\
  \end{cases}  	\qquad	
   f_2(t) =
  \begin{cases}
   t &  \hbox{ if }  t\in[0,\alpha] \\
    f(t)       & \hbox{ if } t\in[\alpha,1]\
  \end{cases} 
\]
\end{Definition}

Note that a function $f\in F$ can have more than two components. Indeed, $f$ can fix more than one finite dyadic fraction. We claim that Definition \ref{com2} is equivalent to Definition \ref{com1}. Clearly, it is enough to consider components of reduced diagrams $\Delta$ in $F$. 

\begin{Lemma}
Let $f\in F$ be an element represented by a reduced diagram $\Delta$.
Then a function $g\in F$ is a component of the function $f$, as in Definition \ref{com2}, if and only if it is represented by a diagram $\Delta'$ which is a component of $\Delta$, as in Definition \ref{com1}.
\end{Lemma}

\begin{proof}
We first consider components of the diagram $\Delta$. 
Assume that there are diagrams $\Psi,\Delta_1$ and $\Delta_2$ in the diagram groupoid $\dd(\kk)$ (where $\kk$ is the Dunce hat), such that $\Delta_1$ is a spherical $(x^n,x^n)$-diagram, $\Delta_2$ is a spherical $(x^m,x^m)$-diagram, $\Psi$ is an $(x^{n+m},x)$-diagram and such that
$$\Delta\equiv\Psi^{-1}\circ(\Delta_1+\Delta_2)\circ\Psi.$$

We consider the diagrams $\Psi^{\pm 1}$, $\Delta_1$ and $\Delta_2$ as subdiagrams of $\Delta$. 
Let $e$ be the first edge of $\topp(\Delta_2)$ (see Figure \ref{fig:components}). We claim that $e$ must be an edge of the positive subdiagram $\Delta^+$. Indeed, if $e$ is not an edge of $\Delta^+$, then $e$ must be the bottom edge of an $(x^2,x)$-cell $\pi$ which belongs to $\Psi^{-1}$. (Indeed, every edge of $\Delta$ which is not an edge of $\Delta^+$ is the bottom edge of some $(x^2,x)$-cell.)
The edge $e$ is also the $(n+1)$-edge of $\bott(\Psi^{-1})$. Let $e'$ be the corresponding edge of $\Psi$. That is, $e'$ is the $(n+1)$-edge in $\topp(\Psi)$, when $\Psi$ 
is viewed as a subdiagram of $\Delta$. Since $\Delta_1$ is an $(x^n,x^n)$-diagram, the edge $e'$ is the first edge of $\bott(\Delta_2)$. Clearly, the edge $e'$ is the top edge of an $(x,x^2)$-cell of $\Psi$, corresponding to the cell $\pi$ of $\Psi^{-1}$. Therefore, $e'$ belongs to the positive subdiagram $\Delta^+$ of $\Delta$. That contradicts the assumption that $e$ is not an edge of $\Delta^+$, as $e$ lies above $e'$. 
A similar argument shows that the first edge $e'$ of $\bott(\Delta_2)$ is an edge of the negative subdiagram $\Delta^-$.

Let $u$ be the left-most positive branch of $\Delta$ which visits the edge $e$ and let $v$ be the left-most negative branch of $\Delta$ which visits the edge $e'$. It is obvious that $u\rightarrow v$ is a pair of branches of the diagram $\Delta$. Indeed, $u^+$ and $v^+$ are the left-most edge on the horizontal $1$-path of the subdiagram $\Delta_2$ of $\Delta$. 
Notice that $u$ has a prefix $u_1$ which is a branch of $\psi^{-1}$ (with terminal edge $e$) and that $u\equiv u_10^{k_1}$ for some $k_1$. Similarly, $v$ has an initial subpath $v_1$ such that $v\equiv v_10^{k_2}$ for some $k_2$ and such that $v_1^+=e'$ (see Figure \ref{fig:components}). Since $e$ and $e'$ are corresponding edges of $\Psi^{-1}$ and $\Psi$, the labels $u_1$ and $v_1$ coincide and so $\Delta$ has a pair of branches $u_10^{k_1}\rightarrow u_10^{k_2}$.

\begin{figure}
\centering
\includegraphics[width=.45\linewidth]{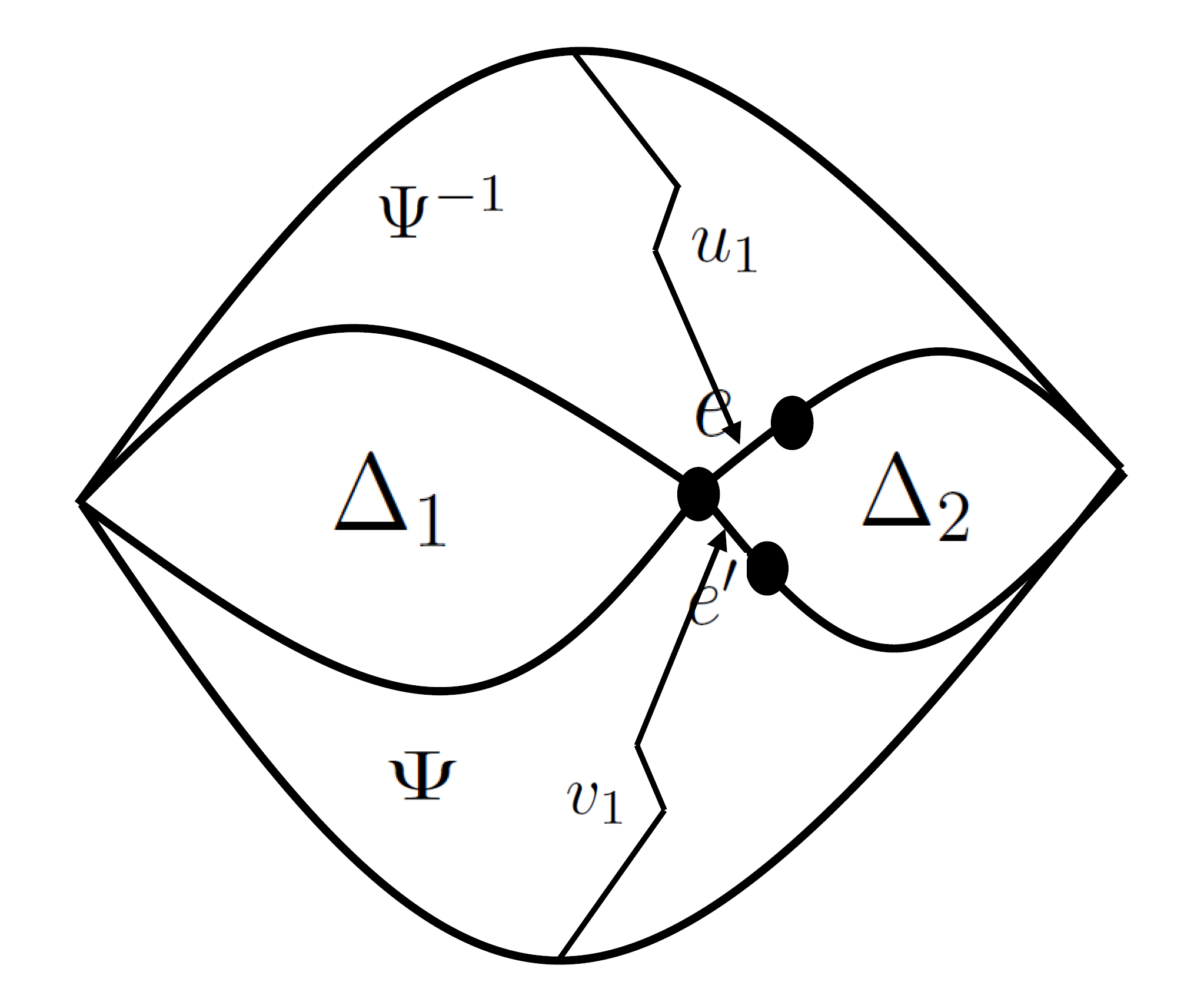}
\caption{The diagram $\Delta$ has components as in Definition \ref{com1}} 
\label{fig:components}
\end{figure}

Let $\alpha=.u_1$. Then the function $f$ fixes $\alpha$. The components 
$$\Psi^{-1}\circ(\Delta_1+\varepsilon(x^m))\circ\Psi \ \mbox{ and }\ \Psi^{-1}\circ(\varepsilon(x^n)+\Delta_2)\circ\Psi $$
of $\Delta$ correspond to the components $f_1$ and $f_2$ of $f$ at $\alpha$, respectively. Indeed, replacing $\Delta_2$ by $\varepsilon(x^m)$ does not affect the pairs of branches of $\Delta$ associated with the action of $f$ on dyadic intervals in $[0,\alpha]$. It replaces the pairs of branches of $\Delta$ associated with the action of $f$ on the interval $[\alpha,1]$ by trivial branches; i.e., branches of the form $b\rightarrow b$ for finite binary words $b$. 
Clearly, the corresponding function of $F$ is $f_1$. A similar argument works for the second component.

In the other direction, assume that a non trivial function $f\in F$ fixes a finite dyadic fraction $\alpha$. We can assume that $f$ does not fix an open neighborhood of $\alpha$. Otherwise, one can replace $\alpha$ by a dyadic fraction $\beta$ such that $f$ fixes the interval $[\alpha,\beta]$ or $[\beta,\alpha]$ and such that $f$ does not fix an open neighborhood of $\beta$ (note that the components of $f$ at $\alpha$ and at $\beta$ coincide in that case). 
We assume that $f$ does not fix a right neighborhood of $\alpha$, the argument for $f$ not fixing a left neighborhood of $\alpha$ is similar.  
Let $u$ be a finite binary word ending with $1$ such that $\alpha=.u$. 
By Lemma \ref{4parts}(3), the reduced diagram $\Delta$ representing $f$ has a pair of branches $u0^{k_1}\rightarrow u0^{k_2}$ for some $k_1,k_2\ge 0$. In particular, $u$ labels a positive and a negative path on $\Delta$. Let $\Psi'$ be the minimal tree-diagram such that $u$ is a branch of $\Psi'$ and let $\Psi\equiv \Psi'^{-1}$. Then $\Psi$ can be viewed as a subdiagram of $\Delta^-$ such that $\bott(\Psi)=\bott(\Delta^-)$ and ${\Psi}^{-1}$ can be viewed as a subdiagram of $\Delta^+$ such that $\topp({\Psi}^{-1})=\topp(\Delta^+)$. 
Let $e$ be the terminal edge of the positive path $u$ in $\Delta$ and $e'$ be the terminal edge of the negative path $u$ in $\Delta$. Clearly, $e$ lies on $\bott(\Psi^{-1})$ and $e'$ lies on $\topp(\Psi)$. The initial vertices $\iota(e)$ and $\iota(e')$ are vertices on the horizontal $1$-path of $\Delta$. The pair of branches $u0^{k_1}\rightarrow u0^{k_2}$ of $\Delta$ implies that $\iota(e)$ and $\iota(e')$ coincide. Thus, if one removes the subdiagrams $\Psi^{-1}$ and $\Psi$ from $\Delta$, the resulting diagram is a sum of two spherical diagrams $\Delta_1$ and $\Delta_2$ such that $\tau(\Delta_1)=\iota(\Delta_2)=\iota(e)$. One can show as above that the components of $\Delta$ defined by the subdiagrams $\Psi^{\pm 1}$, $\Delta_1$ and $\Delta_2$ as in Definition \ref{com1} correspond to the components of $f$ at $\alpha$. 
\end{proof}

To prove Conjecture \ref{GSC}, we prove a stronger result. Namely, that $\Cl(H)$ is generated by the subgroup $H$ and all components of functions in $H$. We will need the following definition.

\begin{Definition}\label{piecewise}
Let $H$ be a subgroup of $F$. A function $f\in F$ is said to be \emph{dyadic-piecewise-$H$} if there exist $n\in \mathbb{N}$, finite dyadic fractions $\alpha_1<\dots<\alpha_{n-1}$ in $(0,1)$ and functions $h_1,\dots,h_n\in H$ such that 
\[
   f(t) =
  \begin{cases}
   h_1(t) &  \hbox{ if }  t\in[0,\alpha_1] \\
   h_2(t)       & \hbox{ if } t\in[\alpha_1,\alpha_2]\\
		\vdots\\
   h_{n-1}(t)	& \hbox{ if } t\in[\alpha_{n-2},\alpha_{n-1}]\\	
	   h_{n}(t)	& \hbox{ if } t\in[\alpha_{n-1},1]\
\end{cases}
\]
In particular, for each $i=1,\dots,n-1$, we have $h_i(\alpha_i)=h_{i+1}(\alpha_i)$.
We say that $H$ is \emph{dyadic-piecewise-closed} if all dyadic-piecewise-$H$ functions belong to $H$. 
\end{Definition}

\begin{Remark}
Let $H$ be a subgroup of $F$. We let $\DPiec(H)$ be the set of all dyadic-piecewise-$H$ functions. Then $\DPiec(H)$ is a dyadic-piecewise-closed subgroup of $F$. 
\end{Remark}

\begin{Lemma}\label{clo_dya}
Let $H$ be a subgroup of $F$. Let $H_1$ be the subgroup of $F$ generated by all elements of $H$ together with all components of functions in $H$. Let $\bar{H}$ be the minimal subgroup of $F$ which contains $H$ and is closed for components. Then
$$H_1=\bar{H}=\DPiec(H).$$
\end{Lemma}

\begin{proof}
We let $G=\DPiec(H)$. It is obvious that $H_1\subseteq \bar{H}$. 
To see that $\bar{H}\subseteq G$ it suffices to note that $G$ is closed for components. 
Indeed, if $f\in G$ fixes a finite dyadic fraction $\alpha\in (0,1)$ then the components 
\[
   f_1(t) =
  \begin{cases}
   f(t) &  \hbox{ if }  t\in[0,\alpha] \\
    t       & \hbox{ if } t\in[\alpha,1]\
  \end{cases}  	\qquad	
   f_2(t) =
  \begin{cases}
   t &  \hbox{ if }  t\in[0,\alpha] \\
    f(t)       & \hbox{ if } t\in[\alpha,1]\
  \end{cases} 
\]
are dyadic-piecewise-$G$, as $f$ and the identity belong to $G$. Since $G$ is dyadic-piecewise-closed, $f_1,f_2\in G$. Thus, it suffices to prove that $G\subseteq H_1$. 

Let $f\in G$ be dyadic-piecewise-$H$. We prove that $f$ belongs to $H_1$ by induction on the number $n$ of pieces in $f$.
If $n=1$, then $f\in H$. If $n=2$, then there are $h_1,h_2\in H$ and a finite dyadic fraction $\alpha_1\in(0,1)$ such that $h_1(\alpha_1)=h_2(\alpha_1)$ and such that 
\[
   f(t) =
  \begin{cases}
   h_1(t) &  \hbox{ if }  t\in[0,\alpha_1] \\
   h_2(t)       & \hbox{ if } t\in[\alpha_1,1]\
	\end{cases}
\]
Since $h_1(\alpha_1)=h_2(\alpha_1)$, we have that $h_2h_1^{-1}(\alpha_1)=\alpha_1$. Since $H_1$ contains all components of elements in $H$, the function 
\[
   k(t) =
  \begin{cases}
   t &  \hbox{ if }  t\in[0,\alpha_1] \\
   h_2h_1^{-1}(t)       & \hbox{ if } t\in[\alpha_1,1]\
	\end{cases}
\]
belongs to $H_1$. 
It suffices to notice that $f=kh_1\in H_1$. 

For $n>2$, let $h_1,\dots,h_n\in H$ and $\alpha_1,\dots,\alpha_{n-1}\in (0,1)$ be finite dyadic fractions such that $h_i(\alpha_i)=h_{i+1}(\alpha_i)$ for all $i=1,\dots,n-1$ and such that 
\[
   f(t) =
  \begin{cases}
   h_1(t) &  \hbox{ if }  t\in[0,\alpha_1] \\
   h_2(t)       & \hbox{ if } t\in[\alpha_1,\alpha_2]\\
		\vdots\\
   h_{n-1}(t)	& \hbox{ if } t\in[\alpha_{n-2},\alpha_{n-1}]\\	
	 h_{n}(t)	& \hbox{ if } t\in[\alpha_{n-1},1]\
\end{cases}
\] 
By the induction hypothesis, the function
\[
   k_1(t) =
  \begin{cases}
   h_1(t) &  \hbox{ if }  t\in[0,\alpha_1] \\
   h_2(t)       & \hbox{ if } t\in[\alpha_1,\alpha_2]\\
		\vdots\\
   h_{n-1}(t)	& \hbox{ if } t\in[\alpha_{n-2},1]\\	
\end{cases}
\]
belongs to $H_1$. 

Since $h_{n-1}(\alpha_{n-1})=h_n(\alpha_{n-1})$ we have that $h_{n}h_{n-1}^{-1}(\alpha_{n-1})=\alpha_{n-1}$. Since $H_1$ contains all components of functions in $H$, the function 
\[
   k_2(t) =
  \begin{cases}
   t &  \hbox{ if }  t\in[0,\alpha_{n-1}] \\
   h_{n}h_{n-1}^{-1}(t)       & \hbox{ if } t\in[\alpha_{n-1},1]\
\end{cases}
\]
belongs to $H_1$. It suffices to notice that $f=k_2k_1$. 
\end{proof}

\begin{Theorem}\label{thm:GS}
Let $H$ be a subgroup of $F$. Then $\Cl(H)=\DPiec(H)$.
 In particular, by Lemma \ref{clo_dya}, the closure of $H$ is the minimal subgroup of $F$ which contains $H$ and is closed for components. 
\end{Theorem}

\begin{proof}
Since $\Cl(H)$ contains $H$ and is closed for components (see Remark \ref{r:GS}), by Lemma \ref{clo_dya}, $\DPiec(H)\subseteq \Cl(H)$. To prove the other direction, we show that if $f\in \Cl(H)$, then $f$ is dyadic-piecewise-$H$. 

Let $f\in \Cl(H)$ and let $\Delta$ be the reduced diagram of $f$. Let $u_i\rightarrow v_i$ for $i=1,\dots,n$ be the pairs of branches of $\Delta$. 
Since $f\in \Cl(H)$, the diagram $\Delta$ is accepted by the core $\La(H)$. It follows that for each $i=1,\dots,n$, $u_i$ and $v_i$ label paths on $\La(H)$ such that $u_i^+=v_i^+$. By Lemma \ref{path_lem}, there exists $k\in\mathbb{N}$, such that for all $i=1,\dots,n$ and each finite binary word $w$ of length $k$, there is a function $h_{i,w}\in H$ with a pair of branches $u_iw\rightarrow v_iw$. Notice that all of these pairs of branches are pairs of branches of the function $f$.
Thus, for each $i=1,\dots,n$ and every finite binary word $w\in\{0,1\}^k$, $f$ coincides with some function of $H$ on the interval $[u_iw]$. It remains to notice that the dyadic intervals $[u_iw]$ for $i=1,\dots,n$ and $w\in \{0,1\}^k$ form a dyadic subdivision of $[0,1]$.%
\end{proof}

We note that since elements of $F$ are piecewise-linear functions where all breakpoints are finite dyadic, Theorem \ref{thm:GS} 
can be formulated as follows: \emph{The closure of a subgroup $H$ of $F$ is the subgroup of $F$ of all piecewise-$H$ functions.} 
\vskip .1cm

Theorem \ref{thm:GS}  implies the following.

\begin{Corollary}\label{cor:clo}
A subgroup $H$ of $F$ is closed if and only if $H$ is closed for components. 
\end{Corollary}

\begin{Corollary}\label{orb_CH}
Let $H$ be a subgroup of $F$, then the actions of $H$ and of $\Cl(H)$ on the interval $[0,1]$ have the same orbits. 
\end{Corollary}


\section{Transitivity of the action of $H$ on the set $\mathcal D$}\label{sec:tra}

 In this section we consider the action of a subgroup $H\le F$ on the set of finite dyadic fractions $\mathcal D$. By Corollary \ref{orb_CH}, it is enough to consider the action of $\Cl(H)$ on $\mathcal D$.

We begin with the following lemma.

\begin{Lemma}\label{cor_ide}
Let $H$ be a subgroup of $F$. If $u$ and $v$ label paths on the core $\La(H)$ such that $u^+=v^+$, then there is a function $f\in \Cl(H)$ with a pair of branches $u\rightarrow v$. Moreover, $f$ is represented by a diagram $\Delta$ accepted by $\La(H)$ which has the pair of branches $u\rightarrow v$. 
\end{Lemma}

\begin{proof}
By Lemma \ref{path_lem}, there is some $k\in\mathbb{N}$ such that for each finite binary word $w\in\{0,1\}^k$, there is a function $h_{w}\in H$ with a pair of branches $uw\rightarrow vw$. Using the functions $h_w$, one can construct a dyadic-piecewise-$H$ function $f$, such that for all $w\in\{0,1\}^k$, $f$ coincides with $h_{w}$ on the interval $[uw]$. In other words, $f$ takes the branch $uw$ onto $vw$. Since that is true for all $w\in\{0,1\}^k$, $f$ has the pair of branches $u\rightarrow v$. By Theorem \ref{thm:GS}, $f\in\Cl(H)$.

 For the last statement, let $\Delta'$ be a reduced diagram of $f$. If $u\rightarrow v$ is a pair of branches of $\Delta'$, then we are done, as $\Delta'$ is accepted by $\La(H)$. Otherwise, $u\equiv u_1s$ and $v\equiv v_1s$ for some common suffix $s$, such that $u_1\rightarrow v_1$ is a pair of branches of $\Delta'$. In particular, on $\La(H)$, $u_1^+=v_1^+$.  Let $\Delta''$ be the minimal diagram of the identity with a pair of branches $s\rightarrow s$. The minimality of $\Delta''$ and the fact that $u_1s,v_1s$ label paths on $\La(H)$ guarantee that for each pair of branches $b\rightarrow b$ of $\Delta''$, $u_1b$ and $v_1b$ label paths on $\La(H)$ (see Remark \ref{min_tree}). Since $u_1^+=v_1^+$, we have $(u_1b)^+=(v_1b)^+$ in $\La(H)$ for each such $b$. 
Now, let $u_1^+=v_1^+=e$ be the edge on the horizontal $1$-path of $\Delta'$. We let $\Delta$ be the diagram resulting from $\Delta'$ by the replacement of the edge $e$ by the diagram $\Delta''$. The diagram $\Delta$ is equivalent to $\Delta'$ and as such represents $f$. By construction, it has the pair of branches $u\rightarrow v$. The above arguments show that for each pair of branches $w_1\rightarrow w_2$ of $\Delta$, $w_1$ and $w_2$ are paths on $\La(H)$ which terminate on the same edge. Hence, $\Delta$ is accepted by $\La(H)$. 
\end{proof}

\begin{Lemma}\label{path_tra}
Let $H$ be a subgroup of $F$. Let $w_1$ and $w_2$ be two finite binary words ending with $1$. 
The finite dyadic fractions $.w_1$ and $.w_2$ belong to the same orbit of the action of $H$ on $\mathcal D$ if and only if one of the following (mutually exclusive) conditions holds.
\begin{enumerate}
\item The words $w_1$ and $w_2$ do not label paths on the core $\La(H)$. In addition, if $w_1\equiv u_1s_1$ and $w_2\equiv u_2s_2$ such that $u_1$ and $u_2$ are the longest prefixes of $w_1$ and $w_2$ which label paths on $\La(H)$, then $u_1^+=u_2^+$ in $\La(H)$ and the suffixes $s_1$ and $s_2$ coincide.
\item There exist $m_1,m_2\ge 0$ such that the words $w_10^{m_1}$ and $w_20^{m_2}$ label paths on $\La(H)$ and the terminal edges $(w_10^{m_1})^+$ and $(w_20^{m_2})^+$ coincide.
\end{enumerate}
\end{Lemma}

\begin{proof}
If condition $(1)$ is satisfied,  then by Lemma \ref{cor_ide}, there is a function $f\in\Cl(H)$ with the pair of branches $u_1\rightarrow u_2$. Since $s_1\equiv s_2$, $f$ maps the fraction $.w_1=.u_1s_1$ to $.w_2=.u_2s_2$.
If condition (2) is satisfied, then by Lemma \ref{cor_ide} there is a function $f\in\Cl(H)$ with a pair of branches $w_10^{m_1}\rightarrow w_20^{m_2}$. In particular, it takes the dyadic fraction $.w_1$ to $.w_2$. Thus, if condition (1) or (2) holds, then $.w_1$ and $.w_2$ belong to the same orbit of the action of $\Cl(H)$ on $\mathcal D$, and thus, by Lemma \ref{orb_CH}, to the same orbit of the action of $H$ on $\mathcal D$. 

In the other direction, assume that $.w_1$ and $.w_2$ belong to the same orbit of $H$ and let $h\in H$ be such that $h(.w_1)=.w_2$. Let $\Delta$ be a reduced diagram of $h$. Let $u$ be the unique positive branch of $\Delta$ which is a prefix of $w_10^\N$. Let $v$ be the negative branch of $\Delta$ such that $u\rightarrow v$ is a pair of branches of $\Delta$. 
There are two cases to consider.

(1) $u$ is a strict prefix of $w_1$. In that case, let $w_1\equiv us$. Since $h(.us)=.vs=.w_2$ and $vs$ and $w_2$ both end with the digit $1$, we must have $w_2\equiv vs$. By Lemma \ref{trivial}, we have $u^+=v^+$ on $\La(H)$. Let $w$ be the longest prefix of $s$ such that $uw$ (equiv., $vw$) labels a path on $\La(H)$. Clearly, $(uw)^+=(vw)^+$ . If $w\equiv s$, then condition (2) in the lemma is satisfied with $m_1=m_2=0$. Otherwise, one can write $s\equiv w s'$. Then condition (1) is satisfied with $u_1\equiv uw$, $u_2\equiv vw$ and $s_1\equiv s_2\equiv s'$. 

(2) $u\equiv w_10^{m_1}$ for some $m_1\ge 0$. In that case, $v\equiv w_20^{m_2}$ for some $m_2\ge 0$. Otherwise, $h(.w_1)=.v\neq.w_2$.
By Lemma \ref{trivial}, we have $u^+=v^+$ in $\La(H)$. In other words, condition (2) in the lemma holds.
\end{proof}

To formulate a simple criterion for the transitivity of the action of a subgroup $H\le F$ on the set of finite dyadic fractions $\mathcal D$, we make the following definition. 

\begin{Definition}\label{Gamma}
Let $H$ be a subgroup of $F$. We define a directed graph $\Gamma(H)$ as follows. 
The set of vertices of $\Gamma(H)$ is the set of edges of the core $\La(H)$ which are not incident to $\iota(\La(H))$. For each pair of vertices $e,e'$ in $\Gamma(H)$ there is a directed edge from $e$ to $e'$ in $\Gamma(H)$ if and only if there is a positive cell $\pi$ in $\La(H)$ such that $e$ is the top edge of $\pi$ and $e'$ is the left bottom edge of $\pi$.
\end{Definition}

Note that each vertex $e$ of $\Gamma(H)$ can have at most one outgoing edge. Thus we have the following.
 
\begin{Remark}\label{derivation}
Let $e_1$ and $e_2$ be two vertices of $\Gamma(H)$ which lie in the same connected component of $\Gamma(H)$ (when $\Gamma(H)$ is considered as an unoriented graph). Then there are directed paths $p_1$ and $p_2$ in $\Gamma(H)$, with initial vertices $e_1$ and $e_2$ respectively, and the same terminal vertex. 
\end{Remark}

Let $u$ and $v$ be two finite binary words which contain the digit $1$ and label paths on $\La(H)$. Then $u^+$ and $v^+$ are vertices of $\Gamma(H)$ (see Corollary \ref{inner}). By Remark \ref{derivation} and the definition of $\Gamma(H)$, $u^+$ and $v^+$ belong to the same connected component of $\Gamma(H)$ if and only if, for some $m,n\ge 0$, $u0^m$ and $v0^n$ label paths on $\La(H)$ such that $(u0^m)^+=(v0^n)^+$. 



\begin{Theorem}\label{thm:tra}
Let $H$ be a finitely generated subgroup of $F$. Then the following assertions hold. 
\begin{enumerate}
\item[(1)] If there is an edge $e$ in $\La(H)$ which is not the top edge of any positive cell in $\La(H)$, then the action of $H$ on $\mathcal D$ has infinitely many orbits. 
\item[(2)] If every edge $e$ in $\La(H)$ is the top edge of some positive cell in $\La(H)$ then the number of orbits of the action of $H$ on $\mathcal D$ is equal to the number of connected components of $\Gamma(H)$ when $\Gamma(H)$ is viewed as an unoriented graph. 
\end{enumerate}
\end{Theorem}

\begin{proof}
To prove part (1), assume that there is an edge $e$ in $\La(H)$ which is not the top edge of any positive cell in the core. Let $u$ be a path on $\La(H)$ such that $u^+=e$. Clearly, if $u$ is a strict prefix of a finite binary word $v$ then $v$ does not label a path on $\La(H)$. We consider the infinite set of finite dyadic fractions $B=\{.u1^n: n\in\mathbb{N}\}$. Lemma \ref{path_tra} implies that any two distinct fractions in $B$ do not belong to the same orbit of $H$. Thus, the action of $H$ on $\mathcal D$ has infinitely many orbits as required. 

For part (2), assume that every edge $e$ in $\La(H)$ is the top edge of some positive cell $\pi$. Then every finite binary word $u$ labels a path on $\La(H)$. 

Let $w_1',w_2'$ be two finite binary words which contain the digit $1$. We claim that the edges ${w_1'}^+$ and ${w_2'}^+$ belong to the same connected component of $\Gamma(H)$ (when viewed as vertices in the graph) if and only if the finite dyadic fractions $\alpha_1=.w_1'$ and $\alpha_2=.w_2'$ belong to the same orbit of the action of $H$ on $\mathcal D$. That would complete the proof of the lemma. 

Let $w_1$ (resp. $w_2$) be the longest prefix of $w_1'$ (resp. $w_2'$) which terminates in the digit $1$. Then $w_1'\equiv w_10^{n_1}$ for some $n_1\ge 0$ (resp. $w_2'\equiv w_20^{n_2}$ for some $n_2\ge 0$). Thus, ${w_1'}^+$ and $w_1^+$ (resp. ${w_2'}^+$ and $w_2^+$) belong to the same connected component of $\Gamma(H)$. Note also that $\alpha_1=.w_1$ and $\alpha_2=.w_2$. Thus, we can replace $w_1'$ by $w_1$ and $w_2'$ by $w_2$. 


Since $w_1$ and $w_2$ label paths in $\La(H)$ and end with the digit $1$, by Lemma \ref{path_tra}, if $.w_1$ and $.w_2$ belong to the same orbit of $H$ then for some $m_1,m_2\ge 0$, we have $(w_10^{m_1})^+=(w_20^{m_2})^+$ in $\La(H)$. 
Notice that the vertex $(w_10^{m_1})^+$ (resp. $(w_20^{m_2})^+$) of $\Gamma(H)$ belongs to the same connected component as the vertex $w_1^+$ (resp. $w_2^+$). Thus, if $.w_1$ and $.w_2$ belong to the same orbit of $H$, then $w_1^+$ and $w_2^+$ belong to the same connected component in $\Gamma(H)$; that of the vertex $(w_10^{m_1})^+$.

In the other direction, assume that $w_1^+$ and $w_2^+$
belong to the same connected component of $\Gamma(H)$. Then by Remark \ref{derivation} and the comments succeeding it, for some $m,n\ge 0$, the edges $(w_10^m)^+$ and $(w_20^n)^+$ coincide in $\La(H)$. By Lemma \ref{cor_ide}, there is an element $k\in\Cl(H)$ with a pair of branches $w_10^m\rightarrow w_20^n$. In particular, $\alpha_1=.w_1$ and $\alpha_2=.w_2$ belong to the same orbit of the action of $\Cl(H)$, and thus of $H$, on the set of finite dyadic fractions $\mathcal D$. 
\end{proof}

If $H$ is a finitely generated subgroup of $F$ then Theorem \ref{thm:tra} gives an algorithm for deciding the transitivity of the action of $H$ on $\mathcal D$. We note that an edge $e$ of $\La(H)$ is a vertex of $\Gamma(H)$ if and only if it is an outgoing edge of some inner vertex in $\La(H)$. Condition (2) in Theorem \ref{thm:tra} can be formulated in terms of the number of inner vertices of the core $\La(H)$ using the following proposition. 

The definition of the initial vertex of $\La(H)$ extends naturally to all $2$-automata considered in the proof of the following proposition. 

\begin{Proposition}\label{inner_vertex}
Let $H$ be a subgroup of $F$ and let $e_1,e_2$ be edges of $\La(H)$ which are not incident to $\iota(\La(H))$. Then $e_1,e_2$ belong to the same connected component of $\Gamma(H)$ (when viewed as an unoriented graph), if and only if $\iota(e_1)=\iota(e_2)$ in $\La(H)$. 
\end{Proposition}

\begin{proof}
To prove that if $e_1$ and $e_2$ belong to the same connected component of $\Gamma(H)$ then $\iota(e_1)=\iota(e_2)$ it suffices to consider the case where $e_1$ and $e_2$ are adjacent vertices of $\Gamma(H)$. Without loss of generality, we can assume that there is a positive cell $\pi$ in $\La(H)$ with top edge $e_1$ and bottom edge $e_2$. Clearly, that implies that $\iota(e_1)=\iota(e_2)=\iota(\pi)$ in $\La(H)$. 
To prove the other direction, we consider the construction of the core $\La(H)$. Let $\La'_n$, $n\ge 0$ be the $2$-automata constructed in the process of constructing $\La(H)$, as defined in the proof of Lemma \ref{path_lem}. It suffices to prove the following lemma. 

\begin{Lemma}
For all $n\ge 0$, if $e_1$ and $e_2$ are edges of $\La'_n$ such that $\iota(e_1)=\iota(e_2)\neq\iota(\La'_n)$, then the images of $e_1$ and $e_2$ in $\La(H)$ (under the natural morphism) belong to the same connected component of $\Gamma(H)$. 
\end{Lemma}

\begin{proof}
To simplify notation we denote the morphism from $\La'_n$ to $\La(H)$ by $\psi$, with no reference to the index $n$. This should not cause confusion as the morphisms from $\La'_n$ to $\La(H)$ are compatible with the natural morphisms from $\La'_i$ to $\La'_j$ for $i<j$. 

We prove the lemma by induction on $n$. If $n=0$ then $\La'_0$ is a bouquet of the generating diagrams $\Delta_i$ (each, with the top and bottom edges identified). Let $e_1$ and $e_2$ be edges of $\La'_0$ such that $\iota(e_1)=\iota(e_2)\neq \iota(\La'_0)$. Then $e_1$ and $e_2$ are edges in one of the diagrams $\Delta_i$. Let $e'$ be the outgoing edge of $\iota(e_1)=\iota(e_2)$ which lies on the horizontal $1$-path of $\Delta_i$. It is easy to see that there is a directed path from $\psi(e_1)$ to $\psi(e')$ in $\Gamma(H)$. Similarly, for $\psi(e_2)$ and $\psi(e')$. Thus, $\psi(e_1)$ and $\psi(e_2)$ belong to the same connected component of $\Gamma(H)$. 

Now, let $n\in\mathbb{N}$ and assume that the claim holds for $n-1$.  We first consider the case where the $n^{th}$ folding is a folding of type $2$. In that case, no vertices are identified in the transition from $\La'_{n-1}$ to $\La'_n$. Let $e_1$ and $e_2$ be edges in $\La'_n$ such that $\iota(e_1)=\iota(e_2)\neq\iota(\La'_n)$. Then $e_1$ and $e_2$ can be lifted to edges $e_1'$ and $e_2'$ in $\La'_{n-1}$. Clearly, 
$\iota(e_1')=\iota(e_2')\neq \iota(\La'_{n-1})$. Thus, by induction we have that $\psi(e_1)=\psi(e_1')$ and $\psi(e_2')=\psi(e_2)$ belong to the same connected component of $\Gamma(H)$, as required. 

Therefore, we can assume that the $n^{th}$ folding is of type $1$. Let $\pi_1$ and $\pi_2$ be the positive cells of $\La'_{n-1}$ which are folded in the transition to $\La'_n$. 
 Let $p_1,p_2$ (resp. $q_1,q_2$) be the left and right bottom edges of $\pi_1$ (resp. $\pi_2$). Let $x=\tau(p_1)=\iota(p_2)$ and $y=\tau(q_1)=\iota(q_2)$.  Then in the transition to $\La'_n$, the vertex $x$ is identified with $y$ to give a vertex $z$ of $\La'_n$. No other vertices of $\La'_{n-1}$ get identified. 

Let $e_1$ and $e_2$ be edges of $\La'_n$ such that $\iota(e_1)=\iota(e_2)\neq\iota(\La'_n)$. We can assume that $\iota(e_1)=\iota(e_2)=z$. Otherwise, we are done by the induction hypothesis, as in the case of a folding of type $2$.  The edges $e_1$ and $e_2$ can be lifted to edges $e_1',e_2'$ in $\La'_{n-1}$. Clearly, $\iota(e_1'),\iota(e_2')\in\{x,y\}$. If $\iota(e_1')=\iota(e_2')$, we are done by the induction hypothesis. Thus, assume that $\iota(e_1')=x$ and $\iota(e_2')=y$.
Since $\iota(p_2)=x$ and $\iota(q_2)=y$, we have by induction, that $\psi(e_1')$ and $\psi(p_2)$ belong to the same connected component of $\Gamma(H)$. Similarly, $\psi(e_2')$ and $\psi(q_2)$ belong to the same connected component of $\Gamma(H)$. Since 
$\psi(p_2)=\psi(q_2)$ (indeed, $p_2$ and $q_2$ get identified already in the transition to $\La'_n$), we get that $\psi(e_1)=\psi(e_1')$ belongs  to the same connected component of $\psi(e_2)=\psi(e_2')$, as required. 
\end{proof}
\end{proof} 

Since any inner vertex in $\La(H)$ has at least one outgoing edge, the following is an immediate corollary of Theorem \ref{thm:tra} and Proposition \ref{inner_vertex}.

\begin{Corollary}\label{cor:tra}
Let $H$ be a finitely generated subgroup of $F$. Then the following assertions hold. 
\begin{enumerate}
\item[(1)] If there is an edge $e$ in $\La(H)$ which is not the top edge of any positive cell in $\La(H)$, then the action of $H$ on $\mathcal D$ has infinitely many orbits. 
\item[(2)] If every edge $e$ in $\La(H)$ is the top edge of some positive cell in $\La(H)$ then the number of orbits of the action of $H$ on $\mathcal D$ is equal to the number of inner vertices of $\La(H)$. 
\end{enumerate}
\end{Corollary}

In particular, we have the following. 

\begin{Corollary}\label{cor_tra}
Let $H\le F$. Then $H$ acts transitively on $\mathcal D$ if and only if the following assertions hold. 
\begin{enumerate}
\item[(1)] Every edge in $\La(H)$ is the top edge of some positive cell in the core. 
\item[(2)] There is a unique inner vertex in $\La(H)$. 
\end{enumerate}
\end{Corollary}

In practice, constructing the graph $\Gamma(H)$ is a simple way to count the number of inner vertices of $\La(H)$. Thus, we would often apply Theorem \ref{thm:tra}. 

\begin{Example}\label{Jones3}
The action of the subgroup $H=\la x_0x_1,x_1x_2,x_2x_3\ra$ of $F$ on the set of finite dyadic fractions $\mathcal D$ has two orbits.
\end{Example}

\begin{proof}
The core $\La(H)$ can be described by the following binary tree.

\Tree[.$e$ [.$f$ [.$f$ ] [.$m$ [.$h$ ] [.$l$ ] ] ] [.$g$ [.$h$ [.$h$ ] [.$n$ [.$l$ ] [.$h$ ] ] ] [.$k$ [.$l$ [.$l$ ] [.$m$ ] ] [.$g$ ] ] ] ]

\noindent We note that every edge of $\La(H)$ is the top edge of some positive cell. 
The graph $\Gamma(H)$ is the following. 

\begin{figure}[h]
	\centering
	\includegraphics[width=0.5\columnwidth]{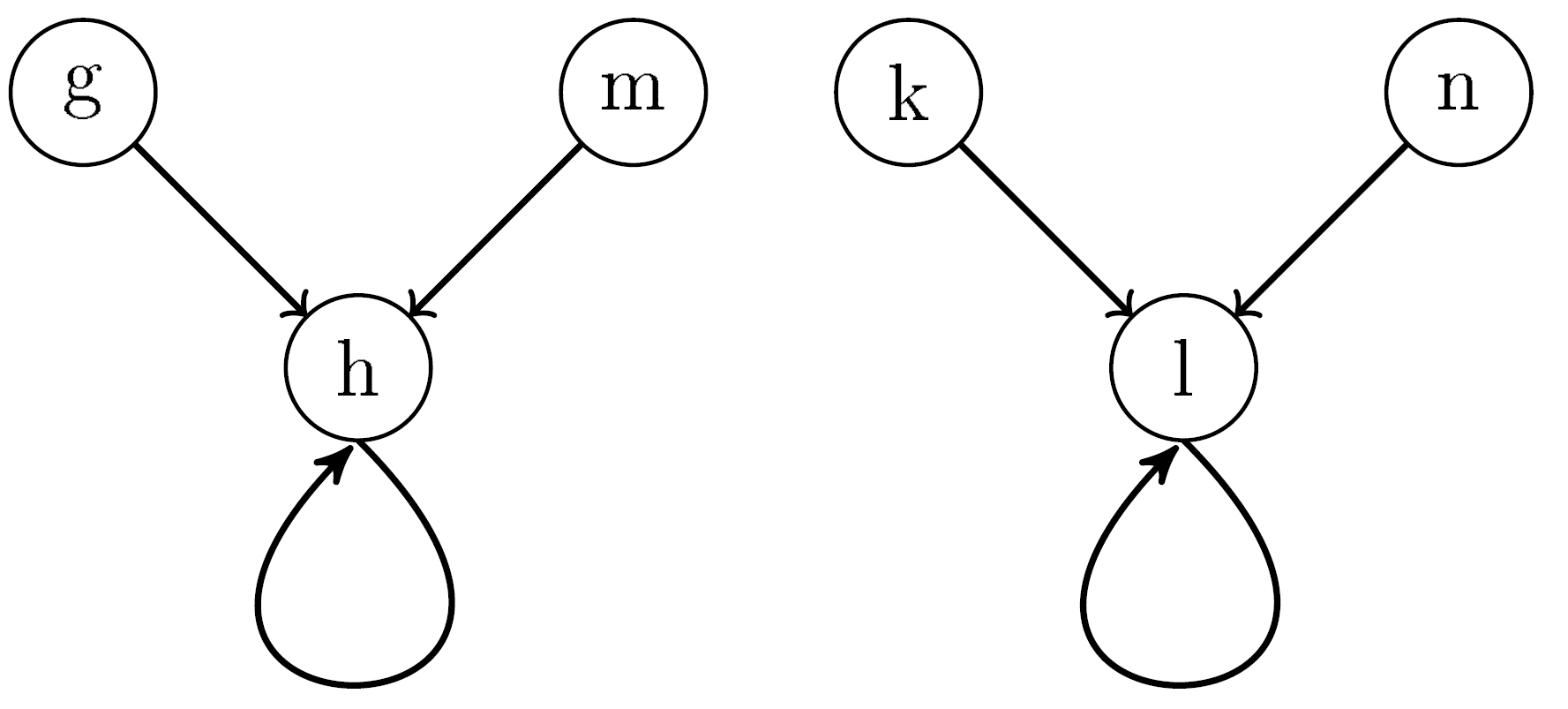}
	\caption{The graph $\Gamma(H)$}
	\label{pic1}
\end{figure}


\noindent Thus, by Theorem \ref{thm:tra}, the action of $H$ on $\mathcal D$ has two orbits. 
\end{proof}

%
%
%
%
%
%

\begin{Remark}
The subgroup $H$ from Example \ref{Jones3} is the group $\overrightarrow F$ recently defined by Jones, in his study of the relation between Thompson's group $F$ and links \cite{Jones}. It was proved in \cite{Jones} (see also \cite{GS}) that the action of $\overrightarrow F$ on $\mathcal D$ has two orbits. Two finite dyadic fractions $\alpha_1$ and $\alpha_2$ belong to the same orbit if and only if the sum of digits in their binary representation is equal modulo $2$. Considering the core $\La(\overrightarrow F)$ and the proof of Theorem \ref{thm:tra} it is easy to see that this is the case. 
\end{Remark}

%
%
%
%
%

\section{The generation problem in $F$}\label{sec:gen}

Let $X$ be a finite subset of $F$. We are interested in determining whether $X$ generates $F$. 
Let $H=\la X\ra$. We make the observation that $H=F$ if and only if (1) $H[F,F]=F$ and (2) $[F,F]\subseteq H$.
To determine if $H[F,F]=F$ it suffices to consider the image of $H$ in the abelianization of $F$. Thus, the generation problem in $F$ reduces to determining whether a finitely generated subgroup $H$ contains the derived subgroup of $F$. (Since $[F,F]$ is simple and the center of $F$ is trivial  \cite{CFP}, this is equivalent to determining if  $H$ is a normal subgroup of $F$.) We start with a condition for $\Cl(H)$ to contain $[F,F]$. 

\begin{Lemma}\label{fin_ind}
Let $H\le F$ be a subgroup of $F$. Then $\Cl(H)$ contains the derived subgroup of $F$ if and only if the following assertions hold.
\begin{enumerate}
\item[(1)] Every finite binary word $u$ labels a path on $\La(H)$.
\item[(2)] For any pair of finite binary words $u$ and $v$ which contain both digits $0$ and $1$, we have $u^+=v^+$ on $\La(H)$. 
\end{enumerate}
Equivalently, in terms of the structure of the core, $\Cl(H)$ is a normal subgroup of $F$ if and only if 
\begin{enumerate}
\item[(1')] every edge in $\La(H)$ is the top edge of some positive cell; and
\item[(2')] there is a unique inner edge in $\Cl(H)$. 
\end{enumerate}
\end{Lemma}

\begin{proof}
Condition (1') is equivalent to the condition that every finite binary word $u$ labels a path on $\La(H)$. When condition (1') holds, condition (2') is equivalent to condition (2) by Corollary \ref{inner}. 

Assume that $\Cl(H)$ contains the derived subgroup of $F$. Then every reduced diagram in $[F,F]$ is accepted by $\La(H)$. 
Recall that $[F,F]$ is the subgroup of $F$ of all functions with slope $1$ both at $0^+$ and at $1^-$. For any pair of finite binary words $u$ and $v$ which contain both digits $0$ and $1$, there is an element $f\in [F,F]$ which maps $[u]$ linearly onto $[v]$. It is easy to construct such an element so that the reduced diagram of $f$ has the pair of branches $u\rightarrow v$. It follows from Lemma \ref{trivial} that $u$ and $v$ label paths in $\La(H)$ such that $u^+=v^+$.
Thus, condition (2) is satisfied. For (1) it suffices to show that for all $n\in\mathbb{N}$, $0^n$ and $1^n$ also label paths on $\La(H)$. Since $0^n1$ and $1^n0$ are paths on the core, we have the result. 

Conversely, if (1) and (2) are satisfied, then every diagram in $[F,F]$ is accepted by the core $\La(H)$. Indeed, a diagram $\Delta$ is accepted by $\La(H)$ if and only if for any pair of branches $u\rightarrow v$ in $\Delta$, both $u$ and $v$ label paths on $\La(H)$ and $u^+=v^+$ in $\La(H)$. If $\Delta\in [F,F]$, then the first and last branches of $\Delta$ are of the form $0^m\rightarrow 0^m$ and $1^n\rightarrow 1^n$ for some $m,n\in\mathbb{N}$. All other pairs of branches $u_i\rightarrow v_i$ are such that both $u_i$ and $v_i$ contain the digits $0$ and $1$. Conditions (1) and (2) clearly imply that $\Delta$ is accepted by $\La(H)$ and thus that $[F,F]\subseteq \Cl(H)$. 
\end{proof}

\begin{Remark}\label{coreF2}
The core $\La(F)$ of Thompson's group $F$ is given by the finite binary tree

\Tree[.$e$ [.$f$ [.$f$ ] [.$h$ [.$h$ ] [.$h$ ] ] ] [.$g$ [.$h$ ] [.$g$ ] ] ]

In particular, it has a unique inner edge, a distinguished edge, a left boundary edge and a right boundary edge. It is the smallest core (in terms of the number of edges and cells) of a non-cyclic subgroup $H\le F$. 
\end{Remark}

The rest of this section is dedicated to the proof of Theorem \ref{thm:der_int}. Namely, we have to show that subgroups $H\le F$ such that $\Cl(H)$ contains the derived subgroup of $F$ and for which Condition  (2) from  Theorem \ref{thm:der_int} is satisfied, contain the derived subgroup of $F$. We begin by studying   subgroups $H\le F$ such that $\Cl(H)\supseteq [F,F]$. For that we introduce some new machinery.

Recall that by Lemma \ref{path_lem}, if $H$  is a subgroup of $F$ and $u,v$ label paths on $\La(H)$ such that $u^+=v^+$, then for any long enough extension $w$, $uw\rightarrow vw$ is a pair of branches of some element in $H$. We define $2$ additional $2$-automata related to a subgroup $H$ of $F$, where the goal is to have this property for paths on the core, with no need to consider extensions (i.e., to have that if $u^+=v^+$ on ``some core'' then  there is an element in $H$ with the pair of branches $u\rightarrow v$). We will need the following remark.

\begin{Remark}\label{fol_12}
Let $\La'$ be a $2$-automaton over the Dunce hat $\kk$ to which no foldings of type $1$ are applicable. If $\La$ results from $\La'$ by a folding of type $2$ then no folding of type $1$ is applicable to $\La$.
\end{Remark}

\begin{Definition}
Let $H$ be a subgroup of $F$ generated by a set $X=\{\Delta_i :i\in\mathcal{I}\}$ of reduced diagrams. We define a $2$-automaton over $\kk$ as follows.  First, we identify all $\topp(\Delta_i)$ with all $\bott(\Delta_i)$ and obtain a 2-automaton $\La'$ over $\kk$ with the distinguished $1$-paths $p_{\La'}=q_{\La'}=\topp(\Delta_i)=\bott(\Delta_i)$.  
Then we apply foldings of type $1$ to $\La'$ as long as possible. The resulting $2$-automaton $\La_{\sem}(X)$, to which no folding of type $1$ is applicable, is called the \emph{semi-core of $H$ associated with the generating set $X$}. 
\end{Definition}

The semi-core of $H$ associated with a generating set $X$ depends only on $X$; i.e., it does not depend on the order in which foldings were applied to $\La'$. Remark \ref{fol_12} implies that if one applies all possible foldings of type $2$ to $\La_{\sem}(X)$, the resulting $2$-automaton is the core $\La(H)$. In particular, there is a natural surjective morphism $\psi$ from $\La_{\sem}(X)$ to $\La(H)$. By Lemma \ref{lif_2}, $\psi$ gives a $1-1$ correspondence between paths on 
$\La_{\sem}(X)$ and paths on $\La(H)$. It follows that a finite binary word $u$ labels a path on $\La_{\sem}(X)$ if and only if it labels a path on the core $\La(H)$ and that each word $u$ labels at most one path on the semi-core. 

The proof of Lemma \ref{path_lem} implies the following.  

\begin{Lemma}\label{sem_lem}
Let $H$ be a subgroup of $F$ generated by a set of reduced diagrams $X$. Let $u$ and $v$ be finite binary words which label paths on the semi-core $\La_{\sem}(X)$ such that $u^+=v^+$. Then there is a function $h\in H$ with a pair of branches $u\rightarrow v$. 
\end{Lemma}

Indeed, no foldings of type $2$ were applied in the construction of $\La_{\sem}(X)$.

Let $H$ be a subgroup of $F$ generated by a set of reduced diagrams $X$. Let $\La_{\sem}(X)$ be the associated semi-core of $H$. We define a $2$-automaton $\La_{\bra}(H)$ as follows. 
For each edge $e$ in $\La_{\sem}(X)$, let $u_e$ be a path on $\La_{\sem}(X)$ such that $u_e^+=e$ (such a path clearly exists).
One can define an equivalence relation $R$ on the set of edges of $\La_{\sem}(X)$, where two edges $e_1,e_2$ are equivalent if and only if there is an element $h\in H$ with a pair of branches $u_{e_1}\rightarrow u_{e_2}$. 
Lemma \ref{sem_lem} implies that $R$ does not depend on the choice of paths $u_e$. Let $\psi$ be the morphism from $\La_{\sem}(X)$ to $\La(H)$. Lemma \ref{trivial} implies that if $e_1$ and $e_2$ are equivalent modulo $R$, then $\psi(e_1)=\psi(e_2)$. Thus, we can identify all equivalent edges of $\La_{\sem}(X)$ and $\psi$ will induce a morphism from the resulting $2$-automaton to $\La(H)$. Let $\La$ be the resulting $2$-automaton. We identify cells of $\La$ which share their top $1$-path as well as their bottom $1$-path and call the resulting $2$-automaton the \emph{branches-core} of $H$, denoted by $\La_{\bra}(H)$. Since $\La(H)$ does not contain two distinct cells with the same top and bottom $1$-paths, it is obvious that $\psi$ induces a morphism from $\La_{\bra}(H)$ to $\La(H)$. 

%
%
%

The construction of $\La_{\bra}(H)$ from $\La_{\sem}(X)$ implies 
 that if one applies all possible foldings of type $2$ to $\La_{\bra}(H)$, the result is the core $\La(H)$. Indeed, this is already true for $\La_{\sem}(X)$ and in the construction of $\La_{\bra}(H)$ from $\La_{\sem}(X)$ only edges (resp. cells) which become identified in $\La(H)$ can become identified in $\La_{\bra}(H)$. Thus, by Lemma \ref{lif_2}, there is a $1-1$ correspondence between paths on $\La_{\bra}(H)$ and paths on $\La(H)$. In particular, a finite binary word $u$ labels a path on $\La_{\bra}(H)$ if and only if it labels a path on $\La(H)$ and every finite binary word $u$ labels at most one path on $\La_{\bra}(H)$. The following is immediate from the construction and is the reason for the name of the branches-core.

\begin{Lemma}\label{bra_cor}
Let $H$ be a subgroup of $F$ and let $u$ and $v$ be paths on $\La_{\bra}(H)$. Then $u^+=v^+$ if and only if there is an element $h\in H$ with a pair of branches $u\rightarrow v$. 
\end{Lemma}

\begin{proof}
As noted above, there is a $1-1$ correspondence between paths on $\La_{\sem(X)}$ to paths on $\La(H)$ to paths on $\La_{\bra}(H)$. Let $u$ and $v$ be paths on $\La_{\bra}(H)$ such that $u^+=v^+$ in $\La_{\bra}(H)$. We consider $u$ and $v$ as paths on $\La_{\sem}(X)$. If $u^+=v^+$, then by Lemma \ref{sem_lem}, there is an element $h\in H$ with a pair of branches $u\rightarrow v$. Otherwise, let $e_1=u^+$ and $e_2=v^+$ in $\La_{\sem}(X)$. Let $u_{e_1}$ and $u_{e_2}$ be the paths on $\La_{\sem}(X)$ chosen in the construction of $\La_{\bra}(H)$, so that $u_{e_1}^+=e_1=u^+$ and $u_{e_2}^+=e_2=v^+$. The edges $e_1$ and $e_2$ being identified in $\La_{\bra}(H)$, means that there is an element $h_2\in H$ with a pair of branches $u_{e_1}\rightarrow u_{e_2}$. By Lemma \ref{sem_lem}, there is an element $h_1\in H$ with a pair of branches $u\rightarrow u_{e_1}$ and an element $h_3\in H$ with a pair of branches $u_{e_2}\rightarrow v$. Hence $h_1h_2h_3\in H$ has the pair of branches $u\rightarrow v$, as required. The proof in the other direction is similar. 
\end{proof}

Lemma \ref{bra_cor} and the $1-1$ correspondence between paths on $\La(H)$ and paths on $\La_{\bra}(H)$ imply that the branches-core of $H$ is determined uniquely by $H$. 

We note that if $H$ is generated by a finite set of reduced diagrams $X$, then one can construct the core $\La(H)$ and the semi-core of $H$ associated with $X$. We do not know how to construct the branches-core of $H$, but studying it with relation to the core $\La(H)$ is useful. 

\begin{Lemma}\label{ide_fin}
Let $H$ be a subgroup of $F$ such that $\Cl(H)$ contains the derived subgroup $[F,F]$. Let $\La_{\bra}(H)$ be the branches-core of $H$. Then there is a positive cell $\pi$ in $\La_{\bra}(H)$ such that the top edge of $\pi$ coincides with the $2$ bottom edges of $\pi$. 
\end{Lemma}

\begin{proof}
Since $\Cl(H)$ contains $[F,F]$, by Lemma \ref{fin_ind}, every finite binary word $u$ labels a path on $\La(H)$ (and thus, also on $\La_{\bra}(H)$). Thus, as noted above,
 for any word $u$ there is a unique path on $\La_{\bra}(H)$ labeled by $u$.
By Lemma \ref{lif_2}, The same is true for any $2$-automaton $\La'$ resulting from $\La_{\bra}(H)$ by applications of foldings of type $2$.

Consider the words $01$ and $010$. By Lemmas \ref{fin_ind} and \ref{path_lem}, there is $k\in\mathbb{N}$ for which there is an element in $H$ with a pair of branches $010^k\rightarrow 0100^k\equiv 010^{k+1}$. It follows that the edges $e=(010^k)^+$ and $e_1=(010^{k+1})^+$ of $\La_{\bra}(H)$ coincide. Clearly, the edge $e$ is the top edge of some positive cell $\pi$ in $\La_{\bra}(H)$. $e_1$ is the left bottom edge of $\pi$. It suffices to prove that the right bottom edge $e_2=(010^k1)^+$ coincides with $e=e_1$. 

Assume by contradiction that $e_2\neq e$. Since the images of $e_2$ and $e$ in $\La(H)$ coincide, there is a finite sequence $S$ of foldings of type $2$ such that, when applied to $\La_{\bra}(H)$, $S$ results in the identification of $e_2$ and $e$. We take $S$ to be such a minimal sequence of foldings and consider the state of the automaton $\La_{\bra}(H)$ right before the last folding in $S$ is applied to it. We denote this automaton by $\La'$ and observe that in $\La'$, $e_2\neq e$ (when we refer to edges and cells of $\La_{\bra}(H)$ as edges and cells of $\La'$, the meaning should be clear). 

Since a folding of type $2$ can be applied to $\La'$ to identify $e_2$ and $e$, it follows that in $\La'$, $e_2$ is the top edge of a positive cell $\pi'$ such that $\bott(\pi')=\bott(\pi)$. Since the bottom right edge of $\pi$ is $e_2$, the same is true for $\pi'$. In other words, the top edge of $\pi'$ coincides with its right bottom edge. 

Let $w\equiv 010^k$. We claim that for every path $u\equiv wv$ for some finite binary word $v$, the terminal edge $u^+$ in $\La'$ is either $e$ or $e_2$. If the last digit of $u$ is $0$ then $u^+=e$, otherwise, $u^+=e_2$. Indeed, this is clearly true if $v$ is empty. If $v$ is of length $n$ for $n\ge 1$ then $v\equiv v'a$ where $a$ is the last digit of $v$. By induction, $(wv')^+$ is either $e$ or $e_2$. Thus, if $a\equiv 0$ then $u^+$ is the left bottom edge of either $\pi$ or $\pi'$. In either case, $u^+=e$. Similarly, if $a\equiv 1$ we get that $u^+=e_2$.

Note that in $\La(H)$, we have $w^+=(w1)^+$. Thus, by Lemma \ref{path_lem}, for some $k'\ge 0$ there is an element $h\in H$ with a pair of branches $w1^{k'}\rightarrow w1^{k'+1}$. In particular, $h$ fixes the finite dyadic fraction $\alpha=.010^{k-1}1$ (indeed, $w\equiv 010^k$) and $h'(\alpha^-)=\frac{1}{2}$. 

Let $\Delta$ be a reduced diagram of $h$. By Lemma \ref{4parts}(4), 
 the diagram $\Delta$ has a pair of branches $w1^{n}\rightarrow w1^{n+1}$ for some $n\ge 0$. (In the notation of \ref{4parts}(4), $u\equiv 010^{k-1}1$, $u'\equiv 010^{k-1}$ and  $\ell=-1$. Thus, $\Delta$ has a pair of branches of the form $u'01^n\equiv w1^n\rightarrow u'01^{n+1}\equiv w1^{n+1}$). We claim that $n>0$. Indeed, if $n=0$ then on $\La_{\bra}(H)$, and thus on $\La'$, we have $w^+=(w1)^+$, in contradiction to the fact that on $\La'$, $w^+=e$ and $(w1)^+=e_2$. Thus, $n>0$.

 Let $u_1,\dots, u_r$ (resp. $v_1,\dots,v_m$) be the positive (resp. negative) branches of $\Delta$, with prefix $w$, ordered from left to right. Since $n>0$, we have $r,m>1$. Clearly, $u_r\equiv w1^n$ and $v_m\equiv w1^{n+1}$ so that $u_r\rightarrow v_m$ is a pair of branches of $\Delta$. 
We assume that $r\le m$ (otherwise, one can consider $\Delta^{-1}$). There is a pair of consecutive positive branches $u_i,u_{i+1}$, $1\le i<r$ such that the edges $u_i^+,u_{i+1}^+$ on the bottom $1$-path of $\Delta^+$ form the bottom $1$-path of an $(x,x^2)$-cell $\pi_1$ in $\Delta^+$. Indeed, if one deletes the prefix $w$ from each of the branches $u_1,\dots,u_r$, the result is the set of branches of a subdiagram of $\Delta^+$. 

Let $j=m-(r-i)$. Then $v_j, v_{j+1}$ are the negative branches of $\Delta$ such that $u_i\rightarrow v_j$ and $u_{i+1}\rightarrow v_{j+1}$ are pairs of branches of $\Delta$. We claim that the edges $v_j^+$ and $v_{j+1}^+$ on $\topp(\Delta^-)$ (which coincide with $u_i^+$ and $u_{i+1}^+$ as edges of $\Delta$), form the top $1$-path of an $(x^2,x)$-cell of $\Delta^-$. That will give a contradiction to the assumption that $\Delta$ is reduced.

Since $u_i^+$ and $u_{i+1}^+$ form the bottom $1$-path of a cell in $\Delta^+$, the last letter in $u_i$ is $0$ and the last letter in $u_{i+1}$ is $1$. Let $p_1$ and $p_2$ be the paths on $\La'$ labeled by $u_i$ and $u_{i+1}$ respectively. Then $p_1^+=e$ and $p_2^+=e_2$ (indeed, $w$ is a prefix of $u_i$ and $u_{i+1}$). If $q_1$ and $q_2$ are the paths on $\La'$ labeled by $v_j$ and $v_{j+1}$ respectively, then we must have $p_1^+=q_1^+$ and $p_2^+=q_2^+$ on $\La'$. Indeed, the corresponding paths on $\La_{\bra}(H)$ already have the same terminal edges, since $h\in H$ has the pairs of branches $u_i\rightarrow v_{j}$ and $u_{i+1}\rightarrow v_{j+1}$.

It follows that $q_1^+=e$ and $q_2^+=e_2$ which in turn implies that the label $v_j$ ends with $0$ and that $v_{j+1}$ ends with $1$. Two consecutive negative branches $v_j,v_{j+1}$ of a diagram satisfy this property if and only if the terminal edges $v_j^+$ and $v_{j+1}^+$ form the top $1$-path of an $(x^2,x)$-cell in the diagram. Thus, $\Delta$ is not reduced and the proof of the lemma is complete.
\end{proof}

\begin{Corollary}\label{3k}
Let $H\le F$ be a subgroup such that $\Cl(H)$ contains the derived subgroup of $F$.
Let $u$ and $v$ be finite binary words which contain both digits $0$ and $1$. Then there exists an integer $k\ge 0$ such that for each pair of finite binary words $w_1,w_2$ of length $\ge k$ there is an element $h\in H$ with a pair of branches $uw_1\rightarrow vw_2$. 
\end{Corollary}

\begin{proof}
We consider the branches-core $\La_{\bra}(H)$. By Lemma \ref{ide_fin} there is a positive cell $\pi$ in $\La_{\bra}(H)$ such that the top edge and bottom edges of $\pi$ coincide. We denote this edge by $e$. It is clear that $e$ is an inner edge of $\La_{\bra}(H)$. (Inner edges of $\La_{\bra}(H)$ are defined in a similar way to inner edges of $\La(H)$.) Let $w$ be a path on $\La_{\bra}(H)$ such that $w^+=e$. It follows that for any finite binary word $w'$, the path $ww'$ on $\La_{\bra}(H)$ terminates on the edge $e$. Thus, by Lemma \ref{bra_cor}, for any pair of binary words $w_1,w_2$ there is an element $h_{ww_1,ww_2}$ in $H$ with a pair of branches $ww_1\rightarrow ww_2$. 

Now let $u$ and $v$ be finite binary words which contain both digits $0$ and $1$. Let $p_u,p_v$ and $p_w$ be the paths on $\La(H)$ with labels $u,v$ and $w$ respectively. Since $u$, $v$ and $w$ all contain both digits $0$ and $1$, by Lemma \ref{fin_ind}, $p_u^+=p_v^+=p_w^+$ in $\La(H)$. Then by Lemma \ref{path_lem} there exists $k\ge 0$ such that for any finite binary word $s$ of length $\ge k$ there are elements $h_{us,ws}$ and $h_{ws,vs}$ in $H$ with pairs of branches $us\rightarrow ws$ and $ws\rightarrow vs$, respectively. 

We claim that the lemma holds for $k$. Indeed, let $w_1,w_2$ be a pair of binary words of length $\ge k$. 
Then the element $h_{uw_1,ww_1}h_{ww_1,ww_2}h_{ww_2,vw_2}$ is an element of $H$ with a pair of branches $uw_1\rightarrow vw_2$.
\end{proof}

Let $H$ be a subgroup of $F$ and let 
 $J$ be a closed sub-interval of $[0,1]$. We denote by $H_J$ the subgroup of $H$ of all functions which fix the interval $J$ pointwise. Recall that a subgroup $G\le F$ has an orbital $(a,b)$ if it fixes $a$ and $b$ but does not fix any point in $(a,b)$. 

\begin{Lemma}\label{H_u}
Let $H$ be a subgroup of $F$ such that $\Cl(H)$ contains the derived subgroup of $F$.
Let $u$ be a finite binary word which contains the digit $1$. Then $(0,.u)$ is an orbital of the group $H_{[u]}$. 
\end{Lemma}

\begin{proof}
Since $[F,F]\le \Cl(H)$, $H$ acts transitively on the set of finite dyadic fractions $\mathcal D$ and in particular, $H$ is not abelian (see Section \ref{sec:sg} below). Thus, there is a nontrivial element $h\in H\cap [F,F]$. Let $b$ be the minimal number in $(0,1)$ such that $h$ fixes $[b,1]$ and note that $b$ is finite dyadic. Since $H$ acts transitively on $\mathcal D$, we can assume (by passing to a conjugate of $h$ if necessary) that $b=.u$.
By the choice of $b$, the element $h$ has an orbital of the form $(a,b)$ for some $a<b$. 
We will prove that $H_{[b,1]}$ does not fix any number in $(0,b)$, so in particular $H_{[u]}$ (which contains $H_{[b,1]}$) does not fix any number in $(0,b)=(0,.u)$. 

Let $u'$ be a finite binary word such that the right endpoint of $[u']$ is $b$. Clearly, that is also true for $u'1^k$ for all $k\in\mathbb{N}$. Thus, we can assume that $[u']$ is contained in $(a,b]$. In particular, $u'$ contains both digits $0$ and $1$.  
Assume by contradiction that $H_{[b,1]}$ fixes a point $x\in (0,b)$. Clearly, $x\in (0,a]$. If $x$ is not finite dyadic, we let $\omega$ be the unique infinite binary word such that $x=.\omega$. If $x$ is finite dyadic, we let $\omega$ be the infinite binary word with a tail of zeros such that $x=.\omega$. 
Let $v$ be a prefix of $\omega$ which contains both digits $0$ and $1$. We also assume that $v$ is long enough so that $b\notin [v]$. By Corollary \ref{3k}, for some large enough $k$ there is an element $f$ in $H$ with a pair of branches $vw\rightarrow u'1^k$ where $w$ is of length $k$ and $vw$ is a prefix of $\omega$. In particular, $x\in [vw]$. By the choice of $\omega$, $x$ is not the right endpoint of $[vw]$.  Notice also that $f(b)>f(.vw1^{\mathbb{N}})=.u'1^{\mathbb{N}}=b$. 

We consider the element $fhf^{-1}$. Since $f(x)\in [u'1^k]\setminus\{.u'1^{\mathbb{N}}\}\subseteq(a,b)$, we have $h(f(x))\neq f(x)$. Indeed, $(a,b)$ is an orbital of $h$. Thus $fhf^{-1}(x)\neq x$. On the other hand, since $h$ fixes the interval $[b,1]$, the conjugate $h^{f^{-1}}$ fixes the interval 
$f^{-1}([b,1])\supseteq [b,1]$. Hence $h^{f^{-1}}\in H_{[b,1]}$, in contradiction to $x$ being a fixed point of $H_{[b,1]}$. 
\end{proof}

\begin{Lemma}\label{2}
	Let $H$ be a subgroup of $F$, such that $\Cl(H)$ contains the derived subgroup of $F$. Let $a<b$ be finite dyadic fractions and let $u'$ be a non-empty finite binary word. Then there is an element $g\in H$ such that $g([a,b])\subseteq [u']$. 
\end{Lemma}

\begin{proof}
	Let $v\equiv u'01$, so that $.v$ is a dyadic fraction in $(0,1)$. Since $\Cl(H)$ contains the derived subgroup of $F$, $H$ acts transitively on the set of finite dyadic fractions $\mathcal D$. Hence, there is an element $h\in H$ such that $h(b)=.v$. In particular, $h(a)<h(b)=.v$.
	Let $u\equiv u'1$ and note that $.v1^{\mathbb{N}}=.u$. By Lemma \ref{H_u}, the interval $(0,.u)$ is an orbital of $H_{[u]}$. Since $h(a)<h(b)=.v<.u$, there is an element $f\in H_{[u]}$ such that $f(h(a))>.v$. Note that $f(h(b))<f(.u)=.u$. Hence $[f(h(a)),f(h(b))]\subseteq [.v,.u]=[v]\subseteq [u']$. Therefore, the element $g=hf\in H$ is as required. 
\end{proof}

To prove Theorem \ref{thm:der_int} we turn to consider subgroups $H$ of $F$ which satisfy both the conditions in the theorem. Recall that the second condition is that there is an element $h\in H$ which fixes a finite dyadic fracion $\alpha$ such that $h'(\alpha^{-})=1$ and $h'(\alpha^+)=2$. We make the following observation. 

\begin{Remark}\label{new}
Let $h\in F$. Then $h$ fixes a finite dyadic fraction $\alpha\in \mathcal D$  such that $h'(\alpha^{-})=1$ and $h'(\alpha^{+})=2$ if and only if there exist non-empty finite binary words $u',u$ such that the following hold: 
\begin{enumerate}
	\item[(1)] The dyadic fraction $.u'1^{\mathbb{N}}=.u$ (and in particular $u$ contains the digit $1$).
	\item[(2)] $h$ fixes the interval $[u']=[.u',.u]$ pointwise and has the pair of branches $u0\rightarrow u$. 
\end{enumerate}
\end{Remark}

\begin{proof}
Assume that $h$ fixes a finite dyadic fraction $\alpha$ and that $h'(\alpha^-)=1$ and $h'(\alpha^+)=2$. Let $u$ be a finite binary word ending with $1$ such that $\alpha=.u$. By Lemma \ref{4parts}(3), $h$ has a pair of branches $u0^m\rightarrow u0^{m-1}$ for some $m\in\mathbb{N}$. Replacing $u$ by $u0^{m-1}$ we get that $h$ has the pair of branches $u0\rightarrow u$.
Since $\alpha\in \mathcal D$, there is a non-empty finite binary word $u'$ such that $\alpha=.u'1^{\mathbb{N}}$. Since $h$ fixes a small left neighborhood of $\alpha$, one can take $u'$ to be long enough (by replacing it by $u'1^\ell$ for some large $\ell\in\mathbb{N}$ if necessary) so that $[u']=[.u',\alpha]=[.u',.u]$ is fixed pointwise by $h$. 

In the other direction, if (1) and (2) hold then for $\alpha=.u$, the element $h$ fixes $\alpha$, $h'(\alpha^-)=1$ and $h'(\alpha^+)=2$. 
\end{proof}

 Let $H$ be a subgroup of $F$ and let $u$ and $w$ be finite binary words. We say that $w$ is \emph{$H$-equivalent to a $0$-extension of $u$} if there is an element $h\in H$ with a pair of branches $u0^k\rightarrow w$ for some $k\in\mathbb{N}$.
 
 \begin{Remark}\label{new2}
 	Let $H$ be a subgroup of $F$ such that $\Cl(H)$ contains the derived subgroup of $F$. Let $u$ be a finite binary word which contains the digit $1$. Let $w$ be a finite binary word which contains both digits $0$ and $1$. Then every long enough extension of $w$ is  $H$-equivalent to a $0$-extension of $u$. (That is, there exists $m\in\mathbb{N}$ such that for any finite binary word $w'$ of length $\ge m$, the word $ww'$ is $H$-equivalent to a $0$-extension of $u$.)
 \end{Remark}

\begin{proof}
	It follows immediately from Corollary \ref{3k} for the words $u0$ and $w$. 
\end{proof}

In the following $2$ technical lemmas the condition that $h$ fixes some finite dyadic fraction $\alpha$, has slope $1$ to the left of $\alpha$ and slope $2$ to the right of $\alpha$ is formulated in terms of the equivalent condition from Remark \ref{new}.

\begin{Lemma}\label{0}
	Let $H$ be a subgroup of $F$ such that $\Cl(H)$ contains the derived subgroup of $F$. Let $u,u',w$ be non-empty finite binary words 
	 and let $h\in H$ be such that the following hold: 
	\begin{enumerate}
		\item[(1)] The dyadic fraction $.u'1^{\mathbb{N}}=.u$.
		\item[(2)] $h$  fixes the interval $[u']=[.u',.u]$ pointwise and has the pair of branches $u0\rightarrow u$.
		\item[(3)] $w$ is $H$-equivalent to a $0$-extension of $u$. 
	\end{enumerate}
	Let $a\in (0,.w)$ and let $n,m\ge 0$.  Then there
		 is an element $g_\ell\in H$ such that $g_\ell$ is a conjugate of a power of $h$, such that $g_\ell$ has the pair of branches $w0^n\rightarrow w0^m$ and such that $g_\ell$ fixes $[a,.w]$ pointwise.

\end{Lemma}

\begin{proof}
By assumption, $w$ is $H$-equivalent to a $0$-extension of $u$. Hence, there exists $k\in\mathbb{N}$ and an element $f\in H$ such that $f$ has the pair of branches $u0^k\rightarrow w$. Let $x=f^{-1}(a)$, so that  $x<f^{-1}(.w)=.u$. By Lemma \ref{H_u}, the interval $(0,.u)$ is an orbital of $H_{[u]}$. Since $x\in (0,.u)$ and $.u'<.u$, there is an element 
$g\in H_{[u]}$ (as such, with a pair of branches $u\rightarrow u$) such that $g(x)>.u'$. 
We consider the element $q=g^{-1}f\in H$. Note that $q$ has the pair of branches $u0^{k}\rightarrow w$. Indeed, $g^{-1}$ takes $u0^{k}$ to itself and $f$ takes $u0^{k}$ to $w$. 
In addition, $q([u'])=f(g^{-1}([.u',.u]))\supseteq f([x,.u])=[a,.w]$. 

Now, recall that $h$ has the pair of branches $u0\rightarrow u$. It follows that 
$h^{n-m}$ has a pair of branches of the form $u0^{n}\rightarrow u0^{m}$. Indeed, if $n\ge m$ then $h^{n-m}$ has a pair of branches of the from $u0^{n-m}\rightarrow u$. Then one can add a common suffix $0^{m}$ and get that $h$ has the pair of branches $u0^{n}\rightarrow u0^{m}$. If $n<m$, then $h^{n-m}=(h^{m-n})^{-1}$ and the result follows from the previous case. 

To finish, let $g_{\ell}=(h^{n-m})^{q}\in H$. From the above, it follows that  $g_\ell$ has a pair of branches $w0^{n}\rightarrow w0^{m}$. Indeed, 
$q^{-1}$ has a pair of branches $w0^{n}\rightarrow u0^{k}0^{n}$,the element  $h^{n-m}$ has a pair of 
branches $u0^{n}0^{k}\rightarrow u0^{m}0^{k}$ and $q$ has a pair of branches 
$u0^{k}0^{m}\rightarrow w0^{m}$. Thus, $g_\ell$ takes the branch $w0^n$ to $w0^m$. In addition, since $h$ fixes the interval $[u']$, $g_{\ell}$ fixes the interval $q([u'])\supseteq [a,.w]$, as required. 
\end{proof}

\begin{Lemma}\label{1}
	Let $H$ be a subgroup of $F$ such that $\Cl(H)$ contains the derived subgroup of $F$. Let $u,u',w_1,w_2$ be non-empty finite binary words 
	and let $h,g\in H$ be  such that  the following hold: 
	\begin{enumerate}
		\item[(1)] The dyadic fraction $.u'1^{\mathbb{N}}=.u$.
		\item[(2)] $h$  fixes the interval $[u']=[.u',.u]$ pointwise and has the pair of branches $u0\rightarrow u$.
		\item[(3)] $w_1$ and $w_2$ are $H$-equivalent to $0$-extensions of $u$.
		\item[(4)]  $g$ has the pair of branches $w_10^{m_1}\rightarrow w_20^{m_2}$
		for some $m_1,m_2\ge 0$.
	\end{enumerate}
 	Let $a_1\in (0,.w_1)$ and let $n_1,n_2\ge 0$. Then there 
	are elements $g_{\ell}$ and $g_r$ in $H$ such that $g_\ell$ and $g_r$ are conjugates of powers of $h$ and such that the element
	$g_1=g_{\ell}gg_r$ coincides with $g$ on the interval $[a_1,.w_1]$ and has a pair of branches 
	$w_10^{n_1}\rightarrow w_20^{n_2}$.
\end{Lemma}

\begin{proof}
	Note that the conditions of Lemma \ref{0} hold for the words $u,u',w\equiv w_1$, the function $h$ and $a=a_1, n=n_1, m=m_1$. 
	Thus, by Lemma \ref{0}, there is an element $g_\ell\in H$ such that $g_\ell$ is a conjugate of a power of $h$, $g_\ell$ has the pair of branches $w_10^{n_1}\rightarrow w_10^{m_1}$ and such that $g^{\ell}$ fixes $[a_1,.w_1]$ pointwise. 
	
	Let $a_2=g(a_1)$ and note that $a_2\in (0,.w_2)$. The  conditions of Lemma \ref{0} hold for the words $u,u',w\equiv w_2$, the function $h$ and $a=a_2, n=m_2, m=n_2$. 
	Thus, by Lemma \ref{0}, there is an element $g_r\in H$ such that $g_r$ is a conjugate of a power of $h$, $g_r$ has the pair of branches $w_20^{m_2}\rightarrow w_20^{n_2}$ and such that $g^{\ell}$ fixes $[a_2,.w_2]$ pointwise.
	
	We let $g_1=g_{\ell}gg_r$. Then $g_1$ has the pair of branches $w_10^{n_1}\rightarrow w_20^{n_2}$. Since $g_{\ell}$ fixes the interval $[a_1,.w_1]$ pointwise and $g_r$ fixes the image $g([a_1,.w_1])=[a_2,.w_2]$ pointwise, the functions $g$ and $g_1$ coincide on $[a_1,.w_1]$, as necessary. 
\end{proof}

The following lemma is the key to the proof of Theorem \ref{thm:der_int}.

\begin{Lemma}\label{3}
	Let $H$ be a subgroup of $F$ such that $\Cl(H)$ contains the derived subgroup of $F$. Assume that there is an element $h\in H$ which fixes a finite dyadic fraction $\alpha\in \mathcal D$ such that $h'(\alpha^-)=1$ and 
	$h'(\alpha^+)=2$.
	Let $a<b$ be finite dyadic fractions in $(0,1)$ and let $f\in [F,F]$.
	Then there is an element $g_1\in H\cap [F,F]$ such that $g_1$ coincides with $f$ on $[a,b]$. 
\end{Lemma}

\begin{proof}
	By Remark \ref{new} there are non-empty finite binary words $u',u$ such that $.u'1^{\mathbb{N}}=.u$ and such that $h$ fixes the interval $[.u']$ and has the pair of branches $u0\rightarrow u$. 
		
	Since $f\in [F,F]$ it fixes a small neighborhood of $0$ and a small neighborhood of $1$. Thus,
	we can choose $a_1<a$ and $b_1>b$ in $(0,1)$ such that $f$ fixes the intervals $[0,a_1]$ and $[b_1,1]$.
	Let $\Delta$ be a diagram of $f$. Let $u_i\rightarrow v_i$, $i=1,\dots,n$ be the pairs of branches of $\Delta$.
	Replacing $\Delta$ by an equivalent diagram if necessary, we can assume that $a_1\notin [u_1]\cup [u_2]$ and that 
	$b_1\notin [u_{n-1}]\cup [u_n]$. In particular $u_1\equiv v_1$, $u_2\equiv v_2$, $u_{n-1}\equiv v_{n-1}$ and $u_n\equiv v_n$. 
	By Remark \ref{new2}, we can also assume that for all $i\in\{2,\dots,n-1\}$ the words $u_i$ and $v_i$ are $H$-equivalent to a $0$-extension of $u$. 
	
	We start by proving that there exists $g\in H$ such that $g$ is a product of conjugates of $h$ and has the pairs of branches $u_{i}\rightarrow v_{i}$ for $i=2,\dots,n-1$. In particular, it will coincide with $f$ on $[.u_2,.u_n]\supseteq[a,b]$.    
	
	Let $f_{1}=1$. We can construct elements $f_{2},\dots,f_{n-1}$ inductively so that for every 
	$j\in\{2,\dots,n-1\}$,
	\begin{enumerate}
		\item[(1)] $f_j$ has the pair of branches $u_j\rightarrow v_j$;
		\item[(2)] $f_j$ coincides with $f_{j-1}$ on $[.u_2,.u_j]$; and
		\item[(3)] $f_j=\ell_j f_{j-1}r_j$ where $\ell_j$ and $r_j$ belong to $H$ and are conjugates of powers of $h$.
	\end{enumerate}
	Then for $g=f_{n-1}$ we will clearly have the result.
	
	For $j=2$, since $u_2\equiv v_2$, we take $f_2=f_1=1$. Clearly, all 3 conditions are satisfied for $f_2$. 
	
	Now assume that for some $j\in \{2,\dots,n-2\}$, the function $f_j$ was constructed to satisfy the $3$ properties above. To construct $f_{j+1}$ we proceed as follows. The diagram 
	$\Delta$ and $f_j$ have the pair of branches $u_j\rightarrow v_j$. Since $u_j$ contains the digit $0$, we can let $p$ be the prefix of $u_j$ such that $u_j\equiv p01^c$ for some $c\ge 0$. Similarly, let $q$ be the prefix of $v_j$ such that $v_j\equiv q01^d$ for some
	$d\ge 0$. Then the pair of branches $u_{j+1}\rightarrow v_{j+1}$ of $\Delta$ must be of the form $u_{j+1}\equiv p10^{c_1}\rightarrow v_{j+1}\equiv q10^{d_1}$ for some $c_1,d_1\ge 0$. 
	Similarly, if $\Delta'$ is a diagram of $f_{j}$ with the pair of branches $u_j\rightarrow v_j$, then the next pair of branches in the diagram must be of the form $p10^{c_2}\rightarrow q10^{d_2}$ for some $d_1,d_2\ge 0$. Let $k=\max\{c_1,d_1\}$. Then by adding the common suffix $0^k$, we get that $f_j$  has the pair of branches $p10^{k+c_2}\rightarrow q10^{k+d_2}$; i.e., the pair of branches 
	$u_{j+1}0^{k-c_1+c_2}\rightarrow v_{j+1}0^{k-d_1+d_2}$. Applying Lemma \ref{1} with $u',u,h$, $g=f_j$, $a_1=.u_2$, $w_1\equiv u_{j+1}$, $w_2\equiv v_{j+1}$, $m_1=k-c_1+c_2$, $m_2=k-d_1+d_2$ and $n_1=n_2=0$, we get that there are elements $\ell_{j+1}$ and $r_{j+1}$ in $H$ which are conjugates of powers of $h$ such that the element $\ell_{j+1} f_j r_{j+1}$  has the pair of branches $u_{j+1}\rightarrow v_{j+1}$ and coincides with $f_j$ on the interval $[.u_2,.u_{j+1}]$. We let $f_{j+1}=\ell_{j+1} f_j r_{j+1}$.
	
	Now, let $g=f_{n-1}$ and note that $g\in H$. Since $g$ is a product of conjugates of $h$, we have that $g'(0^+)=(h'(0^+))^l$ and $g'(1^-)=(h'(1^-))^{l}$  for some $l\in\mathbb{Z}$. By Lemma \ref{2} there is an element $h_1\in H$ such that $h_1([.u_2,.u_n])\subseteq [u']$.
	
	We let $g_1=(h^{h_1^{-1}})^{-l}g$. Clearly, $g_1\in H$. We claim that $g_1\in [F,F]$ and that $g_1$ coincides with $g$ (and thus with $f$) on $[.u_2,.u_n]\supseteq [a,b]$. 
	Since $g'(0^+)=(h'(0^+))^l$, we have that $g_1'(0^+)=1$. Similarly, $g_1'(1^-)=1$. Therefore $g_1\in [F,F]$. Since $h$ fixes the interval $[u']$, $h^{h_1^{-1}}$ fixes the interval $h_1^{-1}([u'])\supseteq[.u_2,.u_n]$, thus $g_1$ coincides with $g$ on $[.u_2,.u_n]$ as required. 
\end{proof}

We are now ready to prove Theorem \ref{thm:der_int}. Recall the theorem. 

\begin{Theorem}\label{main1}
Let $H$ be a subgroup of $F$. Then $H$ contains the derived subgroup $[F,F]$ if and only if the following $2$ conditions are satisfied.
\begin{enumerate}
\item[(1)] $\Cl(H)$ contains the derived subgroup of $F$ (equivalently, $\La(H)$ satisfies the conditions in Lemma \ref{fin_ind}). 
\item[(2)] There is an element $h\in H$ which fixes a finite dyadic fraction $\alpha$ such that $h'(\alpha^-)=1$ and 
$h'(\alpha^+)=2$.
\end{enumerate}
\end{Theorem}

\begin{proof}
Clearly, if $H$ contains the derived subgroup of $F$ then conditions (1) and (2) hold. In the opposite direction, assume that $H$ satisfies conditions (1) and (2). To prove that $H$ contains the derived subgroup of $F$, we note that the derived subgroup of $F$ is the union of the derived subgroups of the subgroups $F_{[a,b]}$ over all intervals $[a,b]\subseteq (0,1)$ with dyadic endpoints. (Recall that $F_{[a,b]}$ is the subgroup of all functions in $F$ with support in $[a,b]$. Its derived subgroup is the group of all functions in $F$ with support in $(a,b)$). Let $[a,b]\subseteq (0,1)$ be an interval with dyadic endpoints. It will suffice to prove that $H$ contains the derived subgroup of $F_{[a,b]}$. Since $F_{[a,b]}$ is isomorphic to Thompson's group $F$ it is generated by two elements $y_0$ and $y_1$. Note that $y_0,y_1\in [F,F]$. 
We construct elements $h_0,h_1\in H\cap [F,F]$ such that  
\begin{enumerate}
	\item[(A)] for $j=0,1$, the element $h_j$ coincides with $y_j$ on $[a,b]$; and
	\item[(B)] the intersection of the support of $h_0$ and the support of $h_1$ is contained in $[a,b]$. 
\end{enumerate}
For $j=0$, by Lemma \ref{3}, there is an element $h_0\in H\cap [F,F]$ such that $h_0$ coincides with $y_0$ on $[a,b]$. 
Since $h_0\in [F,F]$ there are finite dyadic fractions $a_1<b_1$ in $(0,1)$ such that $[a,b]\subseteq (a_1,b_1)$ and the support of $h_0$ is contained in  $(a_1,b_1)$. We apply Lemma \ref{3} to get an element $h_{1}\in H\cap [F,F]$ which coincides with $y_1$ on $[a_1,b_1]$. Then conditions (A) and (B) are satisfied for $h_0$ and $h_1$. The conditions imply that the commutator subgroup of $\la h_0,h_1\ra$ coincides with the commutator subgroup of $\la y_0,y_1\ra=F_{[a,b]}$. Therefore, the commutator subgroup of $F_{[a,b]}$ is contained in $H$ as required. 
\end{proof}

Theorem \ref{main1} implies the following. 

\begin{Theorem}\label{main}
Let $H$ be a subgroup of $F$. Then $H=F$ if and only if the following conditions are satisfied. 
\begin{enumerate}
\item[(1)] $\Cl(H)$ contains the derived subgroup of $F$. 
\item[(2)] $H[F,F]=F$
\item[(3)] There is an element $h\in H$ which fixes a finite dyadic fraction $\alpha$ such that $h'(\alpha^-)=1$ and 
$h'(\alpha^+)=2$.
\end{enumerate}
\end{Theorem}

\begin{proof}
If $H=F$ then $H$ clearly satisfies the conditions in the theorem. In the other direction, if $H$ satisfies conditions (1),(3) then by Theorem \ref{main1}, $H$ contains the derived subgroup $[F,F]$. Thus, by condition (2), we have $H=F$.
\end{proof}

Given a finite number of elements $h_1,\dots,h_n\in F$, it is simple to check if conditions (1) and (2) in the theorem hold for the subgroup $H$ they generate. In the next section we give an algorithm, called the \emph{Tuples algorithm}, for checking if condition (3) of Theorem \ref{main} is satisfied, given that condition (1) holds. Thus, we get an algorithm solving the generation problem in Thompson's group $F$. 

\section{The Tuples algorithm}\label{sec:tuples}

Let $H\le F$ be a subgroup of $F$ generated by a finite set $X$ such that $\Cl(H)$ contains the derived subgroup of $F$. 
In this section we show that the following problem is decidable.

\begin{Problem}\label{pro}
Determine whether there exists an element $h\in H$ which has a dyadic fixed point $\alpha$ such that the slope of $h$ at $\alpha^-$ is 1 and the slope of $h$ at $\alpha^+$ is $2$. 
\end{Problem}

\begin{Lemma}\label{1221}
Let $H\le F$ be a subgroup such that $\Cl(H)$ contains the derived subgroup of $F$.
Then the following are equivalent.
\begin{enumerate}
\item There is an element $h_1\in H$ which fixes a finite dyadic fraction $\alpha_1$ such that $h_1'(\alpha_1^-)=1$ and $h_1'(\alpha_1^+)=2$. 
\item There is an element $h_2\in H$ which fixes a finite dyadic fraction $\alpha_2$ such that $h_2'(\alpha_2^-)=2$ and $h_2'(\alpha_2^+)=1$. 
\end{enumerate}
\end{Lemma}

\begin{proof}
We show that (2) implies (1). The converse implication is similar. 
By Lemma \ref{fin_ind}, $(010)^+=(01)^+$  in the core $\La(H)$. Thus, by Lemma \ref{path_lem} for a large enough $k$ there is an element $h\in H$ with a pair of branches $010^{k+1}\rightarrow 010^{k}$. In particular, $\alpha=.01$ is a fixed point of $h$ and the slope $h'(\alpha^+)=2$.

Since $\Cl(H)$ contains $[F,F]$, the action of $\Cl(H)$, and thus, of $H$, on the set of finite dyadic fractions $\mathcal D$ is transitive. Thus, there is an element $g\in H$ such that $g(\alpha_2)=\alpha$. 
We consider the element $f=h_2^g$. Since $\alpha_2$ is a fixed point of $h_2$, $\alpha$ is a fixed point of $f$. Similarly, $f'(\alpha^-)=2$ and $f'(\alpha^+)=1$. 

Since $h\in F$, the slope $h'(\alpha^-)=2^m$ for some $m\in\mathbb{Z}$. We consider the element $h_1=hf^{-m}\in H$. 
Clearly, $h_1$ fixes $\alpha$. In addition $h_1'(\alpha^-)=2^m\cdot 2^{-m}=1$ and $h_1'(\alpha^+)=2$. Thus, $h_1$ is an element satisfying condition (1) for $\alpha_1=\alpha$. 
\end{proof}

\begin{Definition}
Let $H\le F$ be a subgroup of $F$. We denote by $\mathcal S_H$ the subset of $\mathbb{Z}^2$ of all vectors $(a,b)$ such that there is an element $h\in H$ and a finite dyadic fraction $\alpha\in(0,1)$ such that $h$ fixes $\alpha$, $h'(\alpha^-)=2^a$ and $h'(\alpha^+)=2^b$.
\end{Definition}

It is obvious that if $H$ acts transitively on $\mathcal D$ then $\mathcal S_H$ is a subgroup of $\mathbb{Z}^2$. Lemma \ref{1221} implies the following.

\begin{Corollary}
Let $H\le F$ be a subgroup such that $\Cl(H)$ contains the derived subgroup of $F$, then $(0,1)\in \mathcal S_H$ if and only if $\mathcal S_H=\mathbb{Z}^2$. 
\end{Corollary}

Thus, to solve Problem \ref{pro} it suffices to determine if $\mathcal S_H=\mathbb{Z}^2$. For the rest of this section we fix a finitely generated subgroup $H\le F$ and a generating set $X=\{g_1,\dots,g_n\}$. We assume that $\Cl(H)$ contains $[F,F]$. By Lemma \ref{fin_ind}, every finite binary word $u$ labels a path on $\La(H)$. 

\begin{Definition}[The equivalence relation $R_X$]
We define an equivalence relation $R_X$ on the set of all finite binary words $\mathcal B$ (such that $\emptyset\in \mathcal B$). 
Let $\La_{\sem}(X)$ be the semi-core of $H$ associated with the generating set $X$, when diagrams in $X$ are taken in reduced form (see Section \ref{sec:gen}). 
Two finite binary words $u$ and $v$ are said to be \emph{$R_X$-equivalent} if $u^+=v^+$ in $\La_{\sem}(X)$. We write 
$u\sim_X v$ and denote the equivalence class of $u$ in 
$R_X$ by $[u]_X$.
\end{Definition}

By Lemma \ref{sem_lem}, if $u\sim_X v$ then there is an element $h\in H$ with a pair of branches $u\rightarrow v$. Note also that the number of equivalence classes in $R_X$ is finite (and computable). Indeed, it is equal to the number of edges in $\La_{\sem}(X)$. We remark that $[\emptyset]_X=\{\emptyset\}$. 

Let $\Psi$ be a tree-diagram over $\kk$. Recall that by Remark \ref{cons}, if $u_1$ and $u_2$ are consecutive branches of $\Psi$ and $u$ is the longest common prefix of $u_1$ and $u_2$, then $u_1\equiv u01^{m}\ \mbox{ and }\ u_2\equiv u10^{n}$ for some $m,n\ge 0$.

\begin{Definition}[Tuples associated with a diagram in $H$]
Let $\Delta$ be a diagram of an element in $H$. 
Let $u_1$ and $u_2$ be a pair of consecutive positive branches of $\Delta$ and $v_1$ and $v_2$ be the corresponding pair of consecutive negative branches of $\Delta$, so that $u_1\rightarrow v_1$ and $u_2\rightarrow v_2$ are pairs of branches of $\Delta$. 
Let $u$ be the longest common prefix of $u_1$ and $u_2$. By Remark \ref{cons}, $$u_1\equiv u01^{m_1}\ \mbox{ and }\ u_2\equiv u10^{n_1}\ \mbox{ for some } \ m_1,n_1\ge 0.$$ 
Let $v$ be the longest common prefix of $v_1$ and $v_2$. By Remark \ref{cons}, $$v_1\equiv v01^{m_2}\ \mbox{ and }\ v_2\equiv v10^{n_2}\ \mbox{ for some }\ m_2,n_2\ge 0.$$ 
We define the \emph{tuple associated with the consecutive pairs of branches $u_1\rightarrow v_1$ and $u_2\rightarrow v_2$ of the diagram $\Delta$} to be the tuple
$$(m_1-m_2,n_1-n_2,[u]_X\rightarrow [v]_X),$$ where $[u]_X$ and $[v]_X$ are the equivalence classes of $u$ and $v$ in $R_X$. The tuple can be viewed as an element of $\mathbb{Z}\times\mathbb{Z}\times (R_X/\sim_X \times R_X/\sim_X)$. 
\end{Definition}

Usually, we will refer to tuples as tuples associated with a diagram without mentioning the consecutive pairs of branches. 

\begin{Definition}[The groupoid $\mathcal T_H$]
We define the set $\mathcal T_H$ to be the set of all tuples associated with diagrams of elements in $H$. We define two operations on tuples in $\mathcal T_H$ as follows. 

\noindent\emph{Taking inverse}: For a tuple $t=(a,b,[u]_X\rightarrow [v]_X)$ in $\mathcal T_H$ we define the inverse tuple 
$$t^{-1}=(-a,-b,[v]_X\rightarrow [u]_X).$$

\noindent\emph{(Partial) addition}: Given two tuples $(a,b,[u]_X\rightarrow [v]_X)$ and $(c,d,[v]_X\rightarrow [w]_X)$, we let 
$$(a,b,[u]_X\rightarrow [v]_X)+(c,d,[v]_X\rightarrow [w]_X)=
(a+c,b+d,[u]_X\rightarrow[w]_X).$$
\end{Definition}

The following lemma shows that $\mathcal T_H$ is closed under the operations of taking inverses and addition. It follows easily that $\mathcal T_H$ is a groupoid. 

\begin{Lemma}\label{+}
The set $\mathcal T_H$ is closed under taking inverses and addition.
\end{Lemma}

\begin{proof}
It is obvious that $\mathcal T_H$ is closed under taking inverses. Indeed, if $t=(a,b,[u]_X\rightarrow [v]_X)\in \mathcal T_H$, then the tuple $t$ is associated with a pair of consecutive branches of a diagram $\Delta$ in $H$ . Then $t^{-1}$ is associated with the corresponding pair of consecutive branches of the diagram $\Delta^{-1}\in H$. 

Let $t_1=(a,b,[u]_X\rightarrow [v]_X)$ and $t_2=(c,d,[v]_X\rightarrow [w]_X)$ be tuples in $\mathcal T_H$. 
The tuple $t_1$ belonging to $\mathcal T_H$ implies that there is a diagram  $\Delta_1$ of an element in $H$ which has consecutive pairs of branches $$u_101^{m_1}\rightarrow v_101^{m_2}\ \mbox{ and }\ u_110^{n_1}\rightarrow v_110^{n_2}$$ such that $m_1-m_2=a$, $n_1-n_2=b$, $u_1\in [u]_X$ and $v_1\in [v]_X$.
Similarly, there is a diagram $\Delta_2$ with consecutive pairs of branches  
$$v_201^{k_1}\rightarrow w_201^{k_2}\ \mbox{ and }\ v_210^{l_1}\rightarrow w_210^{l_2},$$ where $k_1-k_2=c$, $l_1-l_2=d$, $v_2\in [v]_X$ and $w_2\in [w]_X$. 

We can assume that $m_2=k_1$ and $n_2=l_1$. Indeed, if $m_2<k_1$, we consider the edge $e$ on the horizontal $1$-path of $\Delta_1$ which is the common terminal edge of the positive branch $u_101^{m_1}$ and the negative branch $v_101^{m_2}$. We replace the edge $e$ with the diagram of the identity with branches $b\rightarrow b$ for all $b\in\{0,1\}^{k_1-m_2}$. 
The resulting diagram is equivalent to $\Delta_1$ and has consecutive pairs of branches $u_101^{m_1+k_1-m_2}\rightarrow v_101^{k_1}$ and $u_110^{n_1}\rightarrow v_110^{n_2}$. Thus, one can replace $m_1$ with $m_1+k_1-m_2$ and $m_2$ with $k_1$. In a similar way, one can treat the case where $m_2>k_1$ or $n_2\neq l_1$. Thus, we can assume that $m_2=k_1$ and $n_2=\ell_1$. 

Since $v_1,v_2\in [v]_X$, there is an element $h\in H$ with a pair of branches $v_1\rightarrow v_2$. (If $v\equiv\emptyset$, we take $h$ to be the identity.)
Let $h_1$ be the element of $H$ represented by $\Delta_1$ and let $h_2$ be the element represented by $\Delta_2$. We consider the element $g=h_1hh_2$. 
$g$ has the following consecutive pairs of branches $u_101^{m_1}\rightarrow w_201^{k_2}$ and $u_110^{n_1}\rightarrow w_210^{l_2}$. It suffices to note that $m_1-k_2=m_1-m_2+k_1-k_2=a+c$, $n_1-l_2=n_1-n_2+l_1-l_2=b+d$, $u_1\in [u]_X$ and $w_2\in [w]_X$. Thus, the tuple 
$$t_1+t_2=(a+c,b+d,[u]_X\rightarrow[w]_X)\in\mathcal T_H.$$
\end{proof}

For a finite binary word $u$, we let $\mathcal T_H([u]_X)$ be the set of all tuples in $\mathcal T_H$ such that the last coordinate is of the form $[u]_X\rightarrow [u]_X$. Such tuples are called \emph{spherical tuples}. 
Clearly, for each $u$, $\mathcal T_H([u]_X)$ is a commutative group with neutral element $(0,0,[u]_X\rightarrow[u]_X)$.

We let $\Psi\colon \mathcal T_H\to \mathbb{Z}^2$ be the natural homomorphism such that 
$$\Psi((a,b,[u]_X\rightarrow [v]_X))=(a,b).$$
Under this homomorphism, each group $\mathcal T_H([u]_X)$ embeds into $\mathbb{Z}^2$. 

\begin{Lemma}\label{conj}
Let $u$ and $v$ be finite binary words. Then the groups $\mathcal T_H([u]_X)$ and $\mathcal T_H([v]_X)$ are conjugate in the groupoid $\mathcal T_H$.
\end{Lemma}

\begin{proof}
We consider the finite binary words $u'\equiv u01$ and $v'\equiv v01$. Since $u'$ and $v'$ contain both digits $0$ and $1$ and $\Cl(H)$ contains $[F,F]$, $(u')^+$ and $(v')^+$ coincide in $\La(H)$ (see Lemma \ref{fin_ind}). By Lemma \ref{path_lem} for $k$ large enough there is an element $h\in H$ with a pair of branches $u'1^k\equiv u01^{k+1}\rightarrow v'1^k\equiv v01^{k+1}$. Let $\Delta$ be a diagram of $h$ with this pair of branches. The following pair must be of the form $u10^{m_1}\rightarrow v10^{m_2}$ for some $m_1,m_2\ge 0$. Thus, the tuple $$t=((k+1)-(k+1),m_1-m_2,[u]_X\rightarrow[v]_X)=(0,m,[u]_X\rightarrow[v]_X)\in \mathcal T_H$$ 
for $m=m_1-m_2$. Conjugating $\mathcal T_H([u]_X)$ by the tuple $t$ gives the group $\mathcal T_H([v]_X)$.
\end{proof}

\begin{Corollary}\label{coincide}
Let $u$ and $v$ be finite binary words. Then $$\Psi(\mathcal T_H([u]_X))=\Psi(\mathcal T_H([v]_X)).$$
\end{Corollary}

\begin{Lemma}\label{S_H}
Let $u$ be a finite binary word. Then $$\Psi(\mathcal T_H([u]_X))=\mathcal S_H.$$ 
\end{Lemma}

\begin{proof}
Recall that $\mathcal S_H$ is the subgroup of $\mathbb{Z}^2$ of all vectors $(a,b)$ for which there is an element $h\in H$ which fixes a finite dyadic fraction $\alpha$ such that $h'(\alpha^-)=2^a$ and $h'(\alpha^+)=2^b$. 

Let $t=(a,b,[u]_X\rightarrow[u]_X)\in \mathcal T_H([u]_X)$. To prove that $(a,b)\in \mathcal S_H$, we consider a diagram $\Delta$ of an element $h_1$ in $H$ with which $t$ is associated. In particular, $\Delta$ has consecutive pairs of branches of the form $$u_101^{m_1}\rightarrow u_201^{m_2}\ \mbox{ and }\ u_110^{n_1}\rightarrow u_210^{n_2}$$
 such that $m_1-m_2=a$, $n_1-n_2=b$ and $u_1,u_2\in[u]_X$. Since $u_1\sim_X u_2$ there is an element $h_2\in H$ with a pair of branches $u_2\rightarrow u_1$. Let $h=h_1h_2$. Then $h$ has consecutive pairs of branches 
$$u_101^{m_1}\rightarrow u_101^{m_2}\ \mbox{ and } \ u_110^{n_1}\rightarrow u_110^{n_2}.$$ In particular, $h$ fixes $\alpha=.u_11$. 
In addition $h'(\alpha^{-})=2^{m_1-m_2}=2^a$ and $h'(\alpha^{+})=2^{n_1-n_2}=2^b$. Thus $(a,b)\in \mathcal S_H$.

In the other direction, let $(a,b)\in\mathcal S_H$. Let $h\in H$ be an element which fixes a finite dyadic fraction $\alpha\in (0,1)$ such that $h'(\alpha^-)=2^a$ and $h'(\alpha^+)=2^b$. In particular, in a small enough left (resp. right) neighborhood of $\alpha$, the slope of $h$ is $2^a$ (resp. $2^b$). Let $v$ be a finite binary word ending with the digit $1$ so that $\alpha=.v$. Let $v'$ be the prefix of $v$ such that $v\equiv v'1$. For all $k\ge 0$, the interval $[v'01^k]$ is a left neighborhood of $\alpha$. For large enough $k>a$, the interval $[v'01^k]$ is a small enough left neighborhood of $\alpha$ so that $h$ has slope $2^a$ on the interval. Since $h$ fixes $\alpha$, the interval $[v'01^k]$ is mapped linearly onto $[v'01^{k-a}]$. In other words, $h$ has the pair of branches $v'01^k\rightarrow v'01^{k-a}$. Let $\Delta$ be a diagram of $h$ which has this pair of branches. Clearly, the following pair must be of the form $v'10^{m_1}\rightarrow v'10^{m_2}$, for some $m_1,m_2\ge 0$. Since $h'(\alpha^+)=2^b$, we have $m_1-m_2=b$. Thus, the tuple $$t=(k-(k-a),m_1-m_2,[v']_X\rightarrow[v']_X)=(a,b,[v']_X\rightarrow[v']_X)\in \mathcal T_H([v']_X).$$ 
It follows that $(a,b)\in\Psi(T_H([v']_X))=\Psi(T_H([u]_X))$, by Corollary \ref{coincide}. 
\end{proof}

To determine whether $\mathcal S_H=\mathbb{Z}^2$, it suffices to find a finite generating set $M$ of $\mathcal S_{H}$.
We start by choosing a generating set of $\mathcal T_H$.

Recall that $X=\{g_1,\dots,g_n\}$ is the fixed generating set of $H$. 
For each $i$, we let $\Delta_i$ be the reduced diagram of $g_i$.
We let $Y$ be the set of all tuples in $\mathcal T_H$ associated with consecutive pairs of branches of the diagrams
 $\Delta_i^{\pm 1}$. For each equivalence class $[u]_X$ of $R_X$, we add to $Y$ the tuple $0_{[u]_X}=(0,0,[u]_X\rightarrow [u]_X)$.  (Notice that all tuples $0_{[u]_X}\in\mathcal T_H$. Indeed, one can consider a diagram of the identity element of $F$ with consecutive pairs of branches of the form $u0\rightarrow u0$ and $u1\rightarrow u1$.) 
To prove that $Y$ is a generating set of $\mathcal T_H$, we will need the following two lemmas. 
The proof of Lemma \ref{add_caret} is simple and is left as an exercise to the reader.

\begin{Lemma}\label{add_caret}
Let $\Delta$ be a diagram of an element in $H$. Let $u\rightarrow v$ be a pair of branches of  $\Delta$. Let $\Delta'$ be the diagram resulting by replacing the edge on the horizontal $1$-path of $\Delta$ at the end of the positive branch $u$  by a dipole of type $1$. Then the tuples in $\mathcal T_H$ corresponding to consecutive pairs of branches of $\Delta'$ are exactly the tuples associated with $\Delta$ and the tuple $(0,0,[u]_X\rightarrow [v]_X)$. \qed
\end{Lemma}

\begin{Lemma}\label{con_12}
Let $h\in H$. Then $h$ has a diagram $\Delta$ which satisfies the following conditions.
\begin{enumerate}
\item[(1)] For each pair of branches  $u\rightarrow v$ of $\Delta$, we have $[u]_X=[v]_X$.
\item[(2)] All the tuples in $\mathcal T_H$ associated with the diagram $\Delta$ belong to the sub-groupoid of $\mathcal T_H$ generated by $Y$. 
\end{enumerate}
\end{Lemma}

\begin{proof}
Before proving the lemma we make the observation that if a diagram $\Delta$ satisfies the conditions in the lemma and $\Delta'$ results from $\Delta$ by the replacement of an edge on the horizontal $1$-path of $\Delta$ by a dipole of type $1$, then $\Delta'$ also satisfies the conditions in the lemma. Indeed, inserting the dipole means replacing a pair of branches $u_1\rightarrow v_1$ by two pairs of branches $u_10\rightarrow v_10$ and $u_11\rightarrow v_11$. Since $[u_1]_X=[v_1]_X$ implies that $[u_10]_X=[v_10]_X$ and $[u_11]_X=[v_11]_X$ (indeed, no foldings of type $1$ are applicable to $\La_{\sem}(X)$), condition (1) of the lemma is satisfied for $\Delta'$. Conditions (1) and (2) for $\Delta$ and Lemma \ref{add_caret} imply that condition (2) of the lemma is satisfied for $\Delta'$.

To prove the lemma we use induction on the word-length $m$ of $h$ with respect to the generating set $X$. 
If $m=1$, then $h=g_j^{\pm 1}$ for $g_j\in X$ and one can take the reduced diagram 
$\Delta_j$ or its inverse.  
By the definition of $\La_{\sem}(X)$, condition (1) is satisfied. Condition (2) is clearly satisfied by the definition of the set $Y$. 

Assume that the lemma is satisfied for every element of $H$ of word length smaller than $m$ and let $h$ be an element of word length $m$. Then $h=fg_i^{\pm 1}$ where $f\in H$ is an element of word-length $m-1$ and $g_i\in X$. 
We assume that $h=fg_i$. The proof in the other case is similar. 
Let $\Delta$ be a diagram for $f$ which satisfies both conditions in the lemma. The reduced diagram 
$\Delta_i$ of the generator $g_i$ also satisfies the conditions. 
By inserting dipoles of type $1$ to $\Delta$ and $\Delta_i$, one can get equivalent diagrams
 $\Delta'$ and $\Delta_i'$ (which also satisfy conditions (1) and (2) above) such that $\Delta'^-\equiv((\Delta_i')^+)^{-1}$. 
Then $\Delta\Delta_i=\Delta'\Delta_i'={\Delta'}^+\circ{\Delta_i'}^-$ is a diagram of the element $h$. We denote ${\Delta'}^+\circ{\Delta_i'}^-$ by $\Delta_h$. 

We note that if $u\rightarrow v$ is a pair of branches of $\Delta_h$, then for some finite binary word $w$, $u\rightarrow w$ is a pair of branches of $\Delta'$ and $w\rightarrow v$ is a pair of branches of $\Delta_i'$. Since $\Delta'$ and $\Delta_i'$ satisfy condition (1), $[u]_X=[w]_X=[v]_X$ and condition (1) is satisfied for the diagram $\Delta_h$.

To see that every tuple in $\mathcal T_H$ associated with the diagram $\Delta_h$ belongs to the sub-groupoid generated by $Y$, it is enough to observe that the tuple associated with the $i$ and $i+1$ pairs of branches of 
$\Delta_h$ is the sum of the tuple associated with the $i$ and $i+1$ pairs of branches of $\Delta'$ and the tuple associated with the $i$ and $i+1$ pairs of branches of $\Delta_i'$. Then the result follows from condition (2) for the diagrams $\Delta'$ and $\Delta_i'$. 
\end{proof}

\begin{Lemma}\label{Y}
The set $Y$ generates the groupoid $\mathcal T_H$.
\end{Lemma}

\begin{proof}
Let $\Delta$ be a diagram of an element $h$ in $H$. It suffices to show that all tuples in $\mathcal T_H$ associated with $\Delta$ belong to the sub-groupoid of $\mathcal T_H$ generated by $Y$. 
By Lemma \ref{con_12}, $h$ can be represented by a diagram $\Delta'$ such that
\begin{enumerate}
\item[(1)] for every pair of branches $u\rightarrow v$ of $\Delta'$ we have $[u]_X=[v]_X$;
\item[(2)] all the tuples in $\mathcal T_H$ associated with $\Delta'$ belong to the sub-groupoid generated by $Y$. 
\end{enumerate}
There is a diagram $\Delta''$ equivalent to both $\Delta$ and $\Delta'$ which results from $\Delta$ and from $\Delta'$ by insertions of dipoles of type $1$. 
It follows from the proof of Lemma \ref{con_12} that all tuples associated with $\Delta''$ also satisfy conditions (1) and (2) above and in particular, they also belong to the sub-groupoid generated by $Y$. Since $\Delta''$ results from $\Delta$ by insertion of dipoles of type $1$, it follows from Lemma \ref{add_caret}, that all tuples associated with $\Delta$ are also associated with $\Delta''$ and as such they all belong to the sub-groupoid $\la Y\ra$ as required.
\end{proof}

Let $N$ be the number of equivalence classes of the relation $R_X$ (i.e., the number of distinct edges in the semi-core $\La_{\sem}(X)$). Let $M'$ be the set of all spherical tuples in $\mathcal T_H$ of word length at most $N$  with respect to the generating set $Y$. Let $M=\Psi(M')$. Clearly, the set $M$ is a finite subset of $\mathcal S_H$. 

\begin{Lemma}\label{M}
The set $M$ is a generating set of $\mathcal S_H$. 
\end{Lemma}

\begin{proof}
Let $(a,b)\in \mathcal S_H$. We claim that $(a,b)$ belongs to $\la M\ra$. Let $u$ be a finite binary word. By Lemma \ref{S_H}, the tuple $t=(a,b,[u]_X\rightarrow [u]_X)$ belongs to $T_H([u]_X)$. 
By Lemma \ref{Y}, $t$ is a product of tuples 
$$t=(a_1,b_1,[v_1]_X\rightarrow[v_2]_X)\cdots(a_m,b_m,[v_m]_X\rightarrow[v_{m+1}]_X)$$
where all tuples $(a_i,b_i,[v_i]_X\rightarrow[v_{i+1}]_X)$ belong to $Y$. Clearly, $[v_1]_X=[v_{m+1}]_X=[u]_X$. 
We prove the lemma by induction on $m$. 
If $m\le N$, then $t\in M'$. Then $(a,b)\in M$ and we are done. If $m>N$, then for some $i<j$ in $\{1,\dots,m\}$ we have $[v_i]_X=[v_j]_X$. Let $[v]_X=[v_i]_X=[v_j]_X$.
We let $$t'=(a_i,b_i,[v_i]_X\rightarrow[v_{i+1}]_X)\cdots (a_{j-1},b_{j-1},[v_{j-1}]_X\rightarrow[v_j]_X)\in\mathcal  T_H([v]_X).$$
Then $t'=(a',b',[v]_X\rightarrow[v]_X)$ for some $a',b'\in\mathbb{Z}$. By induction, 
$\Psi(t')=(a',b')\in \la M\ra$.
It remains to observe that 
\begin{equation*}
\begin{split}
(a,b)=\Psi(t)= & \Psi(t')\Psi((a_1,b_1,[v_1]_X\rightarrow[v_2]_X)\cdots (a_{i-1},b_{i-1},[v_{i-1}]_X\rightarrow[v_i]_X) \\
& (a_j,b_j,[v_j]_X\rightarrow[v_{j+1}]_X)\cdots (a_{m},b_{m},[v_{m}]_X\rightarrow[v_{m+1}]_X))
\end{split}
\end{equation*}
and apply the induction hypothesis. 
\end{proof}

Notice that Lemma \ref{M} provides an algorithm for the solution of Problem \ref{pro}. Indeed, given a finite subset $X$ of $F$, one can construct the finite generating set $M$ of $\mathcal S_H$. Then determining whether $M$ generates $\mathbb{Z}^2$ is a simple linear algebra problem.


\section{$F$ is a cyclic extension of a subgroup $K$ which has a maximal elementary amenable subgroup}\label{sec:max_ame}

In this section we apply the methods developed in this paper to prove the following.

\begin{Theorem}\label{B}
There is a chain of subgroups $B\le K\le F$ in Thompson's group $F$ such that
\begin{enumerate}
\item[(1)] $K$ is a normal subgroup of $F$ and the quotient $F/K$ is infinite cyclic. 
\item[(2)] $B$ is a maximal subgroup of $K$. Moreover, for any $f\in F\setminus B$, we have $K\le \la B,f\ra$.
\item[(3)] $B$ is elementary amenable, $\Cl(B)=B$ and the action of $B$ on the set of finite dyadic fractions $\mathcal D$ is transitive. 
\end{enumerate}
\end{Theorem}

In particular, $B$ from Theorem \ref{B} is an elementary amenable subgroup of $F$ such that the lattice of subgroups of $F$ strictly containing $B$ is isomorphic to the lattice of subgroups of $\mathbb Z$. It is obvious that $K$ and $F$ are co-amenable. Note also that since $K$ contains the derived subgroup of $F$, it contains many copies of $F$. 

\begin{proof}[Proof of Theorem \ref{B}]
In \cite[Section 5]{Brin}, Brin defines an elementary amenable group $G_1$ of elementary class $\omega+2$. The same group was  defined independently about the same time by Navas \cite[Example 6.3]{N}. To realize $G_1$ as a subgroup of $F$, it suffices to let $x$ be an element of $F$ with a single orbital $(a,b)$ and let $y$ be a function in $F$ which maps $(a,b)$ into a fundamental domain of $x$. Then the group $\la x,y\ra$ is a copy of $G_1$ in $F$. 

We let $x=x_0x_1x_2^{-1}x_0^{-1}$ and $y=x_0x_1^{-2}$ and take $B$ to be the subgroup of $F$ generated by $x$ and $y$. 
The pairs of branches of the reduced diagrams of $x$ and $y$ are as follows. 
\[
  x=
  \begin{cases}
  00 &  \rightarrow		00 \\
  010 &  \rightarrow 01\\
	011 &  \rightarrow 100\\
	10 &  \rightarrow 101\\
	11 &  \rightarrow 11\
  \end{cases} 	\qquad	
   y =
  \begin{cases}
	00 & \rightarrow 0\\
	01  & \rightarrow 1000\\
	10 & \rightarrow 1001\\
	110 & \rightarrow 101\\
	111 & \rightarrow 11\
  \end{cases}
\]
Notice that $x$ has a single orbital $(.01,.11)$. The function $y$ maps $(.01,.11)$ onto $(.1,.101)$. Since $x(.1)=.101$, we have that $y(.01,.11)$ is contained in a single fundamental domain of $x$. In particular, $B$ is isomorphic to the group $G_1$ and as such it is elementary amenable.

We let $K=B[F,F]$. Then $K$ is a normal subgroup of $F$. We note that $\pi_{\ab}(K)=\pi_{\ab}(B)=\la (1,1)\ra\le\mathbb{Z}^2$. Since $\mathbb Z^2$ is a cyclic extension of $\pi_{\ab}(K)$, $F$ is a cyclic extension of $K$. Thus, condition (1) of the theorem is satisfied. 

To prove condition (3), we consider the core $\La(B)$. It can be described by the following binary tree. 

\Tree[.$e$ [.$\ell$ [.$\ell$ ] [.$a$ [.$a$ ] [.$c$ ]  ] ] [.$r$ [.$b$ [.$c$ [.$a$ ] [.$b$ ] ] [.$b$ ] ] [.$r$  ] ] ]


To prove that $B$ acts transitively on $\mathcal D$ we note that every edge in $\La(B)$ is the top edge of some positive cell. The edges of $\La(B)$ which are not incident to $\iota(\La(B))$ are $a,b,c$ and $r$. Since $\iota(r)=\iota(b)=\iota(c)=\iota(a)$, there is a unique inner vertex in $\La(B)$. Hence, by Corollary \ref{cor_tra}, $B$ acts transitively on the set $\mathcal D$. 


By Lemma \ref{diagram_group}, $\Cl(B)$ is naturally isomorphic to the diagram group $\DG(\La(B),e)$. 
Applying the algorithm from \cite[Lemma 9.11]{GuSa97} for finding a generating set of a diagram group (over a ``nice enough'' semigroup presentation), one can show that the generating set $\{x,y\}$ of $B$ is also a generating set of $\Cl(B)$. Hence, $B$ is a closed subgroup of $F$. Thus, condition (3) holds for $B$. 

To prove that condition (2) of Theorem \ref{B} is satisfied, we make use of the following lemma. 

\begin{Lemma}\label{eve_ide}
Let $f\in F \setminus B$ and let $H$ be the subgroup of $F$ generated by $B\cup\{f\}$. Then $\Cl(H)$ contains the derived subgroup of $F$ (in fact, $\Cl(H)=F$). 
\end{Lemma}  

\begin{proof}
We observe that every edge in $\La(B)$ is the top edge of some positive cell. Thus, every finite binary word $u$ labels a path on $\La(B)$. We also note that as $\La(B)$ has a unique left boundary edge $\ell$, every non-trivial path $u\equiv 0^n$ for $n\in\mathbb{N}$ must terminate on $\ell$. Similarly, every non-trivial path $1^m$ for $m\in\mathbb{N}$ terminates on the edge $r$. 

Now let $\Delta$ be the reduced diagram of $f$. Since $\Delta$ is not accepted by $\La(B)$ (indeed, $f\notin B=\Cl(B)$), $\Delta$ must have a pair of branches $u_1\rightarrow v_1$, such that on $\La(B)$, $u_1^+$ and $v_1^+$ are distinct edges. It is obvious that both $u_1$ and $v_1$ must contain both digits $0$ and $1$. Thus, $u_1^+,v_1^+$ are inner edges of $\La(B)$ (see Corollary \ref{inner}).

Next, we consider the core $\La(H)$. There is a natural morphism $\psi$ from the core $\La(B)$ to the core $\La(H)$. Indeed, to construct the core of $H$, one can start with the core $\La(B)$, and the diagram $\Delta$; attach the top and bottom edges of $\Delta$ to the distinguished edge of $\La(B)$ and apply foldings. Since every edge in $\La(B)$  is the top edge of some positive cell in $\La(B)$, a series of foldings of type $1$ would guarantee that each cell of the attached diagram $\Delta$ would be folded onto some cell of $\La(B)$. In particular, every edge of $\Delta$ gets identified with some edge of $\La(B)$. We note that in the process of applying foldings, an edge (resp. cell) of $\Delta$ can become identified with more than one edge (resp. cell) of $\La(B)$. So cells of $\La(B)$ might be folded. The morphism $\psi$ maps an edge $e'$ (resp. cell $\pi$) of $\La(B)$ to the edge (resp. cell) of $\La(H)$, identified with the edge $e'$ (resp. cell $\pi$) in this process of folding (or to itself if it did not get identified with anything). It is obvious that $\psi$ is surjective 
and that $\psi$ maps inner (resp. boundary) edges to inner (resp. boundary) edges. 
Thus, the inner edges of $\La(H)$ are $\psi(a),\psi(b)$ and $\psi(c)$. Every path on $\La(B)$ is mapped by $\psi$ to a path on $\La(H)$. In particular, any finite binary word $w$ labels a path on $\La(H)$. 

We claim that $\La(H)$ has a unique inner edge. To prove that, we consider the options for the pair of edges $u_1^+,v_1^+$ in $\La(B)$. If $\{u_1^+,v_1^+\}=\{a,b\}$, then $\Delta$ being accepted by $\La(H)$ implies that in $\La(H)$, the paths $u_1$ and $v_1$ terminate on the same edge. Hence, $\psi(a)=\psi(b)$. Notice that $\psi(a)$ is the top edge of a cell $\pi_1$ in $\La(H)$ with bottom $1$-path $\psi(a)\psi(c)$. Similarly, $\psi(b)$ is the top edge of a cell $\pi_2$ with bottom $1$-path $\psi(c)\psi(b)$. Since $\psi(a)=\psi(b)$ and no foldings are applicable to $\La(H)$, we must have $\pi_1=\pi_2$ and thus, the left bottom edges of the cells satisfy $\psi(a)=\psi(c)$. Thus, $\psi(a)=\psi(b)=\psi(c)$ is the only inner edge of $\La(H)$. A similar argument for the case where
$\{u_1^+,v_1^+\}=\{b,c\}$ or $\{u_1^+,v_1^+\}=\{a,c\}$ shows that $\La(H)$ has a unique inner edge. By Lemma \ref{fin_ind}, $\Cl(H)$ contains the derived subgroup of $F$. 
\end{proof}

Now we can finish the proof using a simple application of Theorem \ref{main1}. We note that for $x\in B$ and $\alpha=.01$ we have $x'(\alpha^+)=2$ and $x'(\alpha^-)=1$. By Lemma \ref{eve_ide}, for any $f\notin B$, the closure of $\la B,f\ra$ contains the derived subgroup of $F$. Hence, by Theorem \ref{main1}, for any $f\notin B$, $\la B,f\ra$ contains $[F,F]$. If follows that $K\le \la B,f\ra$. 
\end{proof}

\begin{Remark}\label{rem_K}
The group $K=B[F,F]$ is $2$-generated. Indeed, one can show that it is generated by $x_0x_1^{-2}$ and $x_1x_3^{-1}$ by another application of Theorem \ref{main1}. If $K'=\la x_0x_1^{-2},x_1x_3^{-1}\ra$ then by considering the image in the abelianization we get that $K'[F,F]=K$. One can check that $\Cl(K')=F$. Then since the element $x_0x_1^{-2}$ fixes the fraction $.1$, has slope $1$ at $.1^-$ and slope $2$ at $.1^+$,  Theorem \ref{main1} implies that $[F,F]\le K'$ and as such $K=K'$. 
\end{Remark}

\section{Computations related to $\La(H)$}\label{sec:tech}

\subsection{On the algorithm for finding a generating set of $\Cl(H)$}\label{ss:alg}

Let $\kk'$ be a directed $2$-complex. (We denote it by $\kk'$, as $\kk$ is still used to denote the Dunce hat.) As noted in Section \ref{ss:dc}, $\kk'$ can be described in a form similar to a semigroup presentation. Let $E$ be the set of edges of $\kk'$.  We let 
$$\P=\la E \mid \topp(f)\rightarrow\bott(f), f\in F^-\ra$$
be the \emph{semigroup presentation associated with $\kk'$}. We use negative $2$-cells instead of positive ones, as it will be more convenient below. The presentation $\P$ defines a semigroup $S$ associated with the directed $2$-complex $\kk'$. 
The semigroup $S$ is closely related to the diagram groupid $\mathcal D(\kk')$. Indeed, let $u$ and $v$ be two $1$-paths in $\kk'$. Then $u$ and $v$ can be viewed as words over the alphabet $E$. Then $u$ and $v$ represent the same element of $S$ if and only if there is a $(u,v)$-diagram over $\kk'$ \cite{GSdc}. In particular, if $u=v$ in $S$, then $\iota(u)=\iota(v)$ and $\tau(u)=\tau(v)$ in $\kk'$. In general, if $u$ and $v$ are two words in the alphabet $E$, then $u$ and $v$ are equal as elements of $S$ if and only if there is a $(u,v)$-diagram over $\kk''$, where $\kk''$ is the directed $2$-complex resulting from $\kk'$ by the identification of all vertices to a single vertex \cite{GSdc}. We also note that if words $u$ and $v$ are equal in $S$ then $u$ is a $1$-path in $\kk'$ if and only if $v$ is a $1$-path in $\kk'$. 


As mentioned in the proof of Theorem \ref{B}, there is an algorithm due to Guba and Sapir \cite{GuSa97} for finding a generating set of a diagram group. Let $p$ be a $1$-path in $\kk'$. To find a generating set for $\DG(\kk',p)$ we make some further assumptions about the directed $2$-complex $\kk'$. First, we assume that the set of edges $E$ is equipped with a total-order $\prec$. The total order $\prec$ induces a lexicographc order on the set $E^*$ of words in the alphabet $E$. If $w_1$ and $w_2$ are two words over $E$ then $w_1$ is smaller than $w_2$ in the \emph{ShortLex} order if $|w_1|<|w_2|$ or $|w_1|=|w_2|$ and $w_1$ precedes $w_2$ in the lexicographic order. We assume that for each positive $2$-cell $f$ in $\kk'$, $\topp(f)$ is smaller than $\bott(f)$ in the ShortLex order. In particular, in each rewriting rule $r_1\rightarrow r_2$ in the presentation $\P$, $r_2$ is smaller than $r_1$ in the ShortLex order. 

Let $\P'$ be a completion of $\P$ (for terminology see \cite{Sa}) such that every rewriting rule $r_1\rightarrow r_2$ in $\P'$ satisfies $r_1>r_2$ in the ShortLex order. We also assume that for any relation $r_1\rightarrow r_2$ in $\P'$ we can point on a derivation over $\P$ from $r_1$ to $r_2$. Notice that if one applies the Knuth-Bendix algorithm (see, for example, \cite{Sa}) for finding a completion $\P'$ of $\P$, then  the resulting completion (if attained) satisfies these $2$ properties. Whenever a presentation $\P$ associated with a directed $2$-complex and a completion $\P'$ of it are mentioned below, we assume that the set of edges $E$ is equipped with a total order $\prec$ and that $\P$ and $\P'$ satisfy the above properties. 

To implement the algorithm from \cite[Lemma 9.11]{GuSa97} for finding a generating set of $\DG(\kk',p)$ one has to find the set $B$ of all tuples $[u,r_1\rightarrow r_2,v]$ such that 
\begin{enumerate}
\item[(1)] $u$ and $v$ are words in the alphabet $E$  which are reduced over the complete presentation $\P'$;
\item[(2)] $r_1\rightarrow r_2$ is a rewriting rule of $\P'$;
\item[(3)] $ur_1v$ and $p$ are equal as elements of $S$; and 
\item[(4)] In the notation of \cite{GuSa97}, if one lets $u\equiv u$, $v\equiv v$, $r\equiv r_2$ and $\ell\equiv r_1$, then the tuple $(u,r\rightarrow \ell,v)$ does not satisfy at least one of the conditions in \cite[Definition 9.1]{GuSa97}.
\end{enumerate}

We did not give here the $4^{th}$ condition in detail as it will not be important to us. Suffice it to know that if one can find the set of all tuples $[u,r_1\rightarrow r_2,v]$ which satisfy conditions (1)-(3) then one can check for each one of them if it satisfies condition (4) and thus, if it belongs to $B$ or not. Given the set $B$, the algorithm in \cite{GuSa97}, shows how to associate an element of $\DG(\kk',p)$ with each tuple in $B$. The set of elements associated with tuples in $B$ is a generating set of $\DG(\kk',p)$. If $\P$ is complete then the generating set is minimal. 

The process of finding the element of $\DG(\kk',p)$ associated with a given tuple in $B$ is straightforward. The difficult parts in the algorithm are (1) finding a completion $\P'$ and (2) finding the set of tuples $B$. We note that if $\kk'$ is a finite directed $2$-complex, then $\P$ is finite. However, this does not imply the existence of a finite completion. Even if there is a finite completion $\P'$, the set $B$ might still be infinite and there may be no simple way to find it. Thus, implementing the algorithm from \cite[Lemma 9.11]{GuSa97} is often impractical. 

If the diagram group in question is the closure of a subgroup $H$ of $F$, the situation is ameliorated. While the first problem remains valid, the task of constructing the set $B$ becomes very simple. Let $H$ be a subgroup of $F$. Let $\La=\La(H)$ and let $p=p_{\La(H)}$ be the distinguished edge of $\La(H)$. Let $\P$ be the semigroup presentation associated with $\La$ and let $S$ be the semigroup it defines. We note that every relation $r_1\rightarrow r_2$ in $\P$ is such that $|r_1|>|r_2|$ (indeed, the top $1$-path of a negative cell is longer than the bottom $1$-path of the cell). Thus, regardless of the order one fixes on the set of edges $E$, the presentation $\P$ is as required.

 There are two main reasons why the algorithm for finding a generating set of $\Cl(H)\cong \DG(\La,p)$ is simplified in our context.


The first reason is that every diagram $\Delta$ over $\La$ (or over the directed $2$-complex $\La$ with all vertices identified) can be viewed as a diagram over the Dunce hat $\kk$. Moreover, as in the proof of Lemma \ref{diagram_group}, if $\Delta$ is reduced as a diagram over $\La$, then it is also reduced as a diagram over $\kk$. Thus, if $\Delta$ is a reduced diagram over $\La$, then there is a horizontal $1$-path in $\Delta$, which passes through all the vertices of $\Delta$ and separates it to a concatenation of a positive subdiagram $\Delta^+$ and a negative subdiagram $\Delta^-$. 
We note that the implication for the semigroup $S$ is that if $w_1$ and $w_2$ are words over $E$, equal as elements of $S$, then to get from $w_1$ to $w_2$ one can apply a positive derivation over $\P$ followed by a negative derivation (for terminology, see \cite{Sa}). 

The second reason is that as usual, if $\Delta$ is a diagram in the core, one can consider paths on $\Delta$ which can be mapped to paths on the core $\La(H)$ and thus use results from previous sections. 

In the proof of the following proposition we will often consider diagrams $\Delta$ alternately as diagrams over $\La$ and as diagrams over $\kk$. When we refer to the label $\lab(e)$ of an edge $e$ or $\lab(q)$ of a $1$-path $q$ in $\Delta$, the label refers to the label of the edge or $1$-path when $\Delta$ is viewed as a diagram over $\La$. In particular, $\lab(e)$ is an edge of $\La$. 

Let $w_1,w_2$ be words over $E$. We say that $w_1$ is a \emph{left divisor} of $w_2$ in $S$ if there is a word $a$ over $E$ such that $w_1a$ is equal to $w_2$ in $S$. Right divisors are defined in a similar way. 

\begin{Proposition}\label{prop:left}
Let $H$ be a subgroup of $F$. Let $\La=\La(H)$ and let $p=p_{\La(H)}$ be the distinguished edge of $\La$. Let $E$ be the set of edges of $\La$. Let $\P$ be the semigroup presentation associated with $\La$ and let $S$ be the semigroup given by $\P$. Let $w,w_1$ and $w_2$ be words over the alphabet $E$. Then the following assertions hold.
\begin{enumerate}
\item[(1)] $w$ is a left divisor of $p$ in $S$ if and only if $w$ is a $1$-path in $\La$ such that $\iota(w)=\iota(\La)$. 
\item[(2)] Assume that $w_1,w_2$ are non-empty words over $E$ such that $w_1$ and $w_2$ are left divisors of $p$ in $S$. Then $w_1$ is equivalent to $w_2$ in $S$ if and only if the $1$-paths $w_1$ and $w_2$ satisfy $\tau(w_1)=\tau(w_2)$.
\end{enumerate}
\end{Proposition}

\begin{proof}
(1) Assume that $w$ is a left divisor of $p$. By definition, there is a word $a$ such that $wa=p$ in $S$. Since $p$ is a $1$-path in $\La$, $wa$ is also a $1$-path in $\La$ and $\iota(wa)=\iota(p)=\iota(\La)$, $\tau(wa)=\tau(\La)$. In particular, $w$ is a $1$-path with initial vertex $\iota(\La)$. 

In the other direction, let $w$ be a $1$-path in $\La$ with initial vertex $\iota(\La)$. Assume by contradiction that $w$ is not a left divisor of $p$. We can assume that $w$ is a minimal $1$-path with these properties. Clearly, $w$ is not a trivial path (i.e., there are edges in the path $w$).  Similarly, $w$ is not composed of a single edge. Otherwise, let $u$ be a path on $\La$ such that $u^+=w$ (where $w$ is viewed as an edge). $u\equiv 0^k$ for some $k\ge 0$, since $w$ is a left boundary edge of $\La$. 
Let $\Psi$ be the minimal tree-diagram over $\kk$ with branch $u$. Since $u$ labels a path on $\La$, $\Psi$ can be viewed as a diagram over $\La$ with $\lab(\topp(\Psi))\equiv p$ and $\lab(\bott(\Psi))\equiv wd$ for some word $d$ over $E$. Then $wd=p$ in $S$ and  $w$ is a left divisor of $p$, which contradicts the assumption. 

Thus, we can write $w\equiv w'c$ where $c$ is the last letter of $w$ and $w'$ is not empty. The minimality of $w$ implies that $w'$ is a left divisor of $p$. Thus, there is a reduced $(p,w'q)$-diagram $\Delta$ over $\La$ for some word $q$ over $E$. We let $e$ be the $|w'|+1$ edge on $\bott(\Delta)$, so that $\lab(e)$ is the first letter of $q$. We let $e'$ be the edge on the horizontal $1$-path of $\Delta$ such that $\iota(e')=\iota(e)$ and let $c'\equiv \lab(e')$. Then $c'$ is an edge of $\La$ such that $\iota(c')=\tau(w')=\iota(c)$. If $\Delta$ is viewed as a diagram over $\kk$, then the positive subdiagram $\Delta^+$ is a tree-diagram. Let $u$ be the path on $\Delta^+$ with terminal edge $e'$. Let $\Psi$ be the minimal tree-diagram over $\kk$ with branch $u$. Then $\Psi$ can be viewed as a subdiagram of $\Delta^+$ with $\topp(\Psi)=\topp(\Delta^+)$  (see Figure \ref{fig1}).  The edge $e'$ lies on the bottom $1$-path of $\Psi$. Thus, $\bott(\Psi)=ae'b$, where $a$ and $b$ are the suitable $1$-paths in $\Delta^+$. Let $a'\equiv \lab(a)$ and $b'\equiv \lab(b)$. It follows that $\Delta$ has an $(a',w')$-subdiagram. Hence $a'$ and $w'$ represent the same element of $S$. We claim that $a'c$ is a left divisor of $p$ in $S$. That would yield the required contradiction, as $w\equiv w'c$ is equivalent to $a'c$ in $S$. We note that, as $1$-paths on $\La$, $\tau(a')=\tau(w')$. Thus, $a'c$ is a $1$-path in $\La$.

\begin{figure}[h]
\centering
\includegraphics[width=0.45\columnwidth]{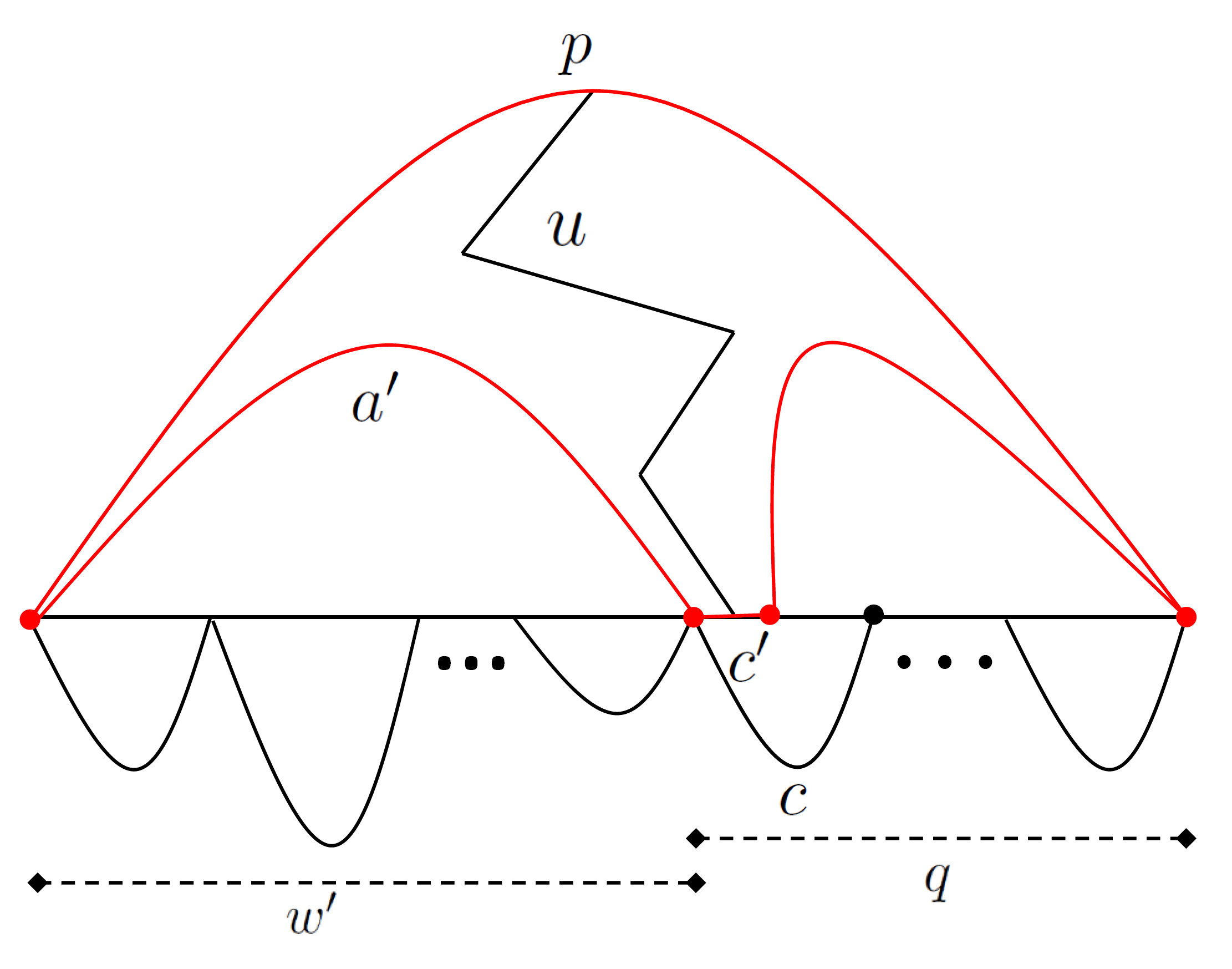}
\caption{The diagram $\Delta$. The top and bottom $1$-paths of the subdiagram $\Psi$ are colored red. The labels of edges or $1$-paths in the figure are their labels when $\Delta$ is viewed as a diagram over $\La$.}
    \label{fig1}
\end{figure}
 
Since $\Delta$ is a diagram over $\La$ with $\lab(\topp(\Delta))\equiv p$, the path $u$ in $\Delta$ implies that $u$ labels a path on $\La$ such that $u^+=c'$. Since $\iota(c)=\iota(c')\neq\iota(\La)$, by Proposition \ref{inner_vertex}, $c$ and $c'$ belong to the same connected component of $\Gamma(H)$. 
Let $v$ be a path on $\La$ with terminal edge $c$. By the comments following Remark \ref{derivation}, there are $m,n\ge 0$ such that $u0^m$ and $v0^n$ label paths on $\La$ and such that $(u0^m)^+=(v0^n)^+$ on $\La$. Thus, by Lemma \ref{cor_ide}, there is a diagram $\Delta'$ in $\Cl(H)$, accepted by $\La$, with a pair of branches $u0^m\rightarrow v0^n$ (see Figure \ref{fig2}). Let $d$ be the edge on the horizontal $1$-path of $\Delta'$ such that the positive branch $u0^m$ and the negative branch $v0^n$ terminate on $d$. We let $d_1$ be the terminal edge of the positive path $u$ on $\Delta'$ and $d_2$ be the terminal edge of the negative branch $v$ on $\Delta'$. Then $\iota(d)=\iota(d_1)=\iota(d_2)$.

Consider $\Delta'$ as a diagram over $\La$. Then $\lab(d_1)\equiv c'$ and $\lab(d_2)\equiv c$. Since $u$ labels a positive path on $\Delta'$, the tree-diagram $\Psi$ can be viewed as a subdiagram of $\Delta'^+$ with $\topp(\Psi)=\topp(\Delta')$. In particular, the bottom $1$-path of $\Psi$ is a $1$-path in $\Delta'$.  Recall that $\lab(\bott(\Psi))\equiv a'c'b'$ where $c'$ is the label of the edge $u^+$ of $\Psi$, i.e., the edge $d_1$ of $\Delta'$. Since $\iota(d_1)=\iota(d_2)$ and $\lab(d_2)\equiv c$, we get that $a'c$ labels a $1$-path in $\Delta'$, which starts from $\iota(\Delta')$. Extending the $1$-path, we get a $1$-path with label $a'cs$ (for some word $s$) with initial vertex $\iota(\Delta')$ and terminal vertex $\tau(\Delta')$. Then $\Delta'$ has a $(p,a'cs)$-subdiagram, which implies that $a'c$ is a left divisor of $p$ in $S$. 

\begin{figure}[h]
\centering
\includegraphics[width=0.45\columnwidth]{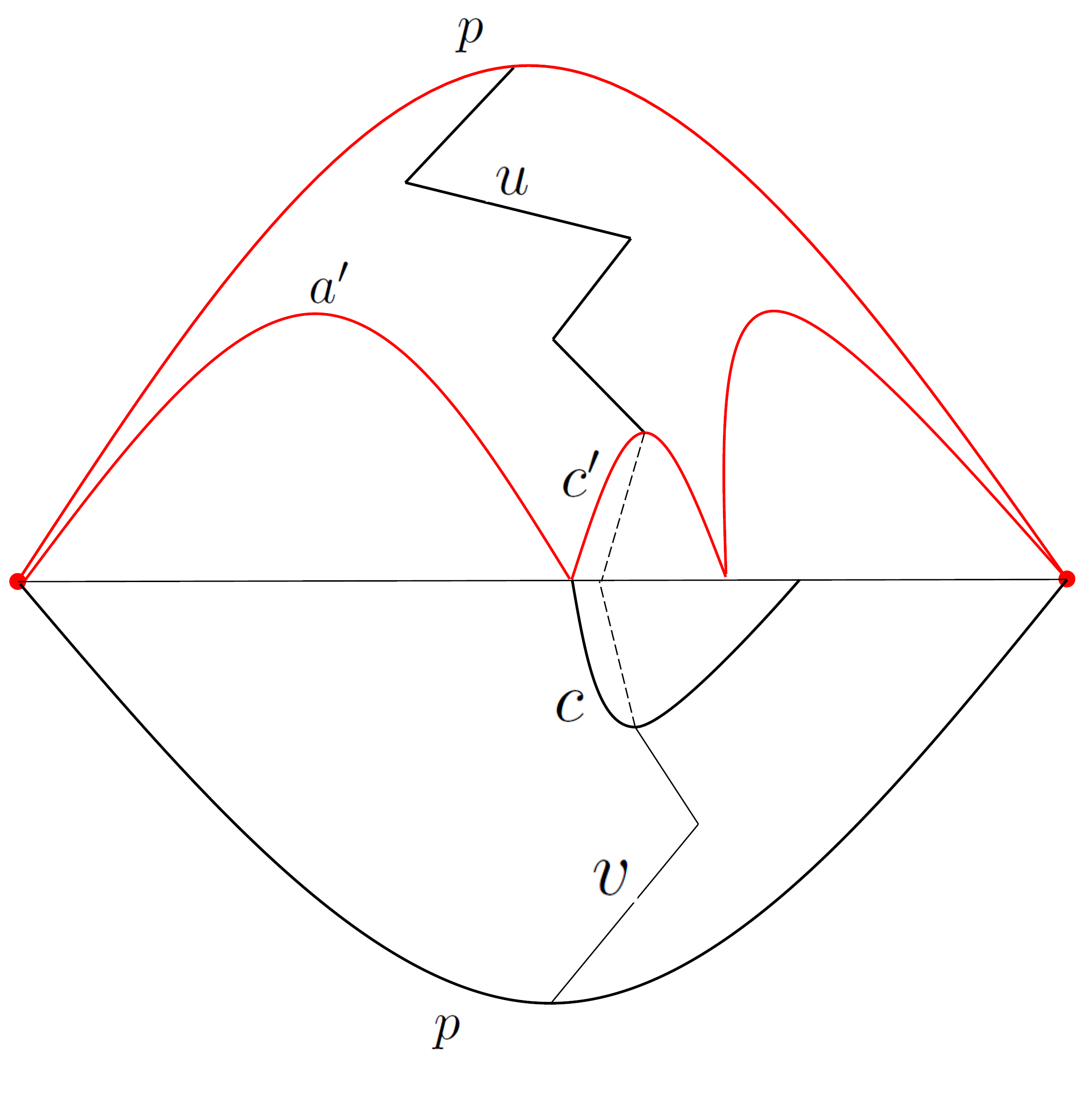}
\caption{The diagram $\Delta'$. The top and bottom $1$-paths of the subdiagram $\Psi$ are colored red. The labels of edges or $1$-paths in the figure are their labels when $\Delta'$ is viewed as a diagram over $\La$.}
    \label{fig2}
\end{figure}

(2) Let $w_1,w_2\neq\emptyset$ be left divisors of $p$ in $S$. By part (1), $w_1$ and $w_2$ are $1$-paths in $\La$ with initial vertex $\iota(\La)$. If $w_1$ and $w_2$ are equivalent in $S$, then there is a $(w_1,w_2)$-diagram over $\La$. Then $\tau(w_1)=\tau(w_2)$ as required. In the other direction, assume that $\tau(w_1)=\tau(w_2)$. 
Let $q_1$ and $q_2$ be words over $E$ such that $w_iq_i$ is equal to $p$ in $S$. If $q_1$ is empty, then $w_1=p$ in $S$ and so, $\tau(w_1)=\tau(p)=\tau(\La)$. It follows that $\tau(w_2)=\tau(\La)$. As $\tau(\La)$ has no outgoing edges, $q_2$ is also empty. It follows that $w_2=p$ in $S$. Thus, $w_1=w_2$ in $S$. 

Therefore, we can assume that $q_1$ and $q_2$ are not empty words. Let $\Delta_i$, $i=1,2$ be a $(p,w_iq_i)$-diagram over $\La$. Let $c_i$ be the first letter of the word $q_i$. As $\tau(w_1)=\tau(w_2)$, the edges $c_1$ and $c_2$ of $\La$ have the same initial vertex. 
We apply an argument similar to the one in part (1). 
Let $e_i$ be the $|w_i|+1$ edge on $\bott(\Delta_i)$, so that $\lab(e_i)\equiv c_i$. Let $e_i'$ be the edge on the horizontal $1$-path of $\Delta_i$ with the same initial vertex as $e_i$. Let $c_i'\equiv \lab(e_i')$ and let $u_i$ be the path on $\Delta_i^+$ with terminal edge $e_i'$. Let $\Psi_i$ be the minimal tree-diagram with path $u_i$. $\Psi_i$ can be viewed as a subdiagram of $\Delta_i^+$ with $\topp(\Psi_i)=\topp(\Delta_i)$ and 
$\bott(\Psi_i)=a_ie_i'b_i$. We let $a_i'\equiv \lab(a_i)$ and $b_i'\equiv \lab(b_i)$. Then $\Delta_i$ has an $(a_i',w_i)$-subdiagram and so, $a_i'$ is equivalent to $w_i$ in $S$. 

Since $\iota(c_i')=\tau(w_i)$, we have $\iota(c_1')=\iota(c_2')$ in $\La$. We note that $u_i$ labels a path on $\La$ with terminal edge $c_i'$. Thus, by Proposition \ref{inner_vertex} and Remark \ref{derivation} there are $m_i\ge 0$ such that $u_i0^{m_i}$ labels a path on $\La$ and such that $(u_10^{m_1})^+=(u_20^{m_2})^+$ in $\La$. 
By Lemma \ref{cor_ide}, there is a diagram $\Delta'$ accepted by $\La$ with a pair of branches $u_10^{m_1}\rightarrow u_20^{m_2}$. 
The minimality of $\Psi_i$ implies that if $\Delta'$ is viewed as a diagram over $\La$, then $\Psi_1$ is a subdiagram of ${\Delta'}^+$ with $\topp(\Psi_1)=\topp({\Delta'}^+)$ and $\Psi_2^{-1}$ is a subdiagram of ${\Delta'}^-$ with $\bott(\Psi_2^{-1})=\bott({\Delta'}^-)$. Then it is easy to see (Figure \ref{fig3}) that $\Delta'$ has an $(a_1',a_2')$-subdiagram. Hence, $a_1'$ and $a_2'$ are equivalent in $S$. Since $a_i'$ is equivalent to $w_i$ in $S$, $w_1=w_2$ in $S$, as required.

\begin{figure}[h]
\centering
\includegraphics[width=0.45\columnwidth]{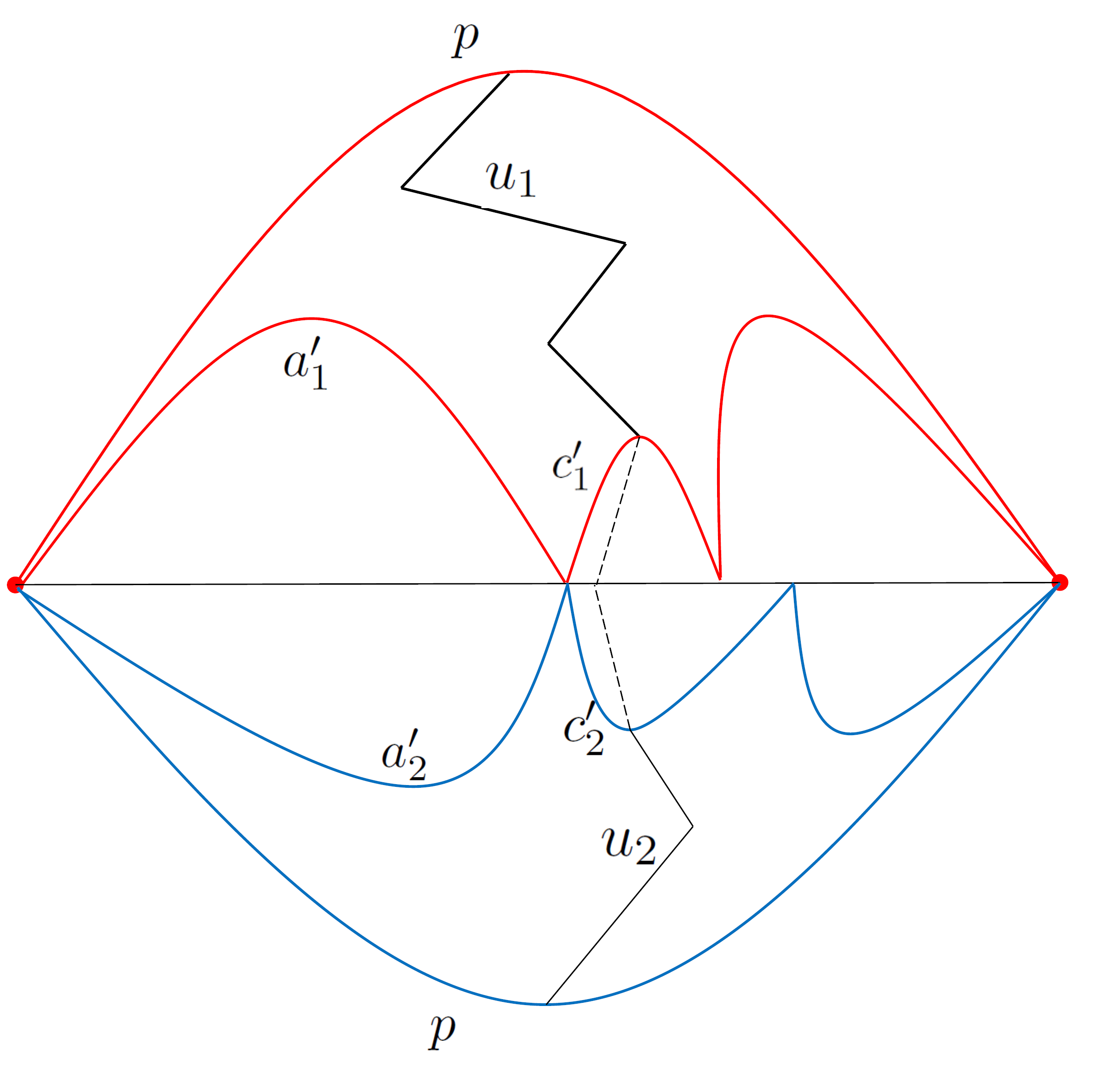}
\caption{The diagram $\Delta'$. The top and bottom $1$-paths of $\Psi_1$ and of $\Psi_2^{-1}$ are colored red and blue respectively. The labels of edges or $1$-paths in the figure are their labels when $\Delta'$ is viewed as a diagram over $\La$.}
    \label{fig3}
\end{figure}

\end{proof}

In a similar way, one can prove the following right-left analogue of Proposition \ref{prop:left}.

\begin{Proposition}\label{prop:right}
Let $H$ be a subgroup of $F$. Let $\La=\La(H)$ and let $p=p_{\La(H)}$ be the distinguished edge of $\La$. Let $E$ be the set of edges of $\La$. Let $\P$ be the semigroup presentation associated with $\La$ and let $S$ be the semigroup given by $\P$. Let $w,w_1$ and $w_2$ be words over the alphabet $E$. Then the following assertions hold.
\begin{enumerate}
\item[(1)] $w$ is a right divisor of $p$ in $S$ if and only if $w$ is a $1$-path in $\La$ such that $\tau(w)=\tau(\La)$. 
\item[(2)] Assume that $w_1,w_2$ are non-empty words over $E$ such that $w_1$ and $w_2$ are right divisors of $p$ in $S$. Then $w_1$ is equivalent to $w_2$ in $S$ if and only if the $1$-paths $w_1$ and $w_2$ satisfy $\iota(w_1)=\iota(w_2)$.
\end{enumerate}
\end{Proposition}

\begin{Corollary}\label{set_B}
Let $H$ be a subgroup of $F$ and let $p=p_{\La(H)}$ be the distinguished edge of $\La(H)$. Let $\P$ be the semigroup presentation associated with $\La=\La(H)$ and let $S$ be the semigroup given by $\P$. Let $\P'$ be a completion of $\P$ with the properties described at the top of the section. 
Then for each relation $r_1\rightarrow r_2$ in $\P'$ there is at most one tuple $[u,r_1\rightarrow r_2,v]$ which satisfies conditions (1)-(3) out of conditions (1)-(4) listed at the top of the section.
\end{Corollary}

\begin{proof}
Let $r_1\rightarrow r_2$ be a relation in $\P'$. If $r_1$ is not a $1$-path in $\La$, then there is no tuple $[u,r_1\rightarrow r_2,v]$ as described. Indeed, condition (3) says that $p=ur_1v$ in $S$. Since $p$ is a $1$-path in $\La$, that would imply that $ur_1v$, and as such, that $r_1$, is a $1$-path in $\La$. 

Now, assume that $r_1$ is a $1$-path in $\La$ (if one follows the standard Knuth-Bendix algorithm \cite{Sa}, then every relation $r_1\rightarrow r_2$ in $\mathcal P'$ is of this kind). If $[u,r_1\rightarrow r_2,v]$ is as described, then since $p=ur_1v$ in $S$, $u$ and $v$ must be $1$-paths in $\La$ such that $\iota(u)=\iota(\La)$, $\tau(u)=\iota(r_1)$, $\iota(v)=\tau(r_1)$ and $\tau(v)=\tau(\La)$. The requirement that $u$ and $v$ are reduced implies, By Propositions \ref{prop:left} and \ref{prop:right}, that there is only one choice for $u$ and $v$; namely, $u$ is the unique $1$-path in $\La$ with $\iota(u)=\iota(\La)$ and $\tau(u)=\iota(r_1)$ such that the word $u$ is reduced over $\P'$ (if $\iota(r_1)=\iota(\La)$ then $u\equiv\emptyset$). Similarly, there is only one option for the choice of $v$. Finally, we note that if $u$ and $v$ are taken to be $1$-paths as described, then conditions (1)-(3) are satisfied for the tuple $[u,r_1\rightarrow r_2,v]$. Indeed, $u$ and $v$ are reduced words over $\mathcal P'$, $r_1\rightarrow r_2$ is a relation of $\mathcal P'$ and since $ur_1v$ is a $1$-path in $\La$ such that $\iota(ur_1v)=\iota(\La)$ and $\tau(ur_1v)=\tau(\La)$, by Proposition \ref{prop:left}(2), $ur_1v$ is equivalent to $p$ in $S$. 
\end{proof}

The proof of Corollary \ref{set_B} shows that given a subgroup $H\le F$ with finite $\La(H)$ and a completion $\P'$ of the semigroup presentation $\P$ associated with $\La(H)$, there is a simple algorithm for finding the set of tuples $B$ (and thus, for implementing the algorithm from \cite{GuSa97} for finding a generating set of $\Cl(H)$). Indeed, for each inner vertex $x$ of $\La(H)$ one can find $1$-paths $p$ and $q$ on $\La(H)$ such that $\iota(p)=\iota(\La(H))$, $\tau(p)=x$, $\iota(q)=x$ and $\tau(q)=\tau(\La(H))$. Then applying rewriting rules from $\P'$ one can find such $1$-paths $p$ and $q$, which are reduced as words in the alphabet $E$ over $\P'$. 

\subsection{Core $2$-automata}\label{ss:core}

In previous sections we sometimes described the core $\La(H)$ of a subgroup $H\le F$ using a labeled binary tree. In this section we  generalize this notion and make it precise. As in the rest of the paper, all $2$-automata $\La$ considered in this section are $2$-automata over the Dunce hat  $\kk$ with distinguished $1$-paths $p_{\La}=q_{\La}$ composed of one edge of $\La$. We always assume that the immersion $\psi_{\La}$ from $\La$ into $\kk$ maps every positive cell of $\La$ to the positive cell of $\kk$. Recall that $\La$ is a folded-automaton if no foldings are applicable to it. The proof of \cite[Lemma 3.21]{GS1} shows that if $\La$ is a folded-automaton and $\Delta$ is a diagram in $F$ accepted by $\La$ then the reduced diagram equivalent to $\Delta$ is also accepted by $\La$. Thus, one can talk about the subgroup of $F$ accepted by $\La$.

\begin{Lemma}
Let $H$ be a subgroup of $F$. Then $H$ is a closed subgroup of $F$ if and only if there is a folded-automaton $\La$ over the Dunce hat $\kk$ such that $H$ is the subgroup of $F$ accepted by $\La$. 
\end{Lemma}

\begin{proof}
If $H$ is a closed subgroup of $F$ then one can take $\La=\La(H)$. 
In the other direction, if $H$ is the subgroup of all diagrams in $F$ accepted by $\La$ then Remark \ref{r:GS} implies that $H$ is closed for components. Then by Corollary \ref{cor:clo}, $H$ is a closed subgroup of $F$. 
\end{proof}

Let $\La$ be a folded automaton over $\kk$. The subgroup of $F$ accepted by $\La$ is naturally isomorphic to the diagram group $G=\DG(\La,p_{\La})$ (indeed, the proof is identical to that of Lemma \ref{diagram_group}). 
In Section \ref{ss:alg} we have seen that if $\La$ is the core $\La(H)$ of some subgroup $H$ of $F$ then the algorithm from \cite{GuSa97} for finding a generating set of $G$ can be simplified. For this and for other reasons (see Section \ref{ss:max} below), it is useful to determine if $\La$ coincides with $\La(H)$ for some subgroup $H\le F$. We are only interested in $\La$ coinciding with $\La(H)$ when  all vertices of $\La$ and all vertices of $\La(H)$ are identified to a single vertex. Indeed, by Remark \ref{one_vertex}, the diagram groups $G$ and $\DG(\La(H),p_{\La(H)})\cong\Cl(H)$ are not affected by identification of vertices in $\La$ and in $\La(H)$. 

\begin{Definition}
Let $\La$ be a folded-automaton over $\kk$. We say that $\La$ is a \emph{core-automaton} if there is a subgroup $H\le F$ and a bijective morphism $\phi$ from $\La(H)$ with all vertices identified to $\La$ with all vertices identified. 
\end{Definition}

We note that the naive approach for deciding if a folded automaton $\La$ is a core automaton is to find a generating set $X$ of $\DG(\La,p_{\La})$, let $H$ be the subgroup of $F$ generated by $X$ (where elements in $X$ are viewed as reduced diagrams in $F$) and construct the core $\La(H)$. Then one has to check if $\La$ and $\La(H)$ coincide up to identification of vertices. This approach would work as long as one can find a generating set of $\DG(\La,p_{\La})$, but as we are trying to simplify the process of finding such a generating set we consider a different approach. Namely, we associate labeled binary trees $T_{\La}$ and $T_{\La}^{\min}$ with the folded-automaton $\La$. The folded automaton will be a core automaton if and only if the tree $T_{\La}^{\min}$ satisfies certain properties (see Lemma \ref{3con} below). 

Given a labeled binary tree $T$,  a \emph{path} $p$ in $T$ is always a simple path starting from the root. Every path is labeled by a finite binary word $u$. As for paths on diagrams, we will rarely distinguish between the path $p$ and its label $u$. Similarly, we will denote by $p^+$ or $u^+$ the terminal vertex of the path $p$ in $T$. $\lab(u^+)$ or $\lab(v^+)$ will denote the label of this terminal vertex. An \emph{inner} vertex of $T$ is a vertex which is not a leaf. \emph{Brother vertices} of $T$ are distinct vertices with a common father.

Let $\La$ be a folded-automaton. The labeled binary tree $T_{\La}$ associated with $\La$ is defined as follows. The labels of vertices in $T_{\La}$ are edges of $\La$.  Recall (Section \ref{sec:paths}) that each finite binary word $u$ labels at most one path on $\La$. We let $T_{\La}$ be the maximal binary tree such that for every path $u$ in $T_{\La}$, the finite binary word $u$ labels a path on $\La$. For example, if every edge in $\La$ is the top edge of some cell, then $T_{\La}$ is the complete infinite binary tree. The label of each vertex $u^+$ of $T_{\La}$ is the edge $u^+$ of $\La$.

Notice that every caret in $T_{\La}$ is labeled with accordance with the top and bottom edges of some positive cell in $\La$. In fact, $T_{\La}$ can be constructed inductively as follows. One starts with a root labeled by the distinguished edge $p_{\La}$ of $\La$. Whenever there is a leaf in the tree whose label is the top edge of some positive cell $\pi$ in $\La$,  one attaches a caret to the leaf and labels the left (resp. right) leaf of the caret by the left (resp. right) bottom edge of $\pi$. 

Now let $T$ be a rooted subtree of $T_{\La}$, maximal with respect to the property that there is no pair of distinct inner vertices in $T$ which have the same label. If $\ell$ is a leaf of $T$ and $\ell$ does not share a label with any inner vertex in $T$, then $\ell$ must be a leaf of $T_{\La}$. Indeed, otherwise one could attach the caret of $T_{\La}$ with  root $\ell$ to the the subtree $T$ and get a larger subtree where no pair of distinct inner vertices share a label. 

If the leaf $\ell$ shares a label with some inner vertex $x$ of $T$, then in $T_{\La}$, $\ell$ has two children. Each  child of $\ell$ is labeled as the respective child of $x$. Continuing in this manner, we see that it is possible to get $T_{\La}$ from $T$, by inductively attaching carets to leaves which share their label with inner vertices of $T$ and labeling the new leaves appropriately. It follows that $T$ and $T_{\La}$ have the same set of labeled carets. Since no labeled caret appears in $T$ more than once, $T$ is a minimal subtree of $T_{\La}$ with respect to the property that the sets of labeled carets of $T$ and $T_{\La}$ coincide. 

We let a \emph{minimal tree associated with $\La$}, denoted $T_{\La}^{\min}$ be a tree $T$ as described in the preceding paragraph. We note that a minimal tree associated with $\La$ is not unique. However, the label of the root of $T_{\La}^{\min}$ and the set of labeled carets of $T_{\La}^{\min}$ are determined uniquely by $\La$. Thus, we can consider different minimal trees associated with $\La$ to be equivalent. 
 All labeled binary trees which appeared in this paper so far were minimal trees associated with the cores $\La(H)$ they described.

 
 \begin{Lemma}\label{lem_ez}
 	Let $\La$ be a folded-automaton over $\kk$. Let $T_{\La}^{\min}$ be an associated minimal tree. Assume that for all $u$ and $v$ which label paths on $T_{\La}^{\min}$ such that $u^+$ and $v^+$ share a label, there is a diagram $\Delta_{u,v}$ accepted by $\La$ with a pair of branches $u\rightarrow v$.
 	Let $K$ be the subgroup of $F$ generated by the diagrams $\Delta_{u,v}$ for each pair of paths $u,v$ on $T_{\La}^{\min}$ with $\lab(u^+)\equiv \lab(v^+)$.
 Let $T_{\La}$ be the labeled binary  tree associated with $\La$. 
  	Let $u_1$ and $v_1$ be paths on the tree $T_{\La}$  such that $\lab(u_1^+)\equiv \lab(v_1^+)$. Then there is an element $k\in K$ with a pair of branches $u_1\rightarrow v_1$. 
 \end{Lemma}
 
 \begin{proof}
 	The proof is similar to the proof of Lemma \ref{path_lem} for foldings of type $1$. The diagrams $\Delta_{u,v}$ show that the lemma holds if the paths $u_1$ and $v_1$ belong to the subtree $T_{\La}^{\min}$ of $T_{\La}$. Recall that $T_{\La}$ can be constructed from $T_{\La}^{\min}$ by inductively attaching carets. We claim that whenever a caret is attached during the inductive construction, the lemma remains true for the constructed tree. 
 	
 	Let $T$ be a subtree of $T_{\La}$ resulting from $T_{\La}^{\min}$ by the attachment of finitely many labeled carets and assume that the lemma holds for $T$. Let $w_1^+$ be a leaf of $T$ which shares a label with some inner vertex $w_2^+$ of $T_{\La}^{\min}$ (for some finite binary words $w_1$ and $w_2$). We let $T'$ be the subtree of $T_{\La}$ which results from $T$ if one attaches a caret to $w_1^+$ and labels the left (resp. right) child of $w_1^+$ by the label of the left (resp. right) child of $w_2^+$.
 	
 	Let $u_1$ and $v_1$ be paths on $T'$ such that $\lab(u_1^+)\equiv  \lab(v_1^+)$. If $w_1$ is not a proper prefix of $u_1$ nor a proper prefix of $v_1$, then $u_1^+$ and $v_1^+$ are vertices of the subtree $T$ and the lemma holds by the induction hypothesis. Thus, we can assume that $w_1$ is a proper prefix of $u_1$. It follows that $u_1\equiv w_10$ or $u_1\equiv w_11$. We consider two possible cases. 
 	
 	(1) $w_1$ is not a proper prefix of $v_1$; i.e., $v_1^+$ is a vertex of $T$. 
 	
 	We assume without loss of generality that $u_1\equiv w_10$, the other case being similar. Since $\lab(v_1^+)\equiv  \lab(u_1^+)\equiv  \lab((w_10)^+)\equiv  \lab((w_20)^+)$ and $v_1^+$ and $(w_20)^+$ are vertices of $T$ (indeed, $w_2^+$ was an inner vertex of $T$), by the induction hypothesis, there is an element $k_1\in K$ with a pair of branches $w_20\rightarrow v_1$. Similarly, since $\lab(w_1^+)\equiv  \lab(w_2^+)$ and $w_1^+,w_2^+$ are vertices of $T$, by the induction hypothesis, there is an element $k_2\in K$ with a pair of branches $w_1\rightarrow w_2$. Thus, $u_1\equiv w_10\rightarrow w_20$ is a pair of branches of $k_2$. It follows that $u_1\rightarrow v_1$ is a pair of branches of $k_2k_1$, as required. 
 	
 	(2) $w_1$ is a proper prefix of $v_1$.
 	
 	We can assume that $u_1\not\equiv v_1$, otherwise the identity element has the pair of branches $u_1\rightarrow v_1$. Hence, either $u_1\equiv w_10$ or $v_1\equiv w_10$. 
 	Then $\lab(v_1^+)\equiv  \lab(u_1^+)\equiv  \lab((w_10)^+)\equiv  \lab((w_20)^+)$. Since $(w_20)^+$ is a vertex of $T$, we have by the previous case that there are elements $k_1,k_2\in K$ with pairs of branches $u_1\rightarrow w_20$ and  $v_1\rightarrow w_20$, respectively. Then $k_1k_2^{-1}$ has a pair of branches $u_1\rightarrow v_1$ as required.
 \end{proof}

\begin{Lemma}\label{cor_aut}
Let $\La$ be a folded-automaton over $\kk$. Let $T_{\La}^{\min}$ be an associated minimal tree. Assume that for all $u$ and $v$ which label paths on $T_{\La}^{\min}$ such that $u^+$ and $v^+$ share a label, there is a diagram $\Delta_{u,v}$ accepted by $\La$ with a pair of branches $u\rightarrow v$.
Let $K$ be the subgroup of $F$ generated by the diagrams $\Delta_{u,v}$ for each pair of paths $u,v$ on $T_{\La}^{\min}$ with $\lab(u^+)\equiv \lab(v^+)$. Then $\Cl(K)$ is the subgroup of $F$ accepted by $\La$. In addition, there is an injective morphism $\phi$ from the core $\La(K)$ with all vertices identified to $\La$ with all vertices identified. 
\end{Lemma}

\begin{proof}
Since $\La$ is a folded-automaton and all diagrams $\Delta_{u,v}$ are accepted by $\La$, all diagrams in $K$ are accepted by $\La$. Thus, by Remark \ref{r:GS} and Theorem \ref{thm:GS}, $\Cl(K)$ is accepted by $\La$. To prove the opposite inclusion, let $\Delta$ be a reduced diagram accepted by $\La$ and let $u_i\rightarrow v_i$, $i=1,\dots,n$, be the pairs of branches of $\Delta$. Let $T_{\La}$ be the labeled binary tree associated with $\La$. Then, by the definition of $T_{\La}$, for every pair of branches $u_i\rightarrow v_i$ of $\Delta$, $u_i$ and $v_i$ label paths on $T_{\La}$ such that $\lab(u_i^{+})=\lab(v_i^{+})$. Then, by Lemma \ref{lem_ez}, for each $i=1,\dots,n$, there is an element $k_i\in K$ such that $\Delta$ and $k_i$ coincide on $[u_i]$. 
Therefore, $\Delta$ is dyadic-piecewise-$K$ and by Theorem \ref{thm:GS}, $\Delta$ belongs to $\Cl(K)$. Thus, $\Cl(K)$ is the subgroup of $F$ accepted by $\La$. 

For the rest of the proof we assume that all vertices of $\La(K)$ and all vertices of $\La$ were identified to a single vertex. To define the injective morphism $\phi$ from $\La(K)$ to $\La$ we only consider edges and cells of the $2$-automata (the unique vertex of $\La(K)$  is mapped to the unique vertex of $\La$). Hence, the non-standard notation should not cause confusion. To prove that there is an injective morphism $\phi$ from $\La(K)$  to $\La$, we let $X=\{\Delta_i:i\in\mathcal I\}$ be a set of reduced diagrams generating $K$. Since every diagram $\Delta_i$ is accepted by $\La(K)$, there is a natural morphism $\psi_{\Delta_i}$ from $\Delta_i$ to $\La(K)$. Moreover, the construction of $\La(K)$ (starting from the bouquet of spheres $\La'$) shows that every edge (resp. cell) of $\La(K)$ is the image under $\psi_{\Delta_i}$ of an edge (resp. cell) of $\Delta_i$, for some $i\in\mathcal I$. 
Since each diagram $\Delta_i$ is accepted by $\La$, there is a morphism $\psi'_{\Delta_i}$ from $\Delta_i$ to $\La$. 

Let $e$ be an edge of $\La(K)$. We choose $i\in\mathcal I$ and an edge $e'$ in $\Delta_i$ such that $\psi_{\Delta_i}(e')=e$ and let $\phi(e)=\psi'_{\Delta_i}(e')$. We define the action of $\phi$ on cells of $\La(K)$ in a similar way. It is easy to see that if $\phi$ is well defined (i.e., does not depend on the choice of preimages of edges and cells  of $\La(K)$ in the diagrams $\Delta_i$, $i\in\mathcal I$), then $\phi$ is a morphism from $\La(K)$ to $\La$. Indeed, the definition of $\phi$ respects adjacency of edges and cells in $\La(K)$. 


We begin by showing that the action of $\phi$ on edges of $\La(K)$ is well defined. 
 Let $e$ be an edge of $\La(K)$ and let $i,j\in \mathcal I$ and $e_1$, $e_2$ be edges of $\Delta_i$ and $\Delta_j$ respectively such that $\psi_{\Delta_i}(e_1)=e$ and $\psi_{\Delta_j}(e_2)=e$. We claim that $\psi'_{\Delta_i}(e_1)=\psi'_{\Delta_j}(e_2)$.

Let $u$ be a (positive or negative) path to $e_1$ on the diagram $\Delta_i$ and let $v$ be a (positive or negative) path to $e_2$ on the diagram $\Delta_j$. Since $\Delta_i$ and $\Delta_j$ are accepted by $\La$, $u$ and $v$ label paths on $\La$, such that $u^+=\psi'_{\Delta_i}(e_1)$ and $v^+=\psi'_{\Delta_j}(e_2)$ in $\La$. Similarly, $u$ and $v$ label paths on $\La(K)$ such that 
$u^+=\psi_{\Delta_i}(e_1)=e=\psi_{\Delta_j}(e_2)=v^+$. Thus, by Lemma \ref{cor_ide}, there is a diagram $\Delta$ in $\Cl(K)$ with a pair of branches $u\rightarrow v$. If $\Delta$ is reduced, then it is accepted by $\La$ and so $u^+=v^+$ on $\La$. Otherwise, there are subpaths $u_1$ and $v_1$ such that $u\equiv u_1s$, $v\equiv v_1s$ and $u_1\rightarrow v_1$ is a pair of branches of the reduced diagram equivalent to $\Delta$. It follows that $u_1^+=v_1^+$ on $\La$, which implies that $u^+=v^+$ on $\La$. Thus,  $\psi'_{\Delta_i}(e_1)=\psi'_{\Delta_j}(e_2)$ as required. Hence, $\phi$ is well defined on edges of $\La(K)$.

Now, let $\pi$ be a positive cell of $\La(K)$. Let $\pi_1$ and $\pi_2$ be cells of diagrams $\Delta_i$ and $\Delta_j$ for $i,j\in\mathcal I$ such that $\psi_{\Delta_i}(\pi_1)=\pi$ and $\psi_{\Delta_j}(\pi_2)=\pi$. Then since $\phi$ is well defined on edges,
$\psi'_{\Delta_i}(\topp(\pi_1))=\phi(\topp(\pi))=\psi'_{\Delta_j}(\topp(\pi_2))$ and 
$\psi'_{\Delta_i}(\bott(\pi_1))=\phi(\bott(\pi))=\psi'_{\Delta_j}(\bott(\pi_2))$. As the top and bottom $1$-paths of a cell in $\La$ determine the cell uniquely, $\psi'_{\Delta_i}(\pi_1)=\psi'_{\Delta_j}(\pi_2)$. Hence, $\phi$ is well defined on cells as well.  

It remains to prove that $\phi$ is injective on edges and cells. We prove injectivity on edges of $\La(K)$. As before, that would imply that $\phi$ is also injective on cells.  Let $e_1$ and $e_2$ be two edges of $\La(K)$ such that $\phi(e_1)=\phi(e_2)$. Let $e_1'$ and $e_2'$ be edges of diagrams $\Delta_i$ and $\Delta_j$, $i,j\in\mathcal I$ such that $\psi_{\Delta_i}(e_1')=e_1$ and $\psi_{\Delta_j}(e_2')=e_2$, Let $u$ and $v$ be paths on $\Delta_i$ and $\Delta_j$ respectively such that $u^+=e_1'$ and $v^+=e_2'$. Then $u$ and $v$ label paths on $\La$ such that $u^+=\phi(e_1)=\phi(e_2)=v^+$. Thus, by Lemma \ref{lem_ez} there is an element $k\in K$ with a pair of branches $u\rightarrow v$. By Lemma \ref{trivial}, that implies that on $\La(K)$ we also have $e_1=u^+=v^+=e_2$. 
\end{proof}

\begin{Corollary}\label{cor_finite}
Let $H$ be a subgroup of $F$ and let $\La(H)$ be the core of $H$. If $\La(H)$ is a finite directed $2$-complex then $\Cl(H)=\Cl(K)$ where $K$ is a finitely generated subgroup of $F$. 
\end{Corollary}

\begin{proof}
Since $\La(H)$ is finite, the tree $T_{\La(H)}^{\min}$ is finite. By Lemma \ref{cor_ide}, for each pair of paths $u$ and $v$ such that $u^+$ and $v^+$ share a label in $T_{\La(H)}^{\min}$ (and thus, $u^+=v^+$ on $\La(H)$), there is a diagram $\Delta_{u,v}$ accepted by $\La(H)$ with a pair of branches $u\rightarrow v$. Then by Lemma \ref{cor_aut}, the group $K$ generated by the finite collection of diagrams $\Delta_{u,v}$ satisfies $\Cl(K)=\Cl(H)$.
\end{proof}

\begin{Lemma}\label{3con}
Let $\La$ be a folded-automaton over $\kk$ with minimal associated tree $T_{\La}^{\min}$. Then $\La$ is a core-automaton if and only if the following conditions are satisfied. 
\begin{enumerate}
\item[(1)] $\La$ is \emph{given} by the tree $T_{\La}^{\min}$, i.e., for every edge (resp. cell) of $\La$ there is a vertex (resp. caret) in $T_{\La}^{\min}$ labeled accordingly. 
\item[(2)] If $x$ is an inner vertex of $T_{\La}^\min$ then $x$ has a descendant $y$ in $T_{\La}^{\min}$ such that $y$ shares a label with some vertex $z\neq y$ in $T_{\La}^{\min}$. 
\item[(3)] For each pair of paths $u,v$ in $T_{\La}^{\min}$ such that $u^+$ and $v^+$ share a label, there is a diagram $\Delta_{u,v}$ such that $\Delta_{u,v}$ is accepted by $\La$ and has a pair of branches $u\rightarrow v$. 
\end{enumerate}
\end{Lemma}

\begin{proof}
If $K$ is a subgroup of $F$ then the minimal tree associated with $\La(K)$ satisfies conditions (1)-(3) from the lemma. Indeed, for condition (1) we note that for any edge $e$ (resp. cell $\pi$) in $\La(K)$ there is a path $u$ on $\La(K)$ terminating on the edge $e$ (resp. ``passing through'' the cell $\pi$). Thus, there is a vertex (resp. caret) in $T_{\La(K)}$ and thus, in $T_{\La(K)}^{\min}$, labeled by the edge $e$ (resp. in accordance with the top and bottom $1$-paths of $\pi$). 

For condition (2), assume by contradiction that $x$ is an inner vertex of $T_{\La(K)}^{\min}$ such that every descendant $y$ of $x$ is identified only with itself. 
Let $e\equiv \lab(x)$ and note that if some vertex labeled $e$ in $T_{\La(K)}$ has a descendant labeled  $e'$ in $T_{\La(K)}$, then $e'$ must label some descendant of $x$ in $T_{\La(K)}^\min$. (Indeed, that can be proved by induction on the length of the simple path in $T_{\La(K)}$ from the vertex labeled $e$ to the vertex labeled $e'$.)
Let $y_1$ and $y_2$ be the left child and right child of $x$ in $T_{\La(K)}^{\min}$, respectively and let $e_1\equiv\lab(y_1)$ and  $e_2\equiv\lab(y_2)$.
Note that in $\La(K)$, $e_1$ and $e_2$ form the bottom $1$-path of a cell $\pi$, whose top edge is $e$. Note also  that $e_1$ and $e_2$ are not bottom edges of any cell in $\La(K)$, other than $\pi$, and that $e_1,e_2$ and $e$ are pairwise distinct. By the construction of $\La(K)$, there is a reduced diagram $\Delta$ in $K$ with a pair of branches $u_1\rightarrow v_1$ such that $u_1$ has an initial subpath $u_1'$ such that ${u_1'}$ labels a path on $\La(K)$ with ${u_1'}^+=e_1$ (indeed, that is true for any edge of $\La(K)$). Since $e_1$ is not on the bottom path of any cell other than $\pi$ and $e_1$ is the left bottom edge of $\pi$, $u_1'$ must be of the form $u_1'\equiv u0$, where $u^+=e$ in $\La(K)$. Let $w_1,\dots,w_r$ be the positive branches of $\Delta$ with prefix $u$. There is a pair of consecutive branches $w_i,w_{i+1}$ such that the edges $w_i^+$, $w_{i+1}^+$ on the bottom $1$-path of $\Delta^+$ form the bottom $1$-path of an $(x,x^2)$-cell $\pi_1$ in $\Delta^+$. Let $e_3= w_i^+$ and $e_4= w_{i+1}^+$ in $\La(K)$ and note that $e_3$ and $e_4$ form the bottom path of a cell $\pi'$ in $\La(K)$. Note also that in $T_{\La(K)}$ the edges $e_3$ and $e_4$ label the vertices $w_i^+$ and $w_{i+1}^+$ which are descendants of the vertex $u^+$ labeled $e$.  Hence,  $e_3$ and $e_4$ label descendants of the vertex $x$ in $T_{\La(K)}^\min$.  By the assumption on all descendants of the vertex $x$ that implies that the edges $e_3$ and $e_4$ are not bottom edges of any cell in $\La(K)$, other than $\pi'$. Let $w'_i$ and $w'_{i+1}$ be the branches of $\Delta^-$ such that $w_i\to w'_i$ and $w_{i+1}\to w_{i+1}'$ are pairs of branches of $\Delta$. Since $\Delta$ is accepted by $\La(K)$, we must have ${w_i'}^+={w_i^+}=e_3$ and ${w_{i+1}'}^+=w_{i+1}^+=e_4$ in $\La(K)$. Since $e_3$ is the left bottom edge of $\pi'$ and is not on the bottom path of any other cell in $\La(K)$, the last digit of $w_i'$ must be $0$. Similarly, the last digit of $w_{i+1}'$ must be $1$. Since the corresponding branches  $w_i$ and $w_{i+1}$ of $\Delta$ also end with  the digits $0$ and $1$, respectively, the diagram $\Delta$ is not reduced, in contradiction to the choice of $\Delta$. Hence, Condition (2) holds for $\La(K)$.
Condition (3) holds for $\La(K)$ by Lemma \ref{cor_ide}.

In the other direction, assume that conditions (1)-(3) are satisfied. Let $K$ be the subgroup of $F$ generated by the diagrams $\Delta_{u,v}$ from condition (3). By Lemma \ref{cor_aut}, the subgroup of $F$ accepted by $\La$ is $\Cl(K)$. Let $\phi$ be the injective morphism from $\La(K)$, with all vertices identified, to $\La$, with all vertices identified, constructed in the proof of Lemma \ref{cor_aut}. It suffices to prove that $\phi$ is surjective. Below, we do not distinguish between edges or cells of $\La(K)$ (resp. $\La$) and edges or cells of the $2$-automaton with identified vertices. In particular, a path on $\La(K)$ (resp. $\La$) can be viewed as a path on the $2$-automaton with identified vertices.  

We observe that $\phi$ induces a mapping from paths on $\La(K)$ to paths on $\La$ which preserves the labels of paths. 
To show that $\phi$ is surjective, it suffices to show that for any path $u$ on $\La$, the word $u$ labels a path on $\La(K)$. In fact, it suffices to consider paths $u$ which correspond to paths on $T_{\La}^{\min}$. Indeed, by condition (1), for every edge $e$ (resp. cell $\pi$) in $\La$, there is a path $u$ in $T_{\La}^{\min}$ such that $e$ is the last edge (resp. $\pi$ is the last cell) visited by the path $u$ on $\La$. Then, if $u$  also labels a path on $\La(K)$ with $u^+=e'$ (resp. $\pi'$ being the last cell through which the path $u$ passes) then $\phi(e')=e$ and $\phi(\pi')=\pi$. 

Thus, let $u$ be a path on $T_{\La}^{\min}$. We write $u\equiv va$ where $a\in\{0,1\}$ is the last digit of $u$. Then $v^+$ is an inner vertex of $T_{\La}^{\min}$. As such, by Condition (2), $v$ can be prolonged to a path $w\equiv vs$ on $T_\La^{\min}$, for a non empty suffix $s$ such that $w^+$ shares a label with some other vertex of $T_{\La}^\min$. It suffices to show that $w$ labels a path on $\La(K)$. Indeed, in that case, since $s\neq\emptyset$, either $v0$ or $v1$ is an initial subpath of $w$, and as such labels a path on $\La(K)$. That implies that both $v0$ and $v1$ label paths on $\La(K)$ and as such, that $u\equiv va$ labels a path on $\La(K)$. 
 
Let ${w'^+}$ be a  vertex of $T_{\La}^\min$, other than $w^+$, which shares a label with $w^+$.
By condition (3) there is a diagram $\Delta=\Delta_{w,w'}$ accepted by $\La$ with a pair of branches $w\rightarrow w'$. By the definition of $K$, we can assume that $\Delta\in K$. Let $\Delta'$ be the reduced diagram equivalent to it. We claim that $w\rightarrow w'$ is a pair of branches of $\Delta'$. Otherwise, there are prefixes $w_1,w_2$ and a non empty common suffix $t$ such that $w\equiv w_1t$, $w'\equiv w_2t$ and $w_1\rightarrow w_2$ is a pair of branches of $\Delta'$. We note that $w_1^+$ and $w_2^+$ are distinct inner vertices of $T_{\La}^{\min}$, as $t$ is not empty.
Since $\Delta'$ is accepted by $\La$, it follows that the vertices $(w_1)^+$ and $(w_2)^+$ of $\T_{\La}^{\min}$ share a label, in contradiction to the definition of $T_{\La}^{\min}$ as the maximal sutree of $T_{\La}$ where distinct inner vertices do not share their label.  Thus, $w\rightarrow w'$ is a pair of branches of $\Delta'$. Since $\Delta'$ is a reduced diagram in $K$ it is accepted by $\La(K)$. In particular, the positive branch $w$ labels a path on $\La(K)$ as required. 
\end{proof}

Let $\La$ be a core-automaton with distinguished edge $p_{\La}$ and let $u$ be a path on $T_{\La}^{\min}$. We denote by $T_u$ the minimal labeled rooted subtree of $T_{\La}^{\min}$ with branch $u$. Assume that $u$ is the $k^{th}$ branch of $T_{u}$ and that $\lab(u^+)\equiv  e$ in $T_u$. Then reading the labels of leaves of $T_u$ from left to right yields a word $p_ueq_u$ in the alphabet $E$ (where $E$ is the set of edges of $\La$) where $|p_u|=k-1$. The pair of words $(p_u,q_u)$ is the \emph{pair of words associated with the path $u$ on $T_{\La}^{\min}$}. We note that $p_ueq_u$ is a $1$-path in $\La$ with $\iota(p_ueq_u)=\iota(p_{\La})$ and $\tau(p_ueq_u)=\tau(p_{\La})$.

\begin{Lemma}\label{cor_sem}
Let $\La$ be a folded automaton and let $T_{\La}^{\min}$ be a minimal associated tree. Let $\P$ be the semigroup presentation associated with the directed $2$-complex $\La$ (as in Section \ref{ss:alg}) and let $S$ be the semigroup given by $\P$. 
Let $u$ and $v$ be a pair of paths in $T_{\La}^{\min}$ such that $\lab(u^+)\equiv\lab(v^+)$ and let $(p_u,q_u)$ and $(p_v,q_v)$ be the associated pairs of words. Then there is a diagram $\Delta_{u,v}$ accepted by $\La$ with a pair of branches $u\rightarrow v$ if and only if $p_u=p_v$ and $q_u=q_v$ in the semigroup $S$. 
\end{Lemma}

\begin{proof}
Let $e$ be the common label of $u^+$ and $v^+$ in $T_{\La}^{\min}$. We let $\Psi_u$ and $\Psi_v$ be the minimal tree-diagrams over $\kk$ with branches $u$ and $v$ respectively. Since $u$ and $v$ label paths on $\La$, $\Psi_u$ and $\Psi_v$ can be naturally viewed as diagrams over $\La$ with $\lab(\topp(\Psi_u))\equiv \lab(\topp(\Psi_v))\equiv p_{\La}$. Clearly, $\lab(\bott(\Psi_u))\equiv p_ueq_u$ and $\lab(\bott(\Psi_v))\equiv p_veq_v$. 

Let $\Delta_{u,v}$ be a diagram in $F$ accepted by $\La$ with a pair of branches $u\rightarrow v$. Then $\Psi_u$ can be viewed as a subdiagram of $\Delta^+$ with $\topp(\Psi_u)=\topp(\Delta^+)$ and $\Psi_v^{-1}$ can be viewed as a subdiagram of $\Delta^-$ with $\bott(\Psi_v^{-1})=\bott(\Delta^-)$. Since $\Delta_{u,v}$ is accepted by $\La$, we can consider it as a diagram over $\La$. If one removes from $\Delta$ the subdiagrams $\Psi_u$ (minus its bottom $1$-path) and $\Psi_v^{-1}$ (minus its top $1$-path) as well as the terminal edge of the positive branch $u$ and the negative branch $v$, one remains with $2$ subdiagrams; a $(p_u,p_v)$-diagram and a $(q_u,q_v)$-diagram. Thus $p_u=p_v$ and $q_u=q_v$ in $S$. 

In the opposite direction, assume that $p_u=p_v$ and $q_u=q_v$ in $S$. Then there is a $(p_u,p_v)$-diagram $\Delta_1$ over $\La$ and a $(q_u,q_v)$-diagram $\Delta_2$ over $\La$ (indeed, $p_u,p_v,q_u,q_v$ are $1$-paths in $\La$). Then $\psi_u\circ (\Delta_1+\varepsilon(e)+\Delta_2)\circ\Psi_v^{-1}$ is a diagram in $\DG(\La,p_{\La})$ with a pair of branches $u\rightarrow v$. 
\end{proof}

In general, Lemma \ref{3con} does not give an algorithm for deciding whether a folded-automaton is a core-automaton since we do not have an algorithm for deciding (in the notation of Lemma \ref{cor_sem}) whether two words $w_1$ and $w_2$ are equivalent in $S$.  However, the condition is often useful.

Let $T$ be a labeled binary tree such that (1) no two inner vertices of $T$ share a label and (2) there are no distinct carets $C_1$ and $C_2$ in $T$ such that the label of each leaf of $C_1$ coincides with the label of the respective leaf of $C_2$.
Then there is a folded automaton $\La$, given by a minimal tree $T_{\La}^{\min}$ such that $T_{\La}^{\min}=T$. Thus, in the following example and the ones in Section \ref{ss:max}, we can talk about a folded automaton given by a minimal tree $T$, as long as $T$ satisfies conditions (1) and (2). 

\begin{Example}
Let $\La$ be a folded-automaton given by the following minimal associated tree $T_{\La}^{\min}$.

\Tree[.$e$ [.$f$ [.$f$ ] [.$h$ [.$k$ ] [.$k$ ]  ] ] [.$g$ [.$h$ ] [.$g$  ] ] ]

\noindent Then $\La$ is not a core-automaton. 
\end{Example}

\begin{proof}
We consider the paths $u\equiv 010$ and $v\equiv 011$ on $T_{\La}^{\min}$. They satisfy $\lab(u^+)\equiv \lab(v^+)\equiv k$. In the notation of Lemma \ref{cor_sem}, we have $(p_u,q_u)\equiv (f,kg)$ and $(p_v,q_v)\equiv (fk,g)$. 
Let $\P=\la e,f,g,h,k\mid fg\rightarrow e,fh\rightarrow f,hg\rightarrow g,kk\rightarrow h \ra$ be the semigroup presentation given by $\La$ and let $S$ be the semigroup with presentation $\mathcal P$. When $\P$ is viewed as a rewriting system, it is confluent and terminating. Since no relation from $\P$ is applicable to $fk$, nor to $f$, they are both reduced words over $\P$. Hence, they are not equivalent in $S$. Thus, by Lemmas \ref{3con} and \ref{cor_sem}, $\La$ is not a core-automaton.
\end{proof}


\subsection{Maximal subgroups of $F$ of infinite index}\label{ss:max}

\subsubsection{Construction of maximal subgroups of $F$ of infinite index which do not fix any number in $(0,1)$}

Recall that by \cite[Proposition 1.4]{Sav}, for every $\alpha\in(0,1)$, the stabilizer of $\alpha$, $H_{\{\alpha\}}$ is a maximal subgroup of $F$. Savchuk asked \cite[Problem 1.5]{Sav} whether all maximal subgroups of $F$ of infinite index are of this form. The core $\La(H)$ of a subgroup $H\le F$ was applied in \cite{GS1} to provide implicit examples of maximal subgroups of $F$ of infinite index which do not fix any number in $(0,1)$. Applying results from this paper, one can use the Stallings $2$-core to provide explicit examples of such maximal subgroups. Indeed, in \cite{GS1}, we showed that $H=\la x_0,x_1x_2x_1^{-1}\ra$ is (1) a proper subgroup of $F$ (2) does not fix any number in $(0,1)$ and (3) is not contained in any finite index subgroup of $F$. The conclusion was that every maximal subgroup of $F$ containing $H$ has infinite index in $F$ and does not fix any number in $(0,1)$. In this section we construct two explicit maximal subgroups of $F$ containing $H$. 

The idea in both constructions is similar. We start with the minimal tree $T_{\La(H)}^{\min}$ associated with $\La(H)$ and ``extend'' it to the minimal tree of some core-automaton $\La$ which accepts a larger group. We apply Lemmas \ref{3con} and \ref{cor_sem} to prove that $\La$ is a core automaton. Then there is a closed subgroup $K\le F$ such that $\La$ coincides with $\La(K)$ up to identification of vertices. Using the algorithm from Section \ref{ss:alg}, one can find a generating set of the closed subgroup $K$. Since $\La(K)$ and $\La$ coincide up to identification of vertices, if $\La$ is chosen properly, then using Theorem \ref{main}, one can show that for any $f\notin K$, we have $\la K,f\ra=F$, which implies that $K$ is a maximal subgroup of $F$. 

The minimal tree $T_{\La(H)}^{\min}$  for $H=\la x_0,x_1x_2x_1^{-1}\ra$ is the following. It was given in Remark \ref{rem:tree} with different labels.

\Tree[.$e$ [.$f$ [.$f$ ] [.$h$ ] ] [.$g$ [.$h$ [.$k$ ] [.$\ell$ [.$\ell$ ] [.$h$ ] ] ] [.$g$ ] ] ]

\begin{Example}\label{ex:1}
The group $K=\la x_0,x_1x_2x_1^{-1},x_1^2x_2^{-1}\ra$ is a maximal subgroup of $F$ which contains $H$. 
\end{Example}

\begin{proof}
Given the group $K$, one could start the proof by finding the core $\La(K)$. We do not do so and instead explain how we constructed the group $K$. 
We let $\La$ be the folded-automaton given by the following minimal tree $T_{\La}^{\min}$. 

\Tree[.$e$ [.$f$ [.$f$ ] [.$h$ ] ] [.$g$ [.$h$ [.$k$ [.$h$ ] [.$k$ ] ] [.$\ell$ [.$\ell$ ] [.$h$ ] ] ] [.$g$ ] ] ]

The tree $T_{\La}^{\min}$ can be viewed as an ``extension'' of $T_{\La(H)}^{\min}$. Indeed, $T_{\La(H)}^{\min}$ is a rooted subtree of $T_{\La}^{\min}$. An immediate implication is that $H$ is accepted by the folded automaton $\La$. Indeed, every diagram accepted by $\La(H)$ is in particular accepted by $\La$.  

Next, we prove that $\La$ is a core-automaton. Indeed, $\La$ satisfies condition (1) from Lemma \ref{3con} by definition. It is easy to verify that it satisfies condition (2). To check condition (3) we use Lemma \ref{cor_sem}. Let $\mathcal P$ be the semigroup presentation given by $\La$ and let $S$ be the semigroup with presentation $\P$. Let $u$ and $v$ be paths on $\T_{\La}^{\min}$ such that $\lab(u^+)\equiv \lab(v^+)$. If $u^+$ and $v^+$ are vertices of the rooted subtree $T_{\La(H)}^{\min}$, then condition (3) of Lemma \ref{3con} holds as $\La(H)$ is a core automaton and any diagram accepted by $\La(H)$ is accepted by $\La$.  
 Thus, we can assume that $u\equiv 1000$ or $u\equiv 1001$. In the first case, $\lab(u^+)\equiv h$. By transitivity arguments, it suffices to check that the condition in Lemma \ref{3con}(3) holds for $u$ and $v\equiv 01$ (we could choose $v$ to be any path to a vertex with label $h$ in $T_{\La(H)}^{\min}$, since condition (3) in Lemma \ref{3con} holds for any pair of such paths in $T_{\La(H)}^{\min}$). In the notation of Lemma \ref{cor_sem}, we have $(p_u,q_u)\equiv (f,k\ell g)$ and $(p_v,q_v)\equiv (f,g)$. Thus, $p_u=p_v$ in the semigroup $S$. Similarly, $q_v\equiv g=hg=(k\ell) g\equiv q_u$ in $S$. The case where $u\equiv 1001$ can be treated in a similar way. Thus, $\La$ is a core automaton. 

Let $K$ be the closed subgroup of $F$ such that $\La$ coincides with $\La(K)$ up to identification of vertices. Then the core $\La(K)$ is given by the above minimal tree. One can apply the simplified algorithm from Section \ref{ss:alg} to find a generating set of the closed group $K$. One gets that $K=\la x_0,x_1x_2x_1^{-1},x_1^2x_2^{-1}\ra$. 

To prove that $K$ is a maximal subgroup of $F$ we use an argument similar to the one in Theorem \ref{B}.
Namely, we prove that for all $f\in F\setminus K$, $\Cl(\la K,f\ra)\supseteq [F,F]$. The proof is identical to the proof of Lemma \ref{eve_ide}, so we do not repeat it. We note that the group $H$, and thus $K$ and any subgroup of $F$ containing it, satisfies conditions (2) and (3) in Theorem \ref{main}. Indeed, $H$ is not contained in any finite index subgroup of $F$ and for $z=x_1x_2x_1^{-1}\in H$, $z$ fixes the dyadic fraction $\alpha=.101$, $z'(\alpha^-)=1$ and $z'(\alpha^+)=2$. Thus, for any $f\notin K$, the group $\la K,f\ra$ satisfies conditions (1)-(3) of Theorem \ref{main}. It follows that for any $f\notin K$, $\la K,f\ra=F$, which implies that $K$ is a maximal subgroup of $F$.
\end{proof}

\begin{Example}\label{ex:2}
The group $K=\la x_0,x_1x_2x_1^{-1},x_1^2x_2x_1^{-3},x_1^3x_2x_1^{-4}\ra$ is a maximal subgroup of $F$ containing $H$. 
\end{Example}

\begin{proof}
We let $\La$ be the folded-automaton given by the following minimal tree $T_{\La}^{\min}$. 

\Tree[.$e$ [.$f$ [.$f$ ] [.$h$ ] ] [.$g$ [.$h$ [.$k$ [.$a$ [.$a$ ] [.$c$ [.$c$ ] [.$c$ ] ] ] [.$b$ [.$c$ ] [.$b$ ] ] ] [.$\ell$ [.$\ell$ ] [.$h$ ] ] ] [.$g$ ] ] ]

 $T_{\La(H)}^{\min}$ is a rooted subtree of $T_{\La}^{\min}$. Thus, $H$ is accepted by $\La$. 

The proof that $\La$ is a core automaton can be done, as in Example \ref{ex:1}, by considering all pairs of paths $u$ and $v$ on $T_{\La}^{\min}$ with $\lab(u^+)\equiv \lab(v^+)$. Alternatively, we note that to construct the tree $T_{\La}^{\min}$, we started with minimal trees $T_{\La(H)}^{\min}$ and $T_{\La(F)}^{\min}$ associated with the cores $\La(H)$ and $\La(F)$ respectively (see Remark \ref{coreF2}), and identified the root of $T_{\La(F)}^{\min}$ with the leaf $(100)^+$ of $T_{\La(H)}^{\min}$ (which was not identified with any inner vertex of $T_{\La(H)}^{\min}$). Thus, if $u$ and $v$ are paths on $\La$ with $\lab(u^+)\equiv \lab(v^+)$, then $u^+$ and $v^+$ both belong to the rooted subtree $T_{\La(H)}^{\min}$ or both belong to the subtree of $T_{\La}^{\min}$ rooted at $(100)^+$. In the first (resp. second) case, condition (3) of Lemma \ref{3con} holds for $u$ and $v$ since the condition holds in $T_{\La(H)}^{\min}$ (resp. $T_{\La(F)}^{\min}).$

An application of the algorithm in Section \ref{ss:alg} shows that the closed group $K$ accepted by $\La$ is the one given in the example. 
As in Example \ref{ex:1}, any subgroup of $F$ containing $K$ satisfies conditions (2) and (3) in Theorem \ref{main}. Thus, the following lemma completes the proof of maximality of $K$ in $F$. 

\begin{Lemma}
Let $f\in F\setminus K$. Then $\Cl(\la K,f\ra)\supseteq [F,F]$. 
\end{Lemma}
 
\begin{proof}
We let $M=\la K,f\ra$. The proof is similar, but more complicated than the proof of Lemma \ref{eve_ide}. The core $\La(K)$ is described by the minimal tree $T_{\La}^{\min}$ given above. We note that every edge in $\La(K)$ is the top edge of some positive cell and that $\La(K)$ has a unique left boundary edge and a unique right boundary edge. Thus, as in the proof of Lemma \ref{eve_ide}, there is a surjective morphism $\phi$ from $\La(K)$ to $\La(M)$. In addition, for some pair of distinct inner edges $e_1\neq e_2$ in $\La(K)$, we must have $\phi(e_1)=\phi(e_2)$ in $\La(M)$. The inner edges of $\La(K)$ are $h,\ell,k,a,b,c$. To get the result, we have to go over all choices of inner edges $e_1\neq e_2$ in $\La(K)$ and show that if $\phi(e_1)=\phi(e_2)$ in $\La(M)$, then $\phi(h)=\phi(\ell)=\phi(k)=\phi(a)=\phi(b)=\phi(c)$, so that there is only one inner edge in $\La(M)$. As there are $15$ possible choices of pairs $e_1\neq e_2$, we describe only some of them. The other cases can be verified in a similar way. 

First, assume that $e_1=h$ and $e_2=k$, so that $\phi(h)=\phi(k)$. As in the proof of Lemma \ref{eve_ide}, we consider the implications to the image of other inner edges in $\La(K)$. In $T_{\La(K)}^{\min}$, a vertex labeled $h$ has a left child labeled $k$ which has a left child labeled $a$. Thus, the identification of $h$ and $k$ under $\phi$ implies that $\phi(k)=\phi(a)$. Considering carets in $T_{\La(K)}^{\min}$ with top vertices labeled $h$, $k$ and $a$, we see that the right children must all be identified under $\phi$. Thus, $\phi(\ell)=\phi(b)=\phi(c)$. Again, considering right children, we get that $\phi(h)=\phi(b)=\phi(c)$. All together, we get that $\phi(h)=\phi(\ell)=\phi(k)=\phi(a)=\phi(b)=\phi(c)$ as required. In a similar manner, one can show that if $e_1=h$ and $e_2$ is any other inner edge of $\La(K)$, then $\La(M)$ has a unique inner edge. Indeed, similar arguments also work when $e_1=\ell$ and $e_2\neq \ell$. Thus, we only have to consider the case where $e_1,e_2\in\{k,a,b,c\}$. We consider the case $e_1=k$ and $e_2=a$, the other cases being similar. 

By assumption $\phi(k)=\phi(a)$. That implies that the right children of $\phi(a)$ and $\phi(k)$ satisfy $\phi(b)=\phi(c)$. We claim that there must be at least one more pair of edges  of $\La(K)$ (other than $\{k,a\}$ and $\{b,c\}$) with the same image under $\phi$. Otherwise, the core $\La(M)$ is given by the following minimal tree. 

\Tree[.$e$ [.$f$ [.$f$ ] [.$h$ ] ] [.$g$ [.$h$ [.$k$ [.$k$ ] [.$b$ [.$b$ ] [.$b$ ] ] ] [.$\ell$ [.$\ell$ ] [.$h$ ] ] ] [.$g$ ] ] ]

\noindent To get a contradiction it suffices to note that the above tree is not a minimal tree associated with a core-automaton. Indeed, consider the paths $u\equiv 1001$ and $v\equiv 10010$ on $T_{\La(M)}^{\min}$. Then, $\lab(u^+)\equiv \lab(v^+)\equiv b$. The pairs of words associated with the paths $u$ and $v$ in $T_{\La(M)}^{\min}$ are $(p_u,q_u)\equiv(fk,\ell g)$ and $(p_v,q_v)\equiv(fk,b\ell g)$. We claim that $q_u\neq q_v$ in the semigroup $S$ with presentation $\P$ associated with $\La(M)$. Indeed, if $\ell g=b\ell g$ in $S$, then, since $\ell g$ is a $1$-path in $\La(M)$, there is an $(\ell g,b\ell g)$-diagram $\Delta$ over $\La(M)$. The diagram $\Delta$ implies that $\iota(\ell g)=\iota(b\ell g)$ in $\La(M)$. In particular, the vertices $\iota(\ell)$ and $\iota(b)$ coincide in $\La(M)$, in contradiction to $\ell$ and $b$ not being in the same connected component of $\Gamma(\La(M))$ (see Proposition \ref{inner_vertex}). 

Hence, at least one more pair of inner edges is identified in the transition from $\La(K)$ to $\La(M)$. If $h$ or $\ell$ is one of the edges in the pair, we are done. Thus, we can assume that $\phi(k)=\phi(b)$. We claim that in this case, again, there must be another pair of identified edges. Otherwise, $\La(M)$ is described by the following associated minimal tree. 

\Tree[.$e$ [.$f$ [.$f$ ] [.$h$ ] ] [.$g$ [.$h$ [.$k$ [.$k$ ] [.$k$ ] ] [.$\ell$ [.$\ell$ ] [.$h$ ] ] ] [.$g$ ] ] ]

As above, one can show that $T_{\La(M)}^{\min}$ is not associated with a core-automaton to get the required contradiction by considering the pair of paths $u\equiv 100$ and $v\equiv 1000$ on $T_{\La(M)}^{\min}$. Therefore, at least one more pair of inner edges of ${\La(K)}$ is identified in 
${\La(M)}$. In particular, either $h$ or $\ell$ is identified with some other edge of $\La(K)$. As noted above, that implies that $\La(M)$ has a unique inner edge and completes the proof of the lemma. 
\end{proof}
\end{proof}

\subsubsection{A maximal subgroup of $F$ which acts transitively on the set $\mathcal D$}


The following can be viewed as a strong counter example to Savchuk's problem \cite[Problem 1.5]{Sav}.

\begin{Proposition}\label{lem:max}
The group $M=\la x_0,x_1x_2x_1^{-3},x_1x_2x_3x_2^{-3}x_1^{-1}\ra$ is a maximal subgroup of infinite index in $F$ which acts transitively on the set of finite dyadic fractions $\mathcal D$. 
\end{Proposition}

\begin{proof}
The core $\La(M)$ is given by the following associated minimal tree.

\Tree[.$e$ [.$f$ [.$f$ ] [.$h$ ] ] [.$g$ [.$h$ [.$k$ [.$m$ [.$k$ ] [.$k$ ] ] [.$h$ ] ] [.$h$ ] ] [.$g$ ] ] ]

The group $M$ was chosen originally to be the subgroup of $F$ accepted by this core (the algorithm from Section \ref{ss:alg} was used to find the given generating set). Hence $M=\Cl(M)$. 

The proof that $M$ is a maximal subgroup of $F$ of infinite index is almost identical to the proof of maximality of $K$ from Example \ref{ex:1}. Indeed, the proof that $M[F,F]=F$ and that for each $f\notin M$, we have $[F,F]\le \Cl(\la M,f\ra)$ is identical to the proof in Example \ref{ex:1}. To prove that for any $f\notin M$ there is an element $z\in \la M,f\ra$ which fixes a finite dyadic fraction $\alpha\in (0,1)$ and such that $z'(\alpha^+)=2$ and $z'(\alpha^-)=1$, we observe that the element 
$y=x_1x_2x_1^{-3} (x_1x_2x_3x_2^{-3}x_1^{-1})^{-1}\in M$ fixes the finite dyadic fraction $\beta=.111$ and has slope $y'(\beta^-)=2$ and $y'(\beta^+)=1$. Since for each $f\notin F$, $[F,F]\le \Cl(\la M,f\ra)$, one can apply Lemma \ref{1221} to get the existence of an element $z\in\la M,f\ra$ as described. 

It remains to note that every edge in $\La(M)$ is the top edge of some positive cell and that the graph $\Gamma(M)$ given in Figure \ref{pic1} is connected and apply Theorem \ref{thm:tra}.

\begin{figure}[h]
	\centering
	\includegraphics[width=0.7\columnwidth]{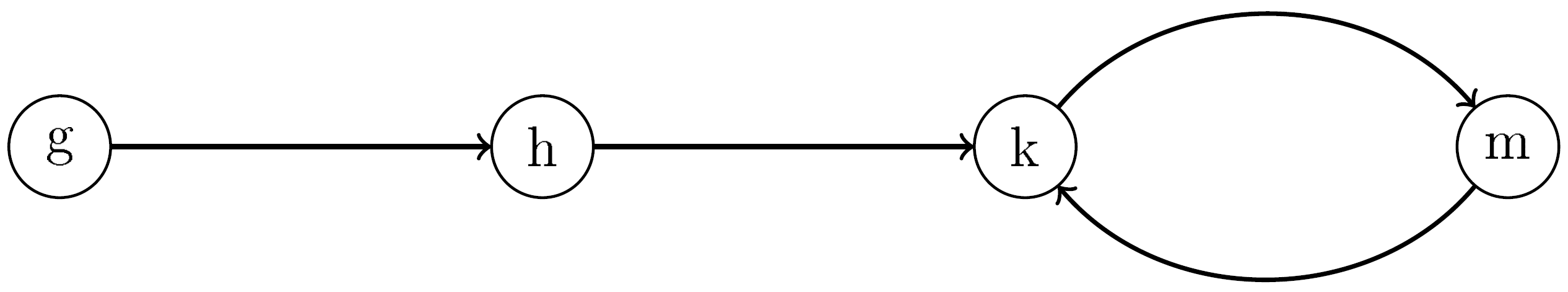}
	\caption{The graph $\Gamma(M)$}
	\label{pic1}
\end{figure}


\end{proof}


\section{Solvable subgroups of Thompson's group $F$}

\subsection{On the closure of solvable subgroups}\label{sec:sg}

In this section we consider the closure of solvable subgroups of $F$ and prove the following theorem. The theorem follows from results of \cite{Bl1,Bl2,BBH}. We prove each part separately below. 

\begin{Theorem}\label{sol}
Let $H$ be a solvable subgroup of $F$ of derived length $n$. Then the following assertions hold. 
\begin{enumerate}
\item[(1)] The action of $H$ on the set of finite dyadic fractions $\mathcal D$ has infinitely many orbits.
\item[(2)] $\Cl(H)$ is solvable of derived length $n$. 
\item[(3)] If $H$ is finitely generated then $\Cl(H)$ is finitely generated. 
\end{enumerate}
\end{Theorem}

Part (1) of Theorem \ref{sol} should be viewed in contrast with Theorem \ref{B}, where an elementary amenable subgroup $B\le F$ which acts  transitively on the set of finite dyadic fractions $\mathcal D$ was constructed. Part (2) of Theorem \ref{sol} should be viewed in contrast with the following example. 

\begin{Example}\label{B_1}
Let $B_1$ be the subgroup of $F$ generated by 
$x=x_0x_1x_2x_3x_5^2(x_0x_1x_2x_4^3)^{-1}$ and $y=x_0^3x_2x_6(x_0x_1^2x_3x_5^2x_7)^{-1}$. Then $B_1$ is a copy of the Brin-Navas group. In particular, it is elementary amenable. The closure $\Cl(B_1)$ contains a copy of Thompson's group $F$. In particular, it is not elementary amenable \cite{CFP}.
\end{Example}

\begin{proof}
The group $B_1$ is a copy of the Brin-Navas group, as realized in \cite[Section 1.2.3]{Bl3}.
To prove that the closure of $B_1$ contains a copy of $F$ we consider the core $\La(B_1)$. 
The following is a minimal tree associated with $\La(B_1)$. 

\Tree[.$e$ [.$f$ [.$m$ [.$f$ ] [.$h$ ] ] [.$\ell$ [.$\ell$ ] [.$a$ ] ] ] [.$g$ [.$h$ [.$a$ [.$h$ ] [.$\ell$ ] ] [.$h$ ] ] [.$k$ [.$\ell$ ] [.$g$ ] ] ] ]

Let $\P$ be the semigroup presentation associated with $\La(B_1)$ and let $p=p_{\La(B_1)}=e$ be the distinguished edge of $\La(B_1)$. Let $S$ be the semigroup with presentation $\P$.  We claim that there is a word $w$ over the alphabet $E$ of $\mathcal P$, such that (1) $w$ divides $p$ in $S$ (i.e., there exist words $a$ and $b$ over $E$ such that $awb$ is equal to $p$ in $S$) and (2) $w=ww$ in $S$. 
By \cite[Theorem 25]{GuSa99}, that would imply that the diagram group $\DG(\La(B_1),p)\cong \Cl(B_1)$, contains a copy of Thompson's group $F$. Let $w$ be the edge $a$ of $\La(B_1)$. Since $w$ is a $1$-path in $\La(B_1)$, it follows from Proposition \ref{prop:right}(1) that $w$ divides $p$ in $S$. (Indeed, one can consider a word $wq$ where $q$ is a $1$-path in $\La(B_1)$ with $\iota(q)=\tau(w)$ and $\tau(q)=\tau(\La(B_1))$.) It remains to note that $a=h\ell=h(\ell a)=(h\ell)a=aa$ in $S$. 
\end{proof}

\begin{Remark}
Using similar arguments to those in the proof of Theorem \ref{B}, one can show that the closure of $B_1$ is a maximal subgroup of $B_1[F,F]$.
\end{Remark}

In \cite{Bl1}, solvable subgroups of $\PL_o(I)$ are characterized by the towers associated with them. The following definition uses different notation than the one in \cite{Bl1}.

\begin{Definition}\label{def:tower}
Let $G$ be a subgroup of $\PL_o(I)$. A \emph{tower} in $G$ is a set of distinct intervals $T=\{(a_i,b_i)\mid i\in \mathcal I\}$ such that for each $i$, $(a_i,b_i)$ is an orbital of some function in $H$ and such that $T$ is totally ordered with respect to inclusion. 

The cardinality of a tower is said to be the \emph{height of the tower}.
The supremum of heights of all towers in $G$ is called the \emph{depth} of $G$. If $G$ is the trivial subgroup, we say that $G$ has depth $0$. 
\end{Definition}

\begin{Theorem}\cite[Theorem 1.1]{Bl1}\label{thm:sol}
Let $G$ be a subgroup of $\PL_o(I)$. Then $G$ is solvable of derived length $n$ if and only if the depth of $G$ is $n$.
\end{Theorem}

\begin{Definition}[\cite{Bl1}]
Let $G\le \PL_o(I)$. The group $G$ \emph{admits a transition chain} if there are elements $g_1,g_2$ in $G$ with orbitals $(a,b)$ and $(c,d)$ respectively, such that $a<c<b<d$.
\end{Definition}

The following lemma follows immediately from Theorem 1.1, Lemma 1.4 and Remark 4.1 of \cite{Bl1}. 

\begin{Lemma}[\cite{Bl1}]\label{Bl1}
Let $G$ be a solvable subgroup of $\PL_o(I)$. Then $G$ does not admit transition chains. Moreover, if $(a,b)$ and $(c,d)$ are distinct orbitals of elements $g,h\in G$ such that $(a,b)\cap(c,d)\neq\emptyset$, then either $a<c<d<b$ or $c<a<b<d$. 
\end{Lemma}

We apply Theorem \ref{thm:sol} and Lemma \ref{Bl1} to prove the following lemma. This is part (1) of Theorem \ref{sol}.

\begin{Lemma}\label{sol_inf}
Let $H$ be a solvable subgroup of $F$. Then the action of $H$ on the set of finite dyadic fractions 
$\mathcal D$ has infinitely many orbits.
\end{Lemma}

\begin{proof}
Let $n$ be the derived length of $H$. By Theorem \ref{thm:sol}, $H$ has a maximal tower $T=\{(a_i,b_i):i=1,\dots,n\}$ of height $n$. We denote by $(a,b)$ the smallest orbital in the tower. Let $h\in H$ be a function with orbital $(a,b)$. Let $f$ in $H$ be a function which maps a number $x\in(a,b)$ to a number $y\neq x$ in this interval. We claim that $(a,b)$ is an orbital of $f$ and in particular that $f$ fixes $a$ and $b$.
Indeed, let $(c,d)$ be the orbital of $f$ containing $x$ and assume by contradiction that $(c,d)$ does not coincide with $(a,b)$. Replacing $f$ by $f^{-1}$ if necessary, we can assume that $(c,d)$ is a push-up orbital of $f$. By Lemma \ref{Bl1}, either $(c,d)$ is strictly contained in $(a,b)$, or we have $c<a<b<d$. Since $T$ is a maximal tower of $H$, we must have $c<a<b<d$. In particular, $a\in (c,d)$. Since $a<x$ and $(c,d)$ is a push-up orbital of $f$ we have $a<f(a)<f(x)=y<b$. 
We consider the element $h^f$. It has an orbital $(f(a),f(b))$. Since $f(a)\in (a,b)$, by Lemma \ref{Bl1}, the orbital $(f(a),f(b))$ must be strictly contained in $(a,b)$. That gives a contradiction to the tower $T$ being a tower of $H$ of maximal height. Thus, $(a,b)$ is an orbital of $f$.

Let $K$ be the subgroup of $H$ of all elements which fix the numbers $a$ and $b$. 
One can naturally map $K$ onto a subgroup $K'$ of $\PL_o(I)$ where all functions have support in $[a,b]$, by sending a function $k\in K$ to the function $k'\in K'$ which coincides with $k$ on $[a,b]$ and is the identity elsewhere. 
Let $\alpha_1,\alpha_2\in (a,b)$ be two finite dyadic fractions. The above argument shows that $\alpha_1$ and $\alpha_2$ belong to the same orbit of the action of $H$ on $\mathcal D$ if and only if they belong to the same orbit of the action of $K'$ on $\mathcal D$. The maximality of the tower $T$ implies that any non-trivial element in $K'$ has orbital $(a,b)$.  

We claim that $K'$ is cyclic. Indeed, let $k_1\in K'$ be an element with slope $2^m$ at $a^+$ for minimal $m>0$. If $k_2\in K'$, has slope $2^{m_2}$ at $a^+$ then $m$ divides $m_2$. Hence, $k'=k_1^{\frac{m_2}{m}}k_2^{-1}$ has slope $1$ at $a^+$. It follows that $k'$ fixes a right neighborhood of $a$ and as such $(a,b)$ is not an orbital of $k'$. Thus, $k'$ is the trivial element of $K'$ 
 which implies that $K'=\la k_1\ra$. It is obvious that the action of $\la k_1\ra$ on $(a,b)\cap\mathcal D$ has infinitely many orbits. Thus, the action of $H$ on $\mathcal D$ has infinitely many orbits. 
\end{proof}

Part (2) of Theorem \ref{sol} follows immediately from results of Bleak \cite{Bl2}. 
For a subgroup $G$ of $\PL_o(I)$, Bleak defined the split group $S(G)$ of $G$. We define $S(G)$ using different terminology to emphasize the relation between $S(G)$ and the closure $\Cl(G)$ in case $G$ is a subgroup of $F$. 

\begin{Definition}\label{def:com3}
Let $f$ be a function in $\PL_o(I)$. If $f$ fixes a number $\alpha\in (0,1)$, then the following functions $f_1,f_2\in F$ are called \emph{general components} of the function $f$, or \emph{general components of $f$ at $\alpha$}.  
\[
   f_1(t) =
  \begin{cases}
   f(t) &  \hbox{ if }  t\in[0,\alpha] \\
    t       & \hbox{ if } t\in[\alpha,1]\
  \end{cases}  	\qquad	
   f_2(t) =
  \begin{cases}
   t &  \hbox{ if }  t\in[0,\alpha] \\
    f(t)       & \hbox{ if } t\in[\alpha,1]\
  \end{cases} 
\]
\end{Definition}

Note that the only difference between Definition \ref{com2} and Definition \ref{def:com3} for functions $f\in F$, is that general components are taken with respect to fixed points which are not necessarily finite dyadic. 

\begin{Definition}[Bleak \cite{Bl2}]
Let $G\le \PL_o(I)$. The \emph{split group} of $G$ is the subgroup of $\PL_o(I)$ generated by all elements of $G$ and all general components of elements in $G$. 
\end{Definition}

\begin{Remark}
A simple adaptation of the proof of Lemma \ref{clo_dya} shows that $S(G)$ can be defined as the subgroup of $\PL_o(I)$ of all piecewise-$G$ functions. It follows that $S(S(G))=S(G)$. We also note that the orbits of the action of $G$ on the interval $[0,1]$ coincide with the orbits of the action of $S(G)$. 
\end{Remark}

\begin{Theorem}\label{Bl2}\cite[Lemma 4.5, Corollary 4.6]{Bl2}
Let $G$ be a solvable subgroup of $\PL_o(I)$. Then $(a,b)$ is an orbital of an element in $G$ if and only if it is an orbital of an element in $S(G)$. In particular, by Theorem \ref{thm:sol}, the derived length of $G$ is equal to the derived length of the split group $S(G)$.
\end{Theorem}

For a subgroup $H\le F$, by Theorem \ref{thm:GS}, we have $H\le \Cl(H)\le S(H)$. Thus, Theorem \ref{Bl2} implies the following. 

\begin{Corollary}[Theorem \ref{sol}(2)]\label{cor:sol_cor}
Let $H\le F$ be a solvable subgroup of $F$. Then the closure $\Cl(H)$ is solvable of the same derived length. 
\end{Corollary}

To prove Theorem \ref{sol}(3) we need the following lemma. 

\begin{Lemma}\cite[Lemma 3.1]{BBH}\label{BBH}
Let $G\le \PL_o(I)$ be a finitely generated subgroup. Then the set of breakpoints of elements in $G$ intersects finitely many orbits of the action of $G$ on $[0,1]$. 
\end{Lemma}

\begin{Proposition}\label{fin_gen_SG}
Let $G\le \PL_o(I)$ be a finitely generated solvable subgroup. Then the split group $S(G)$ is finitely generated. 
\end{Proposition}

\begin{proof}

Assume that $G$ is solvable of derived length $n$. By Theorem \ref{Bl2}, orbitals of elements in $G$ coincide with orbitals of elements in $S(G)$. We say that an orbital $(a,b)$ of some function in $G$ (equiv. $S(G)$) is \emph{minimal} in $G$ if there is no function in $G$ with an orbital strictly contained in $(a,b)$. Note that if $(a,b)$ is a minimal orbital in $G$ and $g\in G$ then 
$(g(a),g(b))$ is also a minimal orbital in $G$. Note also that by Theorem \ref{thm:sol} every orbital of an element in $G$ contains some minimal orbital. In addition, by Lemma \ref{Bl1}, the set of orbitals of elements in $G$ which contain a given minimal orbital forms a tower in $G$. Thus, by Theorem \ref{thm:sol},  a minimal orbital in $G$ is contained in at most $n$ orbitals of elements of $G$.
 
The proof of \cite[Lemma 4.7]{Bl2} shows that the split group $S(G)$ is generated by a set of functions $B=\{g_i\mid i\in\mathcal I\}$ such that for each $i$, $g_i$ has a single orbital $(a_i,b_i)$ and such that the following conditions hold.
\begin{enumerate}
\item[(1)] For every $i\neq j$ and every $g\in G$ we have $(g(a_i),g(b_i))\neq (a_j,b_j)$. In particular, for every $i\neq j$, $(a_i,b_i)\neq(a_j,b_j)$. 
\item[(2)] If $(a_i,b_i)$ is not a minimal orbital in $G$ then it contains an orbital $(a_j,b_j)$ for some $j$, such that $(a_j,b_j)$ is minimal in $G$. 
\end{enumerate}

We claim that the generating set $B$ must be finite. Otherwise, we let $\mathcal J\subseteq I$ be the subset of all indexes $j$ such that the orbitals $(a_j,b_j)$ are minimal orbitals in $G$. Let $C$ be the set of orbitals $(a_j,b_j)$ for $j\in \mathcal J$.
We claim that $C$ must be infinite. Indeed, by condition (2) for the set $B$ one can map each generator $g_i$ in $B$ to some orbital $(a_j,b_j)$ in $C$ such that the orbital $(a_j,b_j)$ is contained in $(a_i,b_i)$. By condition (1), the preimage of each orbital in $C$ is a subset of size at most $n$ of $B$. Thus, $C$ is infinite.

By Theorem \ref{Bl2}, each orbital $(a_j,b_j)$ in $C$ is an orbital of some element $h_j$ in $G$. 
We note that $h_j$ must have a breakpoint $x_j\in(a_j,b_j)$. Indeed, a function in $\PL_o(I)$ cannot be linear on any of its orbitals. 
We let $E$ be the set of breakpoints $x_j$, $j\in\mathcal J$. We note that if $j\neq k$ in $\mathcal J$, then $x_j\neq x_k$. Indeed, $(a_j,b_j)$ and $(a_k,b_k)$ are distinct minimal orbitals in $G$. Hence by Lemma \ref{Bl1}, they are disjoint. Moreover, $x_j$ and $x_k$ do not belong to the same orbit of the action of $G$ on $[0,1]$. Indeed, if $g\in G$, then $g(x_j)\in (g(a_j),g(b_j))\neq (a_k,b_k)$ by condition (1). Since $(g(a_j),g(b_j))$ and $(a_k,b_k)$ are minimal orbitals in $G$, they are disjoint, and so $g(x_j)\neq x_k$. Thus, the set $E$ is a set of breakpoints of elements of $G$ which intersects infinitely many orbits of the action of $G$ on $[0,1]$. We get a contradiction to the finite generation of $G$ by Lemma \ref{BBH}.
\end{proof}

The proof of Theorem \ref{sol}(3) requires a slight adaptation of the proof of Proposition \ref{fin_gen_SG}. 

\begin{Definition}
Let $f\in F$ and assume that $f\in F$ fixes finite dyadic fractions $a,b\in[0,1]$ such that $a<b$. We say that $(a,b)$ is a \emph{dyadic-orbital} of $f$ if $f$ does not fix any finite dyadic fraction in $(a,b)$. 
\end{Definition}

By \cite[Corollary 2.5]{GS2} (See also \cite{Sav}), if $f$ fixes an irrational number $x$ in $(0,1)$ then $f$ fixes an open neighborhood of $x$. Hence, if $(a,b)$ is a dyadic-orbital of $f$, which is not an orbital of $f$, then there are finitely many rational non-dyadic numbers $x_1<\dots<x_n$ in $(a,b)$ such that $(a,x_1),(x_1,x_2),\dots,(x_{n-1},x_n),(x_n,b)$ are orbitals of $f$. Thus, a dyadic-orbital of $f\in F$ can be defined as a minimal sequence of consecutive ``adjacent'' orbitals $(a_1,a_2),(a_2,a_3)\dots,(a_k,a_{k+1})$ such that $a_1,a_{k+1}$ are finite dyadic.

\begin{proof}[Proof of Theorem \ref{sol}(3)]
The proof of the theorem is almost identical to the proof of Proposition \ref{fin_gen_SG}. Indeed, one only has to replace every occurrence of the word ``orbital'' in the proof of Proposition \ref{fin_gen_SG} with the term ``dyadic-orbital'' and replace the group $S(G)$ with the group $\Cl(G)$.
A detailed proof requires an adaptation of Lemma \ref{Bl1} and the first statement of Theorem \ref{Bl2} to the case where $G$ is a subgroup of $F$ and dyadic-orbitals are considered instead of orbitals. It also requires a proof that the conclusion of \cite[Lemma 3.12]{Bl2} (which is used in the proof of \cite[Lemma 4.7]{Bl2}) holds for a solvable subgroup $G\le F$ with ``orbital'' replaced by ``dyadic-orbital''. The adaptations are not hard but as we do not wish to repeat here paper \cite{Bl2}, we omit them.
\end{proof}

\subsection{Characterization of solvable subgroups $H\le F$ in terms of the core $\La(H)$}\label{ss:sol_alg}

Recall that if $(a,b)$ is a push-down orbital of a function $g\in\PL_o(I)$ then the slope $g'(a^+)<1$. By \cite[Corollary 2.5]{GS2}, if $(a,b)$ is an orbital of a function $f\in F$, then $a$ and $b$ must be rational numbers. The proof of the following lemma is similar to the proof of Lemma \ref{4parts} and is left as an exercise to the reader. 

\begin{Lemma}\label{branch}
Let $f\in F$ be a function with a push-down orbital $(a,b)$. Then the reduced diagram $\Delta$ of $f$ has a pair of branches $u\rightarrow us$ for some finite binary words $u$ and $s$ such that $s$ contains the digit $0$ and $a=.us^{\mathbb{N}}$. 
\end{Lemma}

Recall (Section \ref{sec:paths}) that a path on the core $\La(H)$ always starts from the distinguished edge $p_{\La(H)}=q_{\La(H)}$ whereas a trail on $\La(H)$ can start from any edge. We rarely distinguish between a trail and its label, but when describing a trail we are careful to mention its initial edge.

\begin{Definition}\label{semi}
Let $H\le F$. An edge $e$ of the core $\La(H)$ is a \emph{semi-periodic edge} if there is a trail $s$ with initial edge $e$ and terminal edge $e$ such that the label of $s$ contains the digit $0$. 
\end{Definition}

\begin{Definition}\label{opt_tra}
Let $H\le F$ and let $e$ be a semi-periodic edge of $\La(H)$. We let $\mathcal P_e$ be the set of all non-trivial trails from $e$ to itself which do not visit any edge of the core more than twice. By assumption, $\mathcal P_e$ is not empty. We consider the set $\mathcal P_e$ ordered by the lexicographic order on the set of non-empty finite binary words $\{0,1\}^+$ 
(where $0$ is taken to be smaller than $1$). If there is a minimal trail in $\mathcal P_e$, we say that $e$ is a \emph{periodic} edge and denote the minimal trail in $\mathcal P_e$ by $s_e$. We call $s_e$ the \emph{optimal trail from $e$ to itself} and note that $s_e$ must contain the digit $0$ by the definition of periodic edges. 
\end{Definition}


%

\begin{Proposition}\label{prop}
Let $e_1$ and $e_2$ be edges of $\La(H)$. Assume that $e_1$ is periodic and let $s_{e_1}$ be the optimal trail from $e_1$ to itself. Then the following conditions are equivalent.
\begin{enumerate}
\item[(1)] There is a trail $q$ from $e_1$ to $e_2$ on $\La(H)$ and a trail $s$ from $e_1$ to itself such that $.q>.s^{\mathbb{N}}$.
\item[(2)] There is a trail $q$ from $e_1$ to $e_2$ such that for some finite binary words $u,w_1,w_2$, $$(*)\ \ q\equiv u1w_1\ \mbox{ and }\ s_{e_1}\equiv u0w_2,$$
and such that the sub-trail of $q$ given by the suffix $w_1$ is a simple trail; i.e., it does not visit any edge in $\La(H)$ more than once.
\item[(3)] There is a trail $q$ from $e_1$ to $e_2$ such that $.q>.s_{e_1}^{\mathbb{N}}$.
\end{enumerate}
\end{Proposition}

\begin{proof}
It is clear that (2) implies (3). Indeed, the trail $q$ from (2) satisfies $.q>.s_{e_1}^{\mathbb{N}}$. It is also obvious that (3) implies (1). To prove that (1) implies (2), we note that it suffices to prove that there is a trail $q$ from $e_1$ to $e_2$ which satisfies $(*)$. Indeed, if $q$ satisfies $(*)$, but 
 the sub-trail of $q$ given by the suffix $w_1$ is not a  simple trail, we can replace $w_1$ by a simple trail with the same initial and terminal vertices and $(*)$ would still hold. 
 
 Hence, let $q$ and $s$ be trails as described in condition (1). We prove that $q$ and $s$ can be modified so that $q$ would satisfy $(*)$. First, we remove the longest prefix of $q$ which is a power of $s$. Note that if $q\equiv s^nq'$, then since $s^n$ labels a trail from $e_1$ to itself, $q'$ labels a trail from $e_1$ to $e_2$. In addition, since $.q=.s^nq'>.s^{\mathbb{N}}$, we have $.q'>.s^{\mathbb{N}}$. Hence, after removing from $q$  its longest prefix  which is a power of $s$, we still have  $.q>.s^{\mathbb{N}}$.
 
 

 Now, we note that $(*)$ is satisfied for $q$ and $s_{e_1}$ replaced by $s$ (below we shall say shortly that $(*)$ holds for $q$ and $s$). Indeed, consider the binary words $q$ and $s^{\mathbb{N}}$ as paths on the complete infinite binary tree. We let $u$ be the longest common prefix of $q$ and $s^{\mathbb{N}}$. Since $.q>.s^{\mathbb{N}}$, we have that $q\equiv u1w_1$ and $s^{\mathbb N}\equiv u0\omega_2$ where $w_1$ is a possibly empty suffix of $q$ and $\omega_2$ is an infinite suffix of $s^{\mathbb{N}}$. Since $s$ is not a prefix of $q$, and as such, not a prefix of $u$; $u$ must be a strict prefix of $s$. Thus, $s\equiv u0w_2$ for some possibly empty suffix $w_2$. Hence $(*)$ holds for $q$ and $s$. 
 
 Next, we modify $q$ and $s$ to ensure that $s$ does not visit any edge of $\mathcal L(H)$ more than twice, while preserving condition $(*)$ for $q$ and  $s$. 
 If $s$ visits an edge $e$ more than twice, then since $s\equiv u0w_2$, either the sub-trail $u$ or the sub-trail $w_2$ visits the edge $e$ twice. In the first case, we cut out a cycle from $e$ to itself from the trail $u$ and note that the result is a shorter trail $u'$ which visits every edge visited by $u$ at most as many times and visits the edge $e$ less times than the trail $u$. We replace $u$ by $u'$ in both trails $q\equiv u1w_1$ and  $s\equiv u0w_2$ and note that $(*)$ still holds for $q$ and $s$.
In the second case,  we replace $w_2$ by a shorter trail by cutting out a cycle from $w_2$. Again, condition $(*)$ for $q$ and $s$ still holds. Repeating the process as long as $s$ visits some edge of $\mathcal L(H)$ more than twice, we get that $q$ and $s$  satisfy $(*)$ and  that $s$ does not visit any edge more than twice.  
As such, by the definition of optimal trail, $s_{e_1}$ is smaller or equal to $s$ in the lexicographic order on $\{0,1\}^+$. Note that if $s_{e_1} \equiv s$ we are done, as $(*)$ holds for $q$ and $s$. Note also that $s_{e_1}$ cannot be a proper prefix of $s$. Indeed, as $s$ terminates on the edge $e_1$ and $s_{e_1}$ visits $e_1$ twice, if $s_{e_1}$ is a proper prefix of $s$, then $s$ visits $e_1$ at least $3$ times, in contradiction to the above. Hence, we can assume that $s_{e_1}$ is not a prefix of $s\equiv u0w_2$. Since $s_{e_1}$ is smaller than $s$, we have $s_{e_1}\equiv a0b_1$ and $s\equiv a1b_2$ where $a$ is the longest common prefix of $s_{e_1}$ and $s$, and $b_1$ and $b_2$ are finite binary words. We note that $a\not\equiv u$, since $u1$ is not a prefix of $s$. Hence, either $a1$ is a prefix of $u$ or $u0$ is a prefix of $a$. In the first case, $a1$ is a prefix of $q$. Hence $q\equiv a1 w_3$ for some finite binary word $w_3$ and $(*)$ holds. In the second case, $a\equiv u0c$ for some finite binary word $c$. Hence, $s_{e_1}\equiv u0c0b_1$ and $q\equiv u1w_1$ and $(*)$ holds. 
\end{proof}

\begin{Definition}\label{def:PH}
Let $H$ be a subgroup of $F$. We define a directed graph $\mathcal P(H)$ as follows. The vertex set of $\mathcal P(H)$ is the set of periodic edges of $\La(H)$. If $e_1$ and $e_2$ are periodic edges, then there is a directed edge from $e_1$ to $e_2$ in $\mathcal P(H)$ if and only if the equivalent conditions in Proposition \ref{prop} are satisfied for $e_1$ and $e_2$. The \emph{length} of a directed path in $\mathcal P(H)$ is the number of vertices in the path.  
\end{Definition}


\begin{Lemma}\label{11}
Let $H$ be a subgroup of $F$ such that $\Cl(H)$ does not admit transition chains. Then the following assertions hold. 
\begin{enumerate}
	\item[(1)] If there is a directed path of length $n$ in $\mathcal P(H)$ then $\Cl(H)$ admits a tower $T=\{(a_i,b_i),i=1,\dots,n\}$ of height $n$. 
	\item[(2)] If there is a semi-periodic edge $e$ in $\mathcal L(H)$ which is not periodic, then $\Cl(H)$ admits a tower of infinite height. 
\end{enumerate}
\end{Lemma}

\begin{proof}
(1) Let $e_1\rightarrow e_2\rightarrow\dots\rightarrow e_n$ be a directed path in $\mathcal P(H)$. Then by Condition (3) in Proposition \ref{prop}, for each $i=1,\dots,n-1$ there is a trail $q_i$ on $\La(H)$ from $e_i$ to $e_{i+1}$ such that $.q_i>.s_{e_i}^{\mathbb{N}}$. 
Let $u$ be a path on $\La(H)$ with terminal edge $e_1$. For each $i\in\{1,\dots,n\}$, we let $v_i\equiv uq_1\cdots q_{i-1}$ and note that $v_i$ is a path on $\La(H)$ with terminal edge $e_i$. 

Since the edges $e_i$ are all periodic edges, for each $i$, $v_i$ and $v_is_{e_i}$ are paths on $\La(H)$ with terminal edge $e_i$. By Lemma \ref{cor_ide}, there is a function $f_i\in\Cl(H)$ with a pair of branches $v_i\rightarrow v_is_{e_i}$. In particular, $f_i$ fixes the rational  fraction $.v_is_{e_i}^{\mathbb{N}}$. Since $f_i$ is linear on $[v_i]$, it does not fix any number in $[v_i]$ other than $.v_is_{e_i}^{\mathbb{N}}$, and in particular, it does not fix any number in $(.v_is_{e_i}^{\mathbb{N}},.v_i1^{\mathbb{N}}]$ (we note that this interval is not empty as $s_{e_i}$ contains the digit $0$). Thus, $f_i$ has an orbital $(a_i,b_i)$ where $a_i=.v_is_{e_i}^{\mathbb{N}}$ and $b_i>.v_i1^{\mathbb{N}}$.

Since for each $i=1,\dots,n-1$, $v_{i+1}\equiv v_iq_i$ and $.q_i>.s_{e_i}^{\mathbb{N}}$, we have 
$$a_{i+1}=.v_{i+1}s_{e_{i+1}}^{\mathbb{N}}\ge.v_{i+1}=.v_iq_i>.v_is_{e_i}^{\mathbb{N}}=a_i.$$
Thus, $a_1<a_2<\dots<a_n$.

Since for each $i=1,\dots,n$, $b_i>.v_i1^{\mathbb{N}}$, for each $i=1,\dots,n-1$, we have $a_{i+1}=.v_{i+1}s_{e_{i+1}}^{\mathbb{N}}=.v_iq_is_{e_{i+1}}^{\mathbb{N}}<.v_i1^{\mathbb{N}}<b_i$. 
It follows that for each $i=1,\dots,n-1$, $a_{i+1}\in(a_i,b_i)$. Thus, since $\Cl(H)$ does not admit transition chains, we must have $b_1\ge b_2\ge \dots\ge b_n$. Hence, $(a_1,b_1)\supset(a_2,b_2)\supset\cdots\supset(a_n,b_n)$ and so $\{(a_i,b_i)\mid i=1,\dots,n\}$ is a tower  of height $n$ in $\Cl(H)$. 

(2) The proof is similar to that of (1). Since $e$ is not periodic, there is an infinite sequence of trails $s_i$, $i\in\mathbb{N}$ from $e$ to itself, such that for each $i$, $s_{i+1}$ is smaller than $s_i$ in the lexicographic order. Let $u$ be a path on $\mathcal L(H)$ such that $u^+=e$. Then for each $i$, $us_i$ is a path on $\mathcal L(H)$ with terminal edge $e$. As in the proof of part (1), we get for each $i$ an element $h_i\in\Cl(H)$ which has an orbital with right endpoint $\alpha_i=.us_i^{\mathbb{N}}$ and which contains $[.u_i,\alpha_i)$. since for each $i$, $\alpha_{i+1}\in [.u_i,\alpha_i)$, the assumption that $\Cl(H)$ does not admit a transition chain implies that the orbitals of $h_i$, $i\in\mathbb{N}$ with right endpoints $\alpha_i$ form an infinite tower. 
\end{proof}

Let $G\le\PL_o(I)$. A tower $T=\{(a_i,b_i)\mid i\in\mathcal I\}$ in $G$ is \emph{good} if for distinct $i$ and $j$, 
either $a_i<a_j<b_j<b_i$ or $a_j<a_i<b_i<b_j$. A good tower is a weaker version of an \emph{exemplary tower} defined in \cite{Bl1}. 

\begin{Lemma}\label{22}
Let $H$ be a subgroup of $F$ which admits a good tower of height $n$. Assume that every semi-periodic edge in $\mathcal L(H)$ is periodic. Then there is a directed path of length $n$ in the graph $\mathcal P(H)$.
\end{Lemma}

\begin{proof}
Let $T=\{(a_i,b_i)\mid i=1,\dots,n\}$ be a good tower in $H$ and assume that the chain of intervals $(a_i,b_i),i=1,\dots,n$ is decreasing. In particular, $a_1<a_2<\dots<a_n$ and $b_n<b_{n-1}<\dots<b_1$. For each $i$, let $h_i\in H$ be an element with orbital $(a_i,b_i)$. We assume that for all $i$, $(a_i,b_i)$ is a push-down orbital of $h_i$, otherwise, one can replace $h_i$ by $h_i^{-1}$. We use the elements $h_i$ to construct a directed path of length $n$ in $\mathcal P(H)$.

For $i=1$, by Lemma \ref{branch}, the reduced diagram $\Delta_1$ of $h_1$ has a pair of branches of the form $u_1\rightarrow u_1s_1$, where $s_1$ contains the digit $0$ and $a_1=.u_1s_1^{\mathbb{N}}$. 
Since $\Delta_1$ is reduced, the words $u_1$ and $u_1s_1$ label paths on $\La(H)$ which terminate on the same edge. Let $e_1=u_1^+=(u_1s_1)^+$ 
 and note that $e_1$ is a semi-periodic edge of $\La(H)$ and as such, a periodic edge.

Since $h_1$ is linear on $[u_1]$ it does not fix any number in $[u_1]$ apart from $a_1$. Therefore, the interval $(a_1,.u_11^{\mathbb{N}}]$
 (which is not empty as $s_1$ contains the digit $0$), is contained in the orbital $(a_1,b_1)$. 
Let $y_1\in (a_1,.u_11^{\mathbb{N}})$. Since $b_2,y_1\in(a_1,b_1)$ and $(a_1,b_1)$ is an orbital of $h_1$, for some $k_1\in\mathbb{Z}$ we have 
$h_1^{k_1}(b_2)<y_1$. In particular $h_2^{h_1^{k_1}}$ has orbital $(h_1^{k_1}(a_2),h_1^{k_1}(b_2))$ contained in $(a_1,y_1)$. 

We replace all functions $h_i$, $i>1$ in the sequence $(h_i)$, $i=1,\dots,n$, by the conjugates $h_i^{h_1^{k_1}}$ and each orbital $(a_i,b_i)$, $i>1$ with the orbital $(h_1^{k_1}(a_i),h_1^{k_1}(b_i))$. We denote the resulting sequence of functions again by $(h_i)$, $i=1,\dots,n$. Similarly, we will refer to the orbitals in the new sequence of orbitals again by $(a_i,b_i)$, $i=1,\dots,n$. Note that this sequence is decreasing and forms a good tower in $H$.

Now, for $i=2$, by Lemma \ref{branch}, the reduced diagram $\Delta_2$ of $h_2$ has a pair of branches of the form $u_2\rightarrow u_2s_2$, where $s_2$ contains the digit $0$ and $a_2=.u_2s_2^{\mathbb{N}}$. Since $\Delta_2$ is reduced, the words $u_2$ and $u_2s_2$ label paths on $\La(H)$ which terminate on the same edge. Let $e_2=u_2^+=(u_2s_2)^+$  and note that $e_2$ is a periodic edge of $\La(H)$. 
Since $f_2$ is linear on $[u_2]$ and fixes $a_2$, the (non-empty) interval $(a_2,.u_21^{\mathbb{N}}]$ is contained in $(a_2,b_2)\subseteq (a_1,y_1)$ and thus, it is contained in the interior of $[u_1]$. That implies that $[u_2]\subset [u_1]$. (Indeed, the right endpoint of the dyadic interval $[u_2]$ is contained in the interior of $[u_1]$.) 
Since $a_1<a_2$ and $a_2=.u_2s_2^{\mathbb{N}}$, for a large enough $m\in\mathbb{N}$, every number in the interval $[u_2s_2^m]$  is greater than $a_1$. Since $[u_2s_2^m]\subseteq[u_2]\subseteq[u_1]$ we have $u_2s_2^{m}\equiv u_1q_1$ for some suffix $q_1$. Since $u_1$ labels a path on $\La(H)$ with terminal edge $e_1$, and the path $u_2s_2^m$ on $\La(H)$ terminates on $e_2$, the word $q_1$ labels a trail from $e_1$ to $e_2$. In addition, $.u_2s_2^{m}>a_1$. Hence, $.u_1q_1>.u_1s_1^{\mathbb{N}}$ which implies that $.q_1>.s_1^{\mathbb{N}}$. Thus, the trail $q_1$ from $e_1$ to $e_2$ and the trail $s_1$ from $e_1$ to itself both satisfy condition (1) from Proposition \ref{prop}. 
It follows that there is a directed edge in $\mathcal P(H)$ from $e_1$ to $e_2$. 

Continuing in this manner, we get periodic edges $e_1,e_2,\dots,e_n$ in $\La(H)$ such that for each $i=1,\dots,n-1$, there is a directed edge in $\mathcal P(H)$ from $e_i$ to $e_{i+1}$. That completes the proof of the lemma. 
\end{proof}

\begin{Theorem}\label{thm}
Let $H\le F$. Then $H$ is solvable if and only if every semi-periodic edge of $\mathcal L(H)$ is periodic and the graph $\mathcal P(H)$ does not contain arbitrarily long positive paths.
If $H$ is solvable then the derived length of $H$ is the maximal length of a positive path in $\mathcal P(H)$. 
\end{Theorem}

\begin{proof}
Assume that $H$ is solvable. By Corollary \ref{cor:sol_cor}, $\Cl(H)$ is also solvable. Hence, by Lemmas \ref{Bl1} and \ref{thm:sol}, $\Cl(H)$ does not admit transition chains and there is a uniform bound on the height of towers in $\Cl(H)$. Therefore, Lemma \ref{11} implies that $\mathcal P(H)$ does not contain arbitrarily long directed paths and that there is no semi-periodic edge in $\mathcal L(H)$ which is not periodic.
In the other direction, assume that $H$ is not solvable. 
It follows from Lemmas 2.11, 2.12 and Theorem 1.1 of \cite{Bl1} that $H$ must admit good towers of arbitrarily large height. Hence, if  every semi-periodic edge of $\mathcal L(H)$ is periodic, Lemma \ref{22} implies that  there are arbitrarily long directed paths in $\mathcal P(H)$.
%
%
%
\end{proof}



\begin{Corollary}\label{cor_PH}
Let $H$ be a subgroup of $F$ with finite core $\La(H)$. Then $H$ is solvable if and only if there is no directed loop in $\mathcal P(H)$.
\end{Corollary}

\begin{proof}
	Since the core $\mathcal L(H)$ is finite, any semi-periodic edge of $\mathcal L(H)$ must be periodic. In addition, $\mathcal P(H)$ is a finite directed graph. Hence, it contains arbitrarily long directed paths if and only if the graph contains a  directed cycle.
	We claim that  there is a directed cycle in $\mathcal P(H)$ if and only if there is a directed loop from some vertex of $\mathcal P(H)$ to itself. 
	Indeed, one direction is obvious. For the other direction it suffices to show that that if $e_1,e_2,e_3$ are vertices of $\mathcal P(H)$ such that there is a directed edge in $\mathcal P(H)$ from $e_1$ to $e_2$ and a directed edge in $\mathcal P(H)$ from $e_2$ to $e_3$, then there is a directed edge in $\mathcal P(H)$ from $e_1$ to $e_3$.
	Assume that there are directed edges in $\mathcal P(H)$ from $e_1$ to $e_2$ and from $e_2$ to $e_3$. Then, by definition, there is a trail $q$ from $e_1$ to $e_2$ on $\mathcal L(H)$, such that $.q>.s_{e_1}^{\mathbb{N}}$ and there is a trail $q’$ on $\mathcal L(H)$ from $e_2$ to $e_3$. Consider the concatenation of trails $qq’$ in $\mathcal L(H)$. The trail $qq’$ is a trail from $e_1$ to $e_3$. It satisfies that $.qq’\geq.q>s_{e_1}^{\mathbb{N}}$. Hence, by the definition of $\mathcal P(H)$, there is a directed edge in $\mathcal P(H)$ from $e_1$ to $e_3$. It follows that if there is a directed cycle in $\mathcal P(H)$, then there is a loop in $\mathcal P(H)$ from each vertex in the cycle to itself.
\end{proof}

We remark that if $H\le F$ is finitely generated, then Corollary \ref{cor_PH} and Theorem \ref{thm} give a simple algorithm for deciding if $H$ is solvable and for finding its derived length if it is. Indeed, one has to construct the directed graph $\mathcal P(H)$, and check whether or not it contains a cycle, and if not, find the length of the longest directed path.

Note that if $H$ is finitely generated, there are finitely many edges in the core $\La(H)$. For each edge $e$ in $\La(H)$, there are finitely many trails with initial edge $e$ which do not visit any edge more than twice. Thus, one can check if $e$ is a periodic edge and if it is, find the optimal trail $s_e$. 
If $e$ is periodic then to find all outgoing edges of $e$ in $\mathcal P(H)$, it suffices to consider all trails of the form $q\equiv u1w$ with initial edge $e$ where $u0$ is a prefix of $s_{e}$ and the sub-trail of $q$ given by the suffix $w$ is simple. Clearly, there are finitely many trails to consider. 


\begin{Remark}\label{n3}
Let $H$ be a subgroup of $F$ such that $\La(H)$ has $n$ edges. It is possible to construct the  directed graph $\mathcal P(H)$ in $O(n^3)$ time and thus, to implement the algorithm for deciding solvability (and derived length) of $H$ in $O(n^3)$ time.  
\end{Remark}

\begin{proof}
We only give an outline of the algorithm and leave the details to the reader. We explain below how to decide for each edge $e$ of $\La(H)$ if it is periodic or not and if it is, find the optimal trail $s_e$ in $O(n^2)$ time. Thus, we get the vertex set of $\mathcal P(H)$ in $O(n^3)$ time. For each vertex $e$, using condition (2) in Proposition \ref{prop}, one can find all outgoing edges of $e$ in $\mathcal P(H)$ in $O(n^2)$ time. Indeed, the length of the optimal trail $s_e$ is at most $2n$, thus, in the notation of Proposition \ref{prop}(2), there are at most $2n$ prefixes $u$ to consider. For each one, finding all  end vertices  of simple trails $w_1$ on $\La(H)$ with initial edge $(u1)^+$ can be done in $O(n)$ time. 
Thus, the construction of the directed graph $\mathcal P(H)$ takes $O(n^3)$ time. Since $\mathcal P(H)$ has at most $n$ vertices and $O(n^2)$ edges, one can decide if it contains a loop and if not find the length of the longest directed path in it in $O(n^2)$ time. 

For each edge $e$ in $\La(H)$ it is possible to determine in $O(n)$ time whether it is \emph{quasi-periodic}; i.e., whether there is a trail on $\La(H)$ from $e$ to itself. Indeed, one can use a Depth First Search to go over all simple trails on $\La(H)$ with initial edge $e$ and check if they terminate on $e$. For a quasi periodic edge $e$ we define the \emph{optimal trail} $s_e$ from $e$ to itself as in Definition \ref{opt_tra} (here, the optimal trail $s_e$ might not contain the digit $0$). 

Let $e$ be a quasi-periodic edge. We find the optimal trail $s_e$ in $O(n^2)$ steps using a greedy algorithm. Namely, one can check in $O(n)$ time (using a DFS) if the trail labeled $0$ starting from $e$ can be extended to a trail $q\equiv 0w$ such that (1) $q$ terminates on the edge $e$ and (2) the suffix $w$ labels a simple trail on $\La(H)$ . If so, then the optimal trail from $e$ to itself starts with $0$. Otherwise, it starts with $1$. 

Assume that we found that the optimal trail from $e$ to itself has prefix $u$. If the trail $u$ starting from $e$ terminates on $e$, then $s_e\equiv u$. Otherwise, we check in $O(n)$ time (using a DFS) whether the trail labeled $u0$ starting from $e$ can be extended to a trail $q\equiv u0w$ such that (1) $q$ terminates on $e$, (2) the suffix $w$ labels a simple trail on $\La(H)$ and (3) the sub-trail of $q$ given by the suffix $w$ does not visit any edge visited twice by the sub-trail $u$. (We assume here that it takes $O(1)$ time to check if a given edge of $\La(H)$ was visited twice by the sub-trail $u$.) If a trail $q$ as described exists, then $u0$ is a prefix of the optimal trail $s_e$. Otherwise, $u1$ is a prefix of the optimal trail $s_e$. 

Since the optimal trail $s_e$ is of length $2n$ at most, we find the optimal trail $s_e$ in $O(n^2)$ time. The edge $e$ is periodic if and only if $s_e$ contains the digit $0$. Thus, determining if $e$ is periodic and finding the optimal trail $s_e$ if it is, can be done in $O(n^2)$ time. The rest of the algorithm is as described above. 
\end{proof}

We remark that Bleak, Brough and Hermiller \cite{BBH} have an algorithm for deciding the solvability of  finitely generated computable subgroups of $\PL_o(I)$ and in particular, of any finitely generated subgroup of Thompson's group $F$. 
 The algorithm given by Corollary \ref{cor_PH} applies the special properties of Thompson's group $F$ and cannot be generalized to subgroups of $\PL_o(I)$. However, for finitely generated subgroups of $F$, it is much easier to implement than the algorithm in \cite{BBH}. Indeed, the algorithm in \cite{BBH} involves consideration of orbitals of elements in $H$ and is composed of $4$ steps, separated altogether into $29$ sub-steps, most of which have to be iterated. While it is proved that the algorithm in  \cite{BBH} will always terminate, no explicit upper bound on the complexity can be derived from the proof. By Remark \ref{n3}, the complexity of our algorithm is polynomial. 
 

We finish this section with a demonstration of our algorithm for determining if a finitely generated subgroup of $F$ is solvable.

\begin{Example}\label{ex:last}
The subgroup $H=\la x_0,x_1^2x_2^{-1}x_1^{-1}\ra$ of Thompson's group $F$ is solvable of derived length $2$.
\end{Example}

\begin{proof}
The core $\La(H)$ is given by the following minimal associated tree. 

\Tree[.$e$ [.$f$ [.$f$ ] [.$h$ ] ] [.$g$ [.$h$ [.$a$ [.$a$ ] [.$c$ ] ] [.$b$ [.$c$ ] [.$b$ ] ] ] [.$g$ ] ] ]

By considering trails with initial edges $e,f,g,h,a,b$ and $c$ which do not visit any edge of $\La(H)$ more than twice, we get that the periodic edges of $\La(H)$ are $f$ with optimal trail $s_f\equiv 0$ and $a$ with optimal trail $s_a\equiv 0$. Thus, the vertices of $\mathcal P(H)$ are $f$ and $a$. 

To find the outgoing edges of $f$ in $\mathcal P(H)$, we consider all trails of the form $q\equiv u1w$ with initial edge $f$, such that $u0$ is a prefix of $s_f$ and such that the sub-trail of $q$ given by the suffix $w$ is a simple trail. Since $s_f\equiv 0$, we must have $u\equiv \emptyset$. Thus, one has to consider trails $q\equiv 1w$ where the suffix $w$ gives a simple sub-trail. One can check that starting from $f$, one can get via a trail $q\equiv 1w$ of this form to the vertices $h,a,b$ and $c$. Hence, there is a directed edge in $\mathcal P(H)$ from $f$ to $a$. 

Since $s_a\equiv 0$, to find the outgoing edges of $a$ in $\mathcal P(H)$, we  again have to consider trails of the form $q'\equiv 1w'$ with initial edge $a$, where the suffix $w'$ gives a simple sub-trail of $q'$. The only such trail (where $w'\equiv\emptyset$) has terminal edge $c$. Hence $a$ has no outgoing edges in $\mathcal P(H)$. Thus, the graph $\mathcal P(H)$ is the following. 

\begin{figure}[h]
	\centering
	\includegraphics[width=0.3\columnwidth]{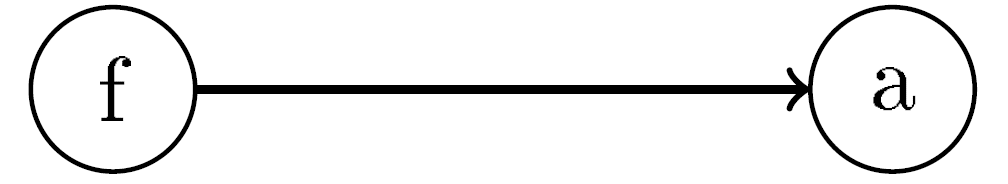}
	\caption{The graph $\mathcal P(H)$}
	\label{pic1}
\end{figure}

It follows from Theorem \ref{thm} that $H$ is solvable of derived length $2$. 
\end{proof}

\begin{Remark}
The group $H$ from Example \ref{ex:last} is a copy of the restricted wreath product $\mathbb{Z} \wr\mathbb{Z}$ considered in \cite{C}. The group $H$ is isomorphic to $\mathbb{Z}\wr\mathbb{Z}$ because the support of $x_1^2x_2^{-1}x_1^{-1}$ is contained in a fundamental domain of $x_0$.
\end{Remark}

\section{Open problems}\label{open}

\subsection{Subgroups of $F$ which contain the derived subgroup of $F$}

Let $H\le F$. By Theorem \ref{main1}, if (1) $[F,F]\le \Cl(H)$ and (2) there exists an element $h\in H$ which fixes a finite dyadic fraction $\alpha\in (0,1)$ such that $h'(\alpha^+)=2$ and $h'(\alpha^-)=1$, then $[F,F]\le H$. However, we do not have examples of subgroups $H$ of $F$ which satisfy the first condition but not the second. We note that the maximal subgroups $H$ constructed in Theorem \ref{B}, Examples \ref{ex:1} and \ref{ex:2} and Proposition \ref{lem:max} were all constructed so that whenever an element $f\notin H$ was added to them they satisfied condition (1) of Theorem \ref{main1}. For any $f\notin H$, the groups $\la H,f\ra$ simply ``happened'' to satisfy condition (2) as well. 

\begin{Problem}\label{pro_der}
Let $H\le F$. Is it true that $H$ contains the derived subgroup $[F,F]$ if and only if $\Cl(H)$ contains $[F,F]$?
\end{Problem}

We note that if the answer to problem \ref{pro_der} is positive, then in particular, we have a positive answer to the following problem. 

\begin{Problem}\cite[Problem 5.12]{GS1}\label{pro_HF}
Let $H$ be a subgroup of $F$ which is not contained in any proper finite index subgroup of $F$. Is it true that $H=F$ if and only if $\Cl(H)=F$?
\end{Problem}

We note that if the answer to Problem \ref{pro_HF} is positive, then it gives a very simple algorithm for solving the generation problem in $F$. 

\subsection{Maximal subgroups of $F$ of infinite index}

\subsubsection{Closed maximal subgroups}

All known maximal subgroups of $F$ of infinite index are closed. Indeed, the Stabilizers $H_{\{\alpha\}}$ in $F$ for $\alpha\in (0,1)$ are closed for components, and as such, by Corollary \ref{cor:clo}, are closed subgroups. 
The maximal subgroup of $F$ of infinite index constructed in \cite{GS} is also closed. In addition, the method for constructing maximal subgroups of $F$ of infinite index demonstrated in Section \ref{ss:max} can only yield closed maximal subgroups. 

\begin{Problem}\cite[Problem 5.11]{GS1}\label{pro_clo}
Is it true that every maximal subgroup of $F$ of infinite index is closed?
\end{Problem}

Note that Problem \ref{pro_clo} is equivalent to Problem \ref{pro_HF}. Indeed, it was observed in \cite{GS1} that a positive answer to Problem \ref{pro_HF} implies a positive answer to Problem \ref{pro_clo}. In the other direction, if $F$ has a proper subgroup $H$ which is not contained in any finite index subgroup of $F$ and such that $\Cl(H)=F$, then a maximal subgroup $M$ of $F$ containing $H$ would be a counter example to Problem \ref{pro_clo}.

\subsubsection{The action of maximal subgroups of infinite index in $F$ on $[0,1]$}

In Proposition \ref{lem:max} we constructed a maximal subgroup of infinite index in $F$ which acts transitively on the set of finite dyadic fractions $\mathcal D$. The following problem remains open.

\begin{Problem}\label{pro_tra}
Does $F$ have a maximal subgroup $H$ of infinite index such that the orbits of the action of $H$ on the interval $[0,1]$ coincide with the orbits of the action of $F$?
\end{Problem}

We note that if the answer to Problem \ref{pro_clo} is negative, then the answer to Problem \ref{pro_tra} is positive. Indeed, if $H$ is a non-closed maximal subgroup of infinite index in $F$, then $\Cl(H)=F$. Then by Corollary \ref{orb_CH}, the actions of $H$ and of $F$ on the interval $[0,1]$ have the same orbits.

If $\alpha\in (0,1)$ is a rational non-dyadic number, then $\alpha=.ps^{\mathbb{N}}$ where $s$ is a minimal period. Clearly, $s$ contains both digits $0$ and $1$. The orbit of $\alpha$ under the action of $F$ is the set of all rational numbers in $(0,1)$ with minimal period $s$. 
To decide if the action of $H$ on this orbit is transitive, one can adapt the algorithm from Theorem \ref{thm:tra}. 
 Indeed, we define a graph $\Gamma_s(H)$ whose vertex set is the vertex set of $\Gamma(H)$ (see Definition \ref{Gamma}). There is a directed edge from $e_1$ to $e_2$ in $\Gamma_s(H)$ if and only if there is a trail labeled $s$ with initial edge $e_1$ and terminal edge $e_2$ in $\La(H)$. Then Theorem \ref{thm:tra}, with the graph $\Gamma(H)$ replaced by $\Gamma_s(H)$ and the set $\mathcal D$ replaced by the orbit of $\alpha$ under the action of $F$ on $[0,1]$, holds. 


A consideration of the graph $\Gamma_s(H)$, where $s\equiv 01$ and $H$ is the maximal subgroup of $F$ from Proposition \ref{lem:max}, 
shows that $H$ does not act transitively on the orbit of $\frac{1}{3}=.(01)^{\mathbb{N}}$ under the action of $F$. Thus, $H$ acts transitively on the set $\mathcal D$ but is not a solution for Problem \ref{pro_tra}. 

\subsubsection{$2$-generated maximal subgroups of $F$}


In \cite{GS2}, the author and Sapir prove that the isomorphism class of the stabilizer $H_{\{\alpha\}}$ of $\alpha\in (0,1)$ depends only on the \emph{type} of $\alpha$; i.e., on whether $\alpha$ is dyadic, rational non-dyadic or irrational. 
It is easy to see that if $\alpha$ is dyadic, then $H_{\{\alpha\}}$ is isomorphic to the direct product of two copies of $F$. Therefore, in that case, the minimal size of a generating set of $H_{\{\alpha\}}$ is $4$. Savchuk observed that if $\alpha$ is irrational then $H_{\{\alpha\}}$ is not finitely generated. In \cite{GS2} we prove that if $\alpha$ is finite dyadic then the stabilizer $H_{\{\alpha\}}$ has a generating set with $3$ elements and does not have any smaller generating set. 
The explicit maximal subgroup of $F$ constructed in \cite{GS1} is isomorphic to $F_3$; the ``brother group'' of $F$ which consists of all piecewise linear homeomorphisms of the unit interval $[0,1]$ where all breakpoints are triadic fractions and all slopes are integer powers of $3$. Thus, it has a generating set of size $3$ and no less \cite{Br}.

We note that all maximal subgroups of $F$ constructed in this paper are $3$- or $4$- generated. Since the generating sets provided are not necessarily optimal, it is possible that one of them is $2$-generated, but we do not know if that is the case. Note that the algorithm from \cite[Lemma 9.11]{GuSa97} enables computing a presentation for the maximal subgroups constructed (as they are all closed). By moving to the abelianization, one can find a lower bound for the minimal size of a generating set. An application of this algorithm for the maximal subgroup from Example \ref{ex:1}, shows that it cannot be generated by less than $3$ elements. We did not apply the algorithm for the other maximal subgroups constructed in the paper. In general, the following problem is open.

\begin{Problem}\label{pro_2gen}
Does Thompson's group $F$ have a $2$-generated maximal subgroup of infinite index?
\end{Problem}

We note that by \cite{BW}, for each prime $p$, Thompson's group $F$ has subgroups of index $p$ which are isomorphic to $F$. Therefore, $F$ has $2$-generated maximal subgroups of finite index. In fact, we go as far as to make the following conjecture.

\begin{Conjecture}\label{con_2gen}
All finite index subgroups of $F$ are $2$-generated.
\end{Conjecture}

Note that all finite index subgroups of $\mathbb{Z}^2$ are $2$-generated. Hence, the image of any finite index subgroup of $F$ in the abelianization of $F$ is $2$-generated. 
The proof in Remark \ref{rem_K} of the group $K$ being $2$-generated, can be adapted to certain finite index subgroups of $F$ considered by the author. We believe that the proof can be adapted to all finite index subgroups of $F$ (note that Bleak and Wassink \cite{BW} found all isomorphism classes of finite index subgroups of $F$). It is interesting to note that if Conjecture \ref{con_2gen} holds and the answer to Problem \ref{pro_2gen} is ``no'' then the minimal number of generators of a maximal subgroup of $F$ is an invariant determining whether the subgroup has finite or infinite index in $F$. 

\subsection{Subgroups $H\le F$ with finite core $\La(H)$}

\begin{Problem}\label{pro_fin}
Let $H$ be a subgroup of $F$ such that the core $\La(H)$ is finite. Is it true that $\Cl(H)$ is finitely generated?
\end{Problem}

We note that by Corollary \ref{cor_finite}, if $\La(H)$ is finite, then $\Cl(H)=\Cl(K)$, where $K$ is finitely generated. Thus, Problem \ref{pro_fin} is equivalent to \cite[Problem 5.7]{GS1}, asking whether the closure of a finitely generated subgroup of $F$ is finitely generated. 

We note that by Theorem \ref{sol}, the answer to Problem \ref{pro_fin} is ``yes'' if $H$ is solvable. If $H$ is elementary amenable, we already do not have an  answer to Problem \ref{pro_fin}. 

If the core $\La(H)$ is finite then the semigroup presentation $\mathcal P$ associated with $\La(H)$ is finite. It follows from Corollary \ref{set_B}, that if $\mathcal P$ has a finite completion $\mathcal P'$ (satisfying the conditions in Section \ref{ss:alg}), then $\Cl(H)$ is finitely generated. 

\begin{Problem}\label{pro:pre}
Let $H$ be a subgroup of $F$ with finite core $\La(H)$. Let $\mathcal P$ be the semigroup presentation associated with $\La(H)$. Is it true that $\mathcal P$ has a finite completion $\mathcal P'$ (satisfying the conditions from Section 
\ref{ss:alg})?
\end{Problem}

We find it unlikely that the answer to Problem \ref{pro:pre} would be ``yes''. But as Propositions \ref{prop:left} and \ref{prop:right} show, semigroup presentations $\mathcal P$ associated with cores $\La(H)$ of subgroups $H\le F$ satisfy some special properties. 

Finally, we note that there are finitely generated closed subgroups of $F$ which are not finitely presented. For example, the subgroup from Example \ref{ex:last}, isomorphic to $\mathbb{Z} \wr \mathbb{Z}$, is a closed subgroup of $F$. Indeed, it is easy to prove that it is closed for components and therefore closed by Corollary \ref{cor:clo}. The wreath product $\mathbb{Z}\wr\mathbb{Z}$ is $2$-generated but is not finitely presented.

\begin{minipage}{3 in}
Gili Golan Polak\\
Department of Mathematics,\\
Ben Gurion University of the Negev,\\
Be'er Sheva, Israel\\
golangi@bgu.ac.il
\end{minipage}


\begin{thebibliography}{AAA}

\bibitem{BM} J. Belk and F. Matucci,
\it Conjugacy and Dynamics in Thompson's Groups
\rm Geometriae Dedicata, 169 (2014), no. 1,  239-261

\bibitem{BHMM} J. Belk, N. Hossain, F. Matucci and R. McGrail
\it Implementation of a Solution to the Conjugacy Problem in Thompson’s Group $F$
\rm ACM Communications in Computer Algebra 47, 3/4 (2014), 120-121


\bibitem{BW} C. Bleak and B. Wassink,
\it Finite index subgroups of R. Thompson’s group F. \rm arXiv:0711.1014 

\bibitem{Bl1} C. Bleak,
\it A geometric classification of some solvable groups of homeomorphisms.
\rm J. Lond. Math. Soc. (2) 78 (2008), no. 2, 352-372.

\bibitem{Bl2} C. Bleak,
\it An algebraic classification of some solvable groups of homeomorphisms
\rm Journal of Algebra, 319 (2008), no. 4, 1368-1397.

\bibitem{Bl3} C. Bleak,
\it A minimal non-solvable group of Homeomorphisms 
\rm Groups Geom. Dyn. 3 (2009), 1-37.

\bibitem{BBH} C. Bleak, T. Brough and S. Hermiller,
\it Determining solubility for finitely generated groups of PL homeomorphisms
\rm arXiv:1507.06908


\bibitem{Brin} M. Brin,
\it Elementary amenable subgroups of R. Thompson's group $F$.
\rm Internat. J. Algebra Comput. 15 (2005), no. 4, 619–642.

\bibitem{Br} K. Brown, 
\it Finiteness properties of groups,
\rm Proceedings of the Northwestern conference on cohomology of groups (Evanston, Ill., 1985). J. Pure Appl. Algebra 44 (1987), no. 1-3, 45-75.


\bibitem{BMV} J.Burillo, F. Matucci and E.Ventura,
\it The Conjugacy Problem in Extensions of Thompson's group $F$,
\rm Israel Journal of Mathematics volume 216, 15–59 (2016).



\bibitem{CFP} J. Cannon, W. Floyd and W. Parry,
\it  Introductory notes on Richard Thompson's groups.
\rm L'Enseignement Mathematique, 42 (1996), 215--256.

\bibitem{C} S. Cleary,
\it Distortion of wreath products in some finitely-presented groups,
\rm (Pacific Journal of Mathematics, Vol. 228, (2006), 1, 53-61.


\bibitem{EM} G. Elek and N Monod,
\it On the topological full group of a minimal Cantor $\mathbb{Z}^2$-system,
\rm Proc. Amer. Math. Soc., 141(10):3549–3552, (2013)




\bibitem{GS} G. Golan and M. Sapir,
\it On Jones' subgroup of R. Thompson group $F$,
\rm Journal of Algebra 470 (2017) 122–159.

\bibitem{GS1} G. Golan and M. Sapir,
\it On subgroups of R. Thompson group $F$,
\rm Trans. Amer. Math. Soc. 369 (2017), 8857--8878

\bibitem{GS2} G. Golan and M. Sapir,
\it  On the stabilizers of finite sets of numbers in the R. Thompson group $F$,
\rm 
Algebra i Analiz, 29 (1) (2017),   70--110

\bibitem{GuSa97} V. Guba and M. Sapir.
\it Diagram groups,
\rm Memoirs of the Amer. Math. Soc. 130, no. 620 (1997), 1--117.
%
%

\bibitem{GuSa99} V Guba and M Sapir,
\it On subgroups of the R. Thompson group F and other diagram groups,
\rm (Russian) Mat. Sb. 190 (1999), no. 8, 3--60; translation in Sb. Math. 190 (1999), no. 7-8, 1077–1130.


\bibitem{GSdc} V.  Guba and M. Sapir,
\it Diagram groups and directed
$2$-complexes: homotopy and homology.
\rm Journal of Pure Appl. Algebra  205
(2006),  no. 1, 1--47.


\bibitem{Jones} Vaughan Jones,
\it Some unitary representations of Thompson's groups $F$ and $T$,
\rm Journal of Combinatorial Algebra 1 (1), (2017), 1--44.


\bibitem{KM}  M.Kassabov and F. Matucci
\it The Simultaneous Conjugacy Problem in Groups of Piecewise Linear Functions, 
\rm Groups, Geometry and Dynamics 6, No. 2 (2012) 279--315



\bibitem{LS}  R. Lyndon and P. Schupp.
\it Combinatorial group theory,
\rm Springer–Verlag, 1977.

\bibitem{N} A. Navas,
\it Quelques Groupes Moyennables de Diff{\'e}omorphismes de L'intervalle
\rm Bol. Soc. Mat. Mexicana 10 (2004), 219-244.



\bibitem{R} E. Rips,
\it Subgroups of small Cancellation Groups,
\rm Bull. London Math. Soc. (1982) 14 (1): 45-47.


\bibitem{Sa} M.~Sapir,
\emph{Combinatorial algebra: syntax and semantics}, Springer Monographs in Mathematics, 2014.

\bibitem{SG_P} M. Sapir,
\it Personal communication.\rm

\bibitem{Sav1}
Dmytro Savchuk,
\it Some graphs related to Thompson's group $F$.
\rm Combinatorial and geometric group theory, 279–-296,
Trends Math., Birkhäuser/Springer Basel AG, Basel, 2010.


\bibitem{Sav} Dmytro Savchuk,
\it Schreier graphs of actions of Thompson's group $F$ on the unit interval and on the Cantor set.
\rm Geom. Dedicata 175 (2015), 355--372.


\bibitem{SU} Vladimir Shpilrain and Alexander Ushakov,
\it Thompson’s group and public key cryptography
\rm Lecture Notes Comp. Sc. 3531 (2005), 151–164.

\bibitem{Sta} John R. Stallings,
\it Foldings of G-trees. Arboreal group theory (Berkeley, CA, 1988), 355–368,
\rm Math. Sci. Res. Inst. Publ., 19, Springer, New York, 1991.

\bibitem{W} D. Wise, 
\it A residually finite version of Rips’s construction,
\rm Bull. London Math. Soc. 35 (2003), 23–29.






\end{thebibliography}
\end{document}